\documentclass[reqno,12pt]{amsart}
\usepackage{amsmath,amsfonts,amssymb,amsthm}
\usepackage{graphics}
\usepackage{epsfig}
\usepackage[usenames,dvipsnames]{color}
\usepackage{tikz}
\usepackage{mathtools}
\usepackage[colorlinks=true,linkcolor=black,menucolor=black,citecolor=black]{hyperref}
\usepackage[margin = .75in]{geometry}
\usepackage{cite}

\renewcommand{\d}{\ensuremath{\mathrm{d}}}
\newcommand{\R}{\ensuremath{\mathbb{R}}}
\newcommand{\Z}{\ensuremath{\mathbb{Z}}}

\newcommand{\N}{\ensuremath{\mathbb{N}}}

\newcommand{\NN}{\ensuremath{\mathcal N}}
\newcommand{\MM}{\ensuremath{\mathcal M}}

\newcommand{\AAA}{\ensuremath{\mathcal A}}
\newcommand{\eps}{\ensuremath{\varepsilon}}
\newcommand{\verti}[1]{\ensuremath{\left\lvert #1 \right\rvert}}
\newcommand{\vertii}[1]{\ensuremath{\left\lVert #1 \right\rVert}}
\newcommand{\vertiii}[1]{{\left\lvert\kern-0.25ex\left\lvert\kern-0.25ex\left\lvert #1
    \right\rvert\kern-0.25ex\right\rvert\kern-0.25ex\right\rvert}}
\renewcommand{\d}{\ensuremath{{\rm d}}}
\newcommand{\spanned}{\ensuremath{{\rm span}}}

\DeclarePairedDelimiter\floor{\lfloor}{\rfloor}

\xdefinecolor{lightblue}{RGB}{160,220,255}

\newtheorem{theorem}{Theorem}[section]
\newtheorem{definition}[theorem]{Definition}
\newtheorem{proposition}[theorem]{Proposition}

\newtheorem{lemma}[theorem]{Lemma}

\numberwithin{equation}{section}
\begin{document}
%
\title[On the regularity for the Navier-slip thin-film equation]{On the regularity for the Navier-slip thin-film equation in the perfect wetting regime}
\thanks{The author is grateful to Lorenzo Giacomelli, Dominik John, Hans Kn\"upfer, and Felix Otto for many fruitful discussions that have served as a motivation to conduct this research. Discussions with Slim Ibrahim, Nader Masmoudi, and Mircea Petrache on related questions are appreciated as well. The author also wishes to thank the anonymous referee for a detailed report on the manuscript that has lead to improvements of the presentation. Funding was obtained by the Fields Institute for Research in Mathematical Sciences in Toronto and the National Science Foundation under Grant No.~NSF DMS-1054115.}
\date{\today}
\author{Manuel~V.~Gnann}
\address{Department of Mathematics, University of Michigan\newline
2074 East Hall, 530 Church Street\newline
Ann Arbor, MI 48109-1043}
\email{mvgnann@umich.edu}
\begin{abstract}
We investigate perturbations of traveling-wave solutions to a thin-film equation with quadratic mobility and a zero contact angle at the triple junction, where the three phases liquid, gas, and solid meet. This equation can be obtained in lubrication approximation from the Navier-Stokes system of a liquid droplet with a Navier-slip condition at the substrate. Existence and uniqueness have been established by the author together with Giacomelli, Kn\"upfer, and Otto in previous work. As solutions are generically non-smooth, the approach relied on suitably subtracting the leading-order singular expansion at the free boundary.

\medskip

In the present note, we substantially improve this result by showing the regularizing effect of the degenerate-parabolic equation to arbitrary orders of the singular expansion. In comparison to related previous work, our method does not require additional compatibility assumptions on the initial data. The result turns out to be natural in view of the properties of the source-type self-similar profile.
\end{abstract}
\maketitle
\tableofcontents
%

\section{Introduction}
\subsection{The thin-film equation as a classical free-boundary problem}
We are interested in the thin-film equation with \emph{quadratic mobility}
\begin{subequations}\label{tfe_free}
\begin{equation}\label{tfe}
\partial_t h + \partial_z\left(h^2 \partial_z^3 h\right) = 0 \quad \mbox{for } \, t > 0 \, \mbox{ and } \, z > Z_0(t).
\end{equation}
This is a \emph{fourth-order} \emph{degenerate-parabolic} equation modeling the evolution of the height $h$ of a two-dimensional thin viscous film on a one-dimensional flat substrate as a function of time $t > 0$ and base point $z > Z_0(t)$ \cite{beimr.2009,d.1985,odb.1997}. For simplicity we assume that the droplet extends infinitely to positive $z$ and has a free boundary $Z_0(t)$ denoting the \emph{contact line}, that is, the \emph{triple junction} between the three phases liquid, gas, and solid (cf.~Fig.~\ref{fig:hodograph}).
\begin{figure}[htp]
\centering
\begin{tikzpicture}[scale=1]
\path[fill=lightblue] (1.5,0) to [out=0,in=195] (9,4) -- (9,0) -- (1.5,0);

\draw [very thick,->] (-1.2,0) -- (10,0);
\draw [very thick,->] (0,-1.2) -- (0,5.2);

\draw[very thick,red,dashed] (1.5,0) to [out=0,in=240] (9,5);
\draw[very thick,blue] (1.5,0) to [out=0,in=195] (9,4);

\draw [gray,dashed] (5.7,-.2) -- (5.7,2);
\draw [gray,dashed] (6.7,-.8) -- (6.7,2);
\draw [gray,dashed] (1.5,0) -- (1.5,-.8);
\draw [gray,dashed] (0,2) -- (6.7,2);

\draw [thick,blue,dashed,<->] (0,-.2) -- (5.7,-.2);
\draw [blue] (2.85,-.2) node[anchor=north] {$Z(t,x)$};

\draw [thick,red,dashed,<->] (1.5,-.8) -- (6.7,-.8);
\draw [red] (4.1,-.8) node[anchor=north] {$x$};

\draw (0,2) node[anchor=east] {$\color{blue}h(t,Z(t,x))\color{black} = \color{red}x^{\frac 32}\color{black}$};

\draw (2,3) node[anchor=east] {gas};

\draw [blue] (8.5,1.5) node[anchor=east] {liquid};

\draw [thick,violet,->] (1.8,.8) -- (1.5,0);
\draw [violet] (1.8,.8) node[anchor=south] {triple junction};

\draw (0,4.9) node[anchor=east] {film height $h$};

\draw (8.7,0) node[anchor=north] {base point $z$ or $x$};

\end{tikzpicture}
\caption{Schematic of a liquid thin film and the hodograph transform \eqref{hodograph}.}
\label{fig:hodograph}
\end{figure}
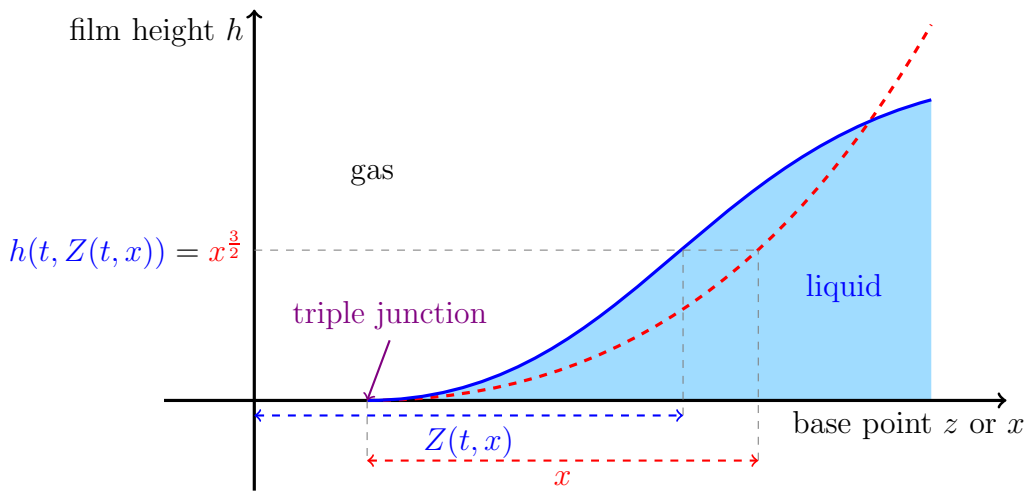
Necessarily
\begin{equation}\label{tfe_bdry_1a}
h = 0 \quad \mbox{for } \, t > 0 \, \mbox{ and } \, z = Z_0(t),
\end{equation}
which simply sets the location of the free boundary to be $Z_0(t)$. Secondly, we assume
\begin{equation}\label{tfe_bdry_1b}
\partial_z h = 0 \quad \mbox{for } \, t > 0 \, \mbox{ and } \, z = Z_0(t),
\end{equation}
leading to a \emph{zero contact angle} at the triple junction, known as \emph{complete} (or \emph{perfect}) \emph{wetting}. The notion can be explained if one considers \emph{quasi-static} droplet motion in which the contact angle is determined by a balance between the surface tensions of the three interfaces at the contact line (\emph{Young's law}). The condition $\partial_z h = 0$ at $z = Z_0(t)$ implies that an equilibrium is generically not achieved and therefore the film will ultimately cover the whole surface. Finally, a condition determining the evolution of $Z_0(t)$ is imposed:
\begin{equation}\label{tfe_bdry_2}
\lim_{z \searrow Z_0(t)} h \partial_z^3 h = \frac{\d Z_0}{\d t}(t) \quad \mbox{for } \, t > 0 \, \mbox{ and } \, z = Z_0(t).\end{equation}
\end{subequations}
This condition may be viewed as a \emph{Rankine-Hugoniot condition} for a viscous shock wave: Since \eqref{tfe} has the form of a (nonlinear) continuity equation
\begin{subequations}\label{continuity_eq}
\begin{equation}
\partial_t h + \partial_z (V h) = 0  \quad \mbox{for } \, t > 0 \, \mbox{ and } \, z > Z_0(t),
\end{equation}
where
\begin{equation}\label{def_vel}
V = h \partial_z^3 h
\end{equation}
\end{subequations}
is the transport velocity of $h$, by compatibility necessarily \eqref{tfe_bdry_2} holds true.

\medskip

Equation~\eqref{tfe} is a particular version of the general class of thin-film equations
\begin{equation}\label{tfe_general}
\partial_t h + \partial_z \left(h^n \partial_z^3 h\right) = 0 \quad \mbox{for } \, t > 0 \, \mbox{ and } \, z \in \{h > 0\},
\end{equation}
where $n$ is a real parameter. Apparently for $n \le 0$ equation~\eqref{tfe_general} has infinite speed of propagation and non-negativity of $h$ is not ensured. On the other hand, for $n = 3$ (corresponding to the \emph{no-slip condition} at the liquid-solid interface) equation~\eqref{tfe_general} exhibits unphysical features as well: The solution is singular at the free boundary, which does not move unless an infinite amount of energy to overcome dissipation is inserted into the system \cite{dd.1974,hs.1971,m.1964}. Since for $n > 3$ equation~\eqref{tfe_general} is even more degenerate, for a model relevant for the motion of fluid films (in which contact lines move with finite and in general nonzero velocity) necessarily $n \in (0,3$). The most important cases are $n = 1$ and $n = 2$, corresponding to the lubrication approximation of \emph{Darcy's flow} in the \emph{Hele-Shaw cell} ($n = 1$)\footnote{For an extensive well-posedness and regularity theory of \eqref{tfe_general} for $n = 1$ in the complete wetting case, we refer to \cite{gk.2010,gko.2008,g.2014,j.2015}. The partial wetting case is addressed in \cite{km.2015,km.2013}.} or the lubrication approximation of the \emph{Navier-Stokes equations} with a (linear) \emph{Navier-slip condition} ($n = 2$)\footnote{We refer to \cite{ggko.2014} for a well-posedness and partial regularity result of \eqref{tfe_general} in the complete wetting regime with $n = 2$ and to \cite{k.2011} for a result in the partial wetting regime.} at the liquid-solid interface, respectively \cite{d.1985,go.2003,km.2015,km.2013,odb.1997}. More detailed discussions of the literature, also addressing the well-established existence theory of weak solutions, can be found in \cite{ag.2004,b.1998,gs.2005}.

\medskip

This note addresses the regularity of solutions at the free boundary and may therefore be considered as a contribution to a regularity theory of higher-order degenerate-parabolic equations, an only insufficiently explored field. The author hopes that the present study is also relevant regarding an existence, uniqueness, and regularity theory for the (Navier-)Stokes equations with a moving contact line, an essentially open problem. Here, a thorough understanding of the singular behavior at the free boundary can potentially help in the construction of suitable function spaces. Additionally, in view of the aforementioned no-slip paradox, the question of how the Navier-slip condition (or even general nonlinear slip conditions) de-singularizes the solution at the contact line $z = Z_0(t)$ is also of interest from the applied view point. Loosely speaking, the following arguments will demonstrate that higher spatial regularity (i.e., regularity in the variable $z$) goes hand in hand with higher regularity in the time variable $t$. In numerical studies in \cite{p.2015}, this observation is used to test numerical schemes regarding their precision in resolving the dynamics of moving contact lines.

\subsection{Transformations}
We review the transformations discussed in detail in \cite[Sec.~1.2, Sec.~1.3, Sec.~A.1]{ggko.2014}: Generic solutions to \eqref{tfe_free} are traveling-wave solutions, that is, solutions of the form $h(t,z) = H_\mathrm{TW}(x)$, where $x = z - V_0 t$ and $V_0 > 0$ is the speed of the traveling wave. By rescaling, we may without loss of generality assume $V_0 = \frac 3 8$ and one may conclude that in the case of a moving infinite cusp a strictly monotone similarity solution is given by\footnote{A discussion of traveling-wave solutions can be found in \cite{bko.1993}.} $H_\mathrm{TW}(x) = x^{\frac 3 2}$. Considering perturbations of this profile, we set
\begin{equation}\label{hodograph}
h(t,Z(t,x)) = x^{\frac 3 2} \quad \mbox{for } \, t, x > 0.
\end{equation}
Under the assumption that $h(t,z)$ is strictly monotone in $z$ for $z > Z_0(t)$, equation~\eqref{hodograph} uniquely determines the function $Z = Z(t,x)$, thus interchanging dependent and independent variables and fixing the boundary to $x = 0$. The transformation \eqref{hodograph} is known as the \emph{hodograph} transform (cf.~Fig.~\ref{fig:hodograph}). The related \emph{von Mises} transform has been applied already in the context of the porous-medium equation and the thin-film equation with linear mobility in higher dimensions \cite{j.2015, k.1999}. Plugging expression~\eqref{hodograph} into \eqref{tfe} and noting that the chain rule (applied to \eqref{hodograph}) gives the transformations
\begin{equation}\label{tfe_trafo_0}
\partial_t h = - Z_t \partial_z h \qquad \mbox{and} \qquad \partial_z = Z_x^{-1} \partial_x,
\end{equation}
equation \eqref{tfe_free} now reads\footnote{Here and in what follows, differential operators act on everything to their right, that is, e.g.
\[
Z_x^{-1} \partial_x x^3 \left(Z_x^{-1} \partial_x\right)^3 x^{\frac 3 2} \equiv Z_x^{-1} \partial_x \left(x^3 \left(Z_x^{-1} \partial_x \left(Z_x^{-1} \partial_x \left(Z_x^{-1} \partial_x\right)\right)\right)\right).
\]
}
\begin{equation}\label{tfe_trafo_1}
- Z_t Z_x^{-1} \partial_x x^{\frac 3 2} + Z_x^{-1} \partial_x x^3 \left(Z_x^{-1} \partial_x\right)^3 x^{\frac 3 2} = 0 \quad \mbox{for } \, t, x > 0.
\end{equation}
Defining
\begin{equation}\label{trafo}
F := Z_x^{-1}
\end{equation}
and noting that $F_t = - F^2 Z_{xt}$, we observe that \eqref{tfe_trafo_1} alters to
\[
x \partial_t F + F^2 x \partial_x x^{- \frac 1 2} \partial_x x^3 F \partial_x F \partial_x F x^{\frac 1 2} = 0 \quad \mbox{for } \, t, x > 0.
\]
By commuting the powers of $x$ with the differential operators $\partial_x$, we then arrive at the equation
\begin{equation}\label{tfe_trafo_2}
x \partial_t F + \MM(F,\cdots,F) = 0 \quad \mbox{for } \, t, x > 0,
\end{equation}
where $\MM$ is a 5-linear form explicitly given by
\begin{equation}\label{5linear}
\MM(F_1,\cdots,F_5) := F_1 F_2 D \left(D + \frac 3 2\right) F_3 \left(D - \frac 1 2\right) F_4 \left(D + \frac 1 2\right) F_5
\end{equation}
and $D := x \partial_x$ denotes the scaling-invariant (logarithmic) derivative in space (setting $s := \ln x$ we have $D = \partial_s$). We observe $Z_\mathrm{TW} \stackrel{\eqref{hodograph}}{=} x - \frac 3 8 t$ and $F_\mathrm{TW} \stackrel{\eqref{trafo}}{\equiv} 1$ for the traveling-wave profile (indeed $\MM(1,\cdots,1) \stackrel{\eqref{5linear}}{=} 0$), so that by setting
\begin{equation}\label{trafo_2}
u := F-1
\end{equation}
in fact $u_\mathrm{TW} = 0$ and we are lead to study the Cauchy problem
\begin{subequations}\label{cauchy}
\begin{align}
x \partial_t u + p(D) u &= \NN(u) \quad \mbox{for } \, t, x > 0, \label{pde}\\
u_{|t = 0} &= u^{(0)} \quad \mbox{for } \, x > 0.
\end{align}
\end{subequations}
Here we use the following notation:
\begin{itemize}
\item[$\bullet$] $p(\zeta)$ is a fourth-order polynomial
\begin{equation}\label{poly}
p(\zeta) := \zeta \left(\zeta^2 + \frac 1 2 \zeta - \frac 3 4\right) \left(\zeta + \frac 3 2\right) = \zeta (\zeta - \beta) \left(\zeta + \beta + \frac 1 2\right) \left(\zeta + \frac 3 2\right)
\end{equation}
with the irrational root $\beta := \frac{\sqrt{13} - 1}{4} \approx 0.65$. The operator $p(D)$ can be derived from the 5-linear form $\MM$ (cf.~\eqref{5linear}) by noting that
\[
p(D) u = \MM(1,\cdots,1,u) + \cdots + \MM(u,1,\cdots,1).
\]
\item[$\bullet$] $\NN(u)$ stands for the nonlinearity of the equation and has the structure
\begin{equation}\label{nonlinearity}
\NN(u) := p(D) u - \MM(1+u,\cdots,1+u).
\end{equation}
\end{itemize}
The structure of \eqref{pde} is such that the left-hand side is linear in $u$, whereas the right-hand side contains terms that are at least quadratic and at most quintic in $\{u, Du, \cdots, D^4 u\}$. The form of \eqref{pde} can be guessed from \eqref{tfe} immediately: By \eqref{hodograph} the linearization of the spatial part $\partial_z \left(h^2 \partial_z^3 h\right)$ of \eqref{tfe} scales as $x^{-1}$ and therefore the spatial operator $x^{-1} p(D)$ appears, where $p(\zeta)$ has to be of order $4$. This implies the space-time scaling $x \sim t$, unlike $x \sim t^{\frac 1 4}$ for non-degenerate fourth-order parabolic equations. Two of the roots of the polynomial $p(\zeta)$ are immediate as well: The root $-\frac 3 2$ is directly related to the exponent of the right-hand side of the hodograph transform \eqref{hodograph}. The root $0$ appears as the nonlinear equation \eqref{tfe} has divergence form and this feature is preserved in a linearization. In other words, one may understand the occurrence of this root by noting that the traveling wave $F_{\mathrm{TW}} = 1$ is simultaneously a solution of the respective linear and nonlinear equations, and the equations for $u$ and $F$ yield the same linear operator. In contrast to these two roots, the emergence of the roots $\beta$ and $-\beta-\frac 1 2$ in the polynomial $p(\zeta)$ in \eqref{poly} is a genuinely non-trivial feature, specific to higher-order degenerate-parabolic equations: In the second-order case, $p(\zeta)$ would be a second-order polynomial and the only two roots of $p(\zeta)$ would be immediate by the same arguments as above. An accessible way to understand this uncommon phenomenon in the fourth-order case for source-type self-similar solutions using the language of dynamical systems theory is explained in \cite{ggo.2013,bgk.2016}: There the roots discussed above turn out to be the eigenvalues of a linearized dynamical system at a stationary point (representing the contact line) and the situation can be reduced due to known regularity results of the corresponding invariant manifolds (Hartman-Grobman theorem).

\medskip

The rest of the note will be mainly concerned with the analysis of \eqref{cauchy}.

\subsection{Notation}
For $f, g \ge 0$ we write $f \lesssim_S g$ if there exists a constant $C > 0$ depending on parameters $S$ such that $f \le C g$. In this case we say that $f$ can be estimated by $g$ or equivalently that $g$ bounds/controls $f$. Furthermore, we write $f \sim_S g$ if $f \lesssim_S g$ and $f \gtrsim_S g$. For a non-negative quantity $r$ we say that a property is true for $r \ll_S 1$ (or $r \gg_S 1$, respectively) if there exists a (sufficiently large) constant $C > 0$ depending on $S$ such that  the property is true for $r \le C^{-1}$ ($r \ge C$). Then we say that $r \ge 0$ has to be sufficiently small (large). If $S = \emptyset$ or if the dependence is specified elsewhere, we just write $f \lesssim g$ etc. The space $C_0^\infty((0,\infty))$ denotes the space of test functions, i.e., the space of all $u: (0,\infty) \to \R$ which are infinitely differentiable with compact support contained in $(0,\infty)$. For $\alpha \in \R$, we denote by $\floor{\alpha} := \max\{k \in \Z: \, k \le \alpha\}$ the integer part (floor) of $\alpha$. We write $\verti{A}$ for the number of elements (cardinality) of a finite set $A$.

\subsection{Weighted $L^2$-norms}
For the subsequent results, we introduce a scale of weighted $L^2$-norms $\verti{\cdot}_\alpha$. These are given by
\begin{equation}\label{l2_norm}
\verti{u}_\alpha^2 := \int_0^\infty x^{-2 \alpha} (u(x))^2 \frac{\d x}{x} = \int_{-\infty}^\infty e^{-2 \alpha s} \left(u(e^s)\right)^2 \d s \quad \mbox{with } \, \alpha \in \R.
\end{equation}
The larger $\alpha$ is, the better the decay of $u(x)$ as $x \searrow 0$ if $\verti{u}_\alpha < \infty$. To make this a point-wise statement, it is convenient to define Sobolev norms
\begin{equation}\label{sobolev_norm}
\verti{u}_{k,\alpha}^2 := \sum_{\ell = 0}^k \int_0^\infty x^{-2 \alpha} \left(D^\ell u(x)\right)^2 \frac{\d x}{x} = \sum_{\ell = 0}^k \int_{-\infty}^\infty e^{-2 \alpha s} \left(\partial_s^\ell u(e^s)\right)^2 \d s \quad \mbox{with } \, k \in \N_0, \; \, \alpha \in \R.
\end{equation}
We also use the notation $(\cdot,\cdot)_{k,\alpha}$ and $(\cdot,\cdot)_\alpha$ for the induced inner products. Setting $v(s) := e^{- \alpha s} u\left(e^s\right)$, it is elementary to see that $\verti{u}_{k,\alpha} \sim_\alpha \vertii{v}_{W^{k,2}(\R)}$. In particular $\verti{u}_{1,\alpha} < \infty$ implies $u(x) = o\left(x^\alpha\right)$ as $x \searrow 0$. Note, however, that control of an increasing number of $D$-derivatives does not lead to better regularity properties of $u(x)$ in $x = 0$ as $D$ is scaling-invariant in $x$. Essentially the index $k$ in the norm $\verti{u}_{k,\alpha}$ controls the \emph{interior regularity} of $u$, whereas $\alpha$ yields control on the regularity at the boundary $x = 0$. The identification with the standard Sobolev norms also guarantees that the test functions $C_0^\infty((0,\infty))$ are dense in the spaces $\{u \mbox{ locally integrable}: \, \verti{u}_{k,\alpha} < \infty\}$. More details are contained in \cite[Sec.~4]{ggko.2014}.

\medskip

The well-posedness proof of \cite{ggko.2014} relied on the control of the initial data in the norm $\vertiii{\cdot}_0$, where
\begin{equation}\label{norm_initial}
\vertiii{u^{(0)}}_0^2 := \verti{u^{(0)}}_{k+6,-\delta}^2 + \verti{u^{(0)} - u^{(0)}_0}_{k+6,\delta}^2,
\end{equation}
$k \ge 3$, $u^{(0)}_0 = u^{(0)}\left(x = 0\right)$ (the boundary value of $u^{(0)}$), and $\delta > 0$ is chosen sufficiently small. Indeed, the subsequent result will use the same norm with sufficiently large $k \in \N$ and sufficiently small $\delta > 0$ depending on the value of $N_0$. We notice that by quite elementary arguments (cf.~\cite[Eq.~(8.5)]{ggko.2014}) the Lipschitz norm of the hodograph coordinates $Z(0,x)-x$ is controlled by this norm, i.e., $\sup_{x > 0} \verti{u^{(0)}(x)} \lesssim \vertiii{u^{(0)}}_0$. Smallness of the Lipschitz norm of $Z(0,x)-x$ ensures strict monotonicity and thus invertibility of transformation~\eqref{hodograph}. This was in fact the crucial assumption in \cite{j.2015, k.1999}.
\section{The main result}
\subsection{A regularity result\label{ssec:regularity}}
For motivating our main result, we may have a heuristic look at the properties of problem~\eqref{cauchy}: Since we are interested in a perturbative result, that is, $u^{(0)}$ and $u$ are assumed to be small in suitable norms (cf.~e.g.~\eqref{norm_initial}), the precise understanding of the linearized problem
\begin{subequations}\label{lin_cauchy}
\begin{align}
x \partial_t u + p(D) u &= f \quad \mbox{for } \, t, x > 0, \label{lin_pde}\\
u_{|t = 0} &= u^{(0)} \quad \mbox{for } \, x > 0
\end{align}
\end{subequations}
with a right-hand side $f(t,x)$, turns out the be essential. For $x \ll 1$, the term $p(D) u$ will be dominant. Considering $p(D) u \approx f$ for $x \ll 1$, all solutions of this ordinary differential equation (ODE) are given by the sum of a particular solution and a linear combination of solutions to the homogeneous equation $p(D) u \equiv 0$. The solution space of $p(D) u \equiv 0$ is spanned by $x^0$, $x^\beta$, $x^{-\beta-\frac 1 2}$, and $x^{-\frac 3 2}$. Clearly, the last two powers are ruled out by compatibility with transformations~\eqref{hodograph} and \eqref{trafo}. Nonetheless, the other two, $x^0$ and $x^\beta$, generically appear in the expansion of the solution $u$ close to the boundary $x = 0$ as a detailed analysis of the resolvent equation shows \cite[Sec.~6]{ggko.2014}. Due to the perturbation $x \partial_t u$, one may then expect that $u = v + x^\beta w$, where both $v$ and $w$ are smooth functions in $x$ up to the boundary $x = 0$. While it is in principle possible to construct solutions of this form for the linear problem \eqref{lin_cauchy}, such a structure is incompatible with the nonlinear equation \eqref{pde} that mixes the powers $x$ and $x^\beta$. What we are instead aiming at is
\begin{equation}\label{structure_u}
u(t,x) = \bar u\left(t,x,x^\beta\right),
\end{equation}
where $\bar u = \bar u(t,x,y)$ is a smooth function in $(x,y) \in \left\{\left(x^\prime,y^\prime\right) \in \R^2: \, x^\prime, y^\prime  \ge 0\right\}$. Indeed, such a result can be expected in view of the analysis of source-type self-similar solutions by the author joint with Giacomelli and Otto \cite{ggo.2013}, where a stronger result is shown: There
\[
h(t,z) = t^{-\frac 1 6} H(x) \quad \mbox{with } \, x = z t^{-\frac 1 6} \, \mbox{ and } \, H(x) = C x^{\frac 3 2}\left(1 + \bar u\left(x,x^\beta\right)\right),
\]
where $C$ is a positive constant and $\bar u(x,y)$ is analytic in a neighborhood of $(x,y) = (0,0)$. Not surprisingly, as in \eqref{cauchy} or \eqref{lin_cauchy}, $\beta = \frac{\sqrt{13}-1}{4}$ turns out to be the root of a polynomial that determines the linearized problem for $H$.

\medskip

We denote by
\begin{equation}\label{def_kn}
K_{N_0} := \left\{n_1 + \beta n_2: \, (n_1, n_2) \in \N_0^2, \; n_1 + \beta n_2 < N_0\right\}
\end{equation}
the set of admissible exponents up to order $O\left(x^{N_0}\right)$ (cf.~Fig.~\ref{fig:power}). The main result of the note reads as follows:
\begin{theorem}\label{th:main}
For any $N_0 \in \N_0$ there exist $\eps > 0$, $k \in \N$, and $\delta > 0$ such that if
\[
\vertiii{u^{(0)}}_0 \stackrel{\eqref{norm_initial}}{=} \sqrt{\verti{u^{(0)}}_{k+6,-\delta}^2 + \verti{u^{(0)} - u^{(0)}_0}_{k+6,\delta}^2} \le \eps,
\]
the unique solution $u$ of \eqref{cauchy} (constructed in \cite[Th.~3.1]{ggko.2014}) fulfills the point-wise expansion (cf.~Fig.~\ref{fig:power})
\begin{equation}\label{expansion}
u(t,x) = \sum_{i \in K_{N_0}} u_i(t) x^i + R_{N_0}(x,t) x^{N_0} \quad \mbox{as } \, x \searrow 0 \, \mbox{ and } \, t > 0,
\end{equation}
where the coefficients $u_i = u_i(t)$ are continuous functions in $t > 0$, and the correction $R_{N_0} = R_{N_0}(x,t)$ is again continuous and for every $t > 0$ uniformly bounded in $x > 0$. Furthermore,
\begin{subequations}\label{decay_as}
\begin{equation}
u_i(t) = o\left(t^{-i}\right) \quad \mbox{as } \, t \to \infty
\end{equation}
and
\begin{equation}
R_{N_0}(x,t) = o\left(t^{-N_0}\right) \quad \mbox{as } \, t \to \infty \, \mbox{ uniformly in } \, x > 0.
\end{equation}
\end{subequations}
\end{theorem}

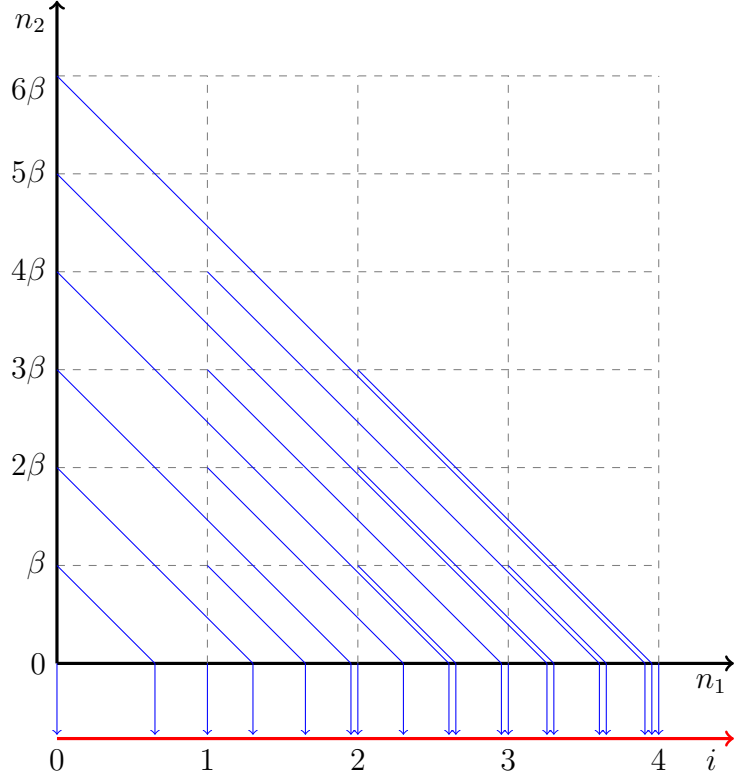
\begin{figure}[htp]
\centering
\begin{tikzpicture}[scale=1]
\draw[very thick,->] (-4,0) -- (5,0);
\draw (4.7,0) node[anchor=north] {$n_1$};
\draw (-4,0) node[anchor=east] {$0$};
\draw [gray,dashed] (-4,1.3028) -- (4,1.3028);
\draw (-4,1.3028) node[anchor=east] {$\beta$};
\draw [gray,dashed] (-4,2.6056) -- (4,2.6056);
\draw (-4,2.6056) node[anchor=east] {$2\beta$};
\draw [gray,dashed] (-4,3.9083) -- (4,3.9083);
\draw (-4,3.9083) node[anchor=east] {$3\beta$};
\draw [gray,dashed] (-4,5.2111) -- (4,5.2111);
\draw (-4,5.2111) node[anchor=east] {$4\beta$};
\draw [gray,dashed] (-4,6.5139) -- (4,6.5139);
\draw (-4,6.5139) node[anchor=east] {$5\beta$};
\draw [gray,dashed] (-4,7.8167) -- (4,7.8167);
\draw (-4,7.6167) node[anchor=east] {$6\beta$};
\draw[very thick,->] (-4,0) -- (-4,8.8167);
\draw (-4,8.5167) node[anchor=east] {$n_2$};
\draw [gray,dashed] (-2,0) -- (-2,7.8167);
\draw [gray,dashed] (0,0) -- (0,7.8167);
\draw [gray,dashed] (2,0) -- (2,7.8167);
\draw [gray,dashed] (4,0) -- (4,7.8167);
\draw [blue] (-4,1.3028) -- (-2.6972,0);
\draw [blue] (-2,1.3028) -- (-.6972,0);
\draw [blue] (0,1.3028) -- (1.3028,0);
\draw [blue] (2,1.3028) -- (3.3028,0);
\draw [blue] (-4,2.6056) -- (-1.3944,0);
\draw [blue] (-2,2.6056) -- (.6056,0);
\draw [blue] (0,2.6056) -- (2.6056,0);
\draw [blue] (-4,3.9083) -- (-.0917,0);
\draw [blue] (-2,3.9083) -- (1.90832,0);
\draw [blue] (0,3.9083) -- (3.9083,0);
\draw [blue] (-4,5.2111) -- (1.2111,0);
\draw [blue] (-2,5.2111) -- (3.2111,0);
\draw [blue] (-4,6.5139) -- (2.5139,0);
\draw [blue] (-4,7.8167) -- (3.8167,0);
\draw[blue,->] (-4,0) -- (-4,-.95);
\draw[blue,->] (-2,0) -- (-2,-.95);
\draw[blue,->] (0,0) -- (0,-.95);
\draw[blue,->] (2,0) -- (2,-.95);
\draw[blue,->] (4,0) -- (4,-.95);
\draw[blue,->] (-2.6972,0) -- (-2.6972,-.95);
\draw[blue,->] (-.6972,0) -- (-.6972,-.95);
\draw[blue,->] (1.3028,0) -- (1.3028,-.95);
\draw[blue,->] (3.3028,0) -- (3.3028,-.95);
\draw[blue,->] (-1.3944,0) -- (-1.3944,-.95);
\draw[blue,->] (.6056,0) -- (.6056,-.95);
\draw[blue,->] (2.6056,0) -- (2.6056,-.95);
\draw[blue,->] (-.0917,0) -- (-.0917,-.95);
\draw[blue,->] (1.90832,0) -- (1.90832,-.95);
\draw[blue,->] (3.90832,0) -- (3.90832,-.95);
\draw[blue,->] (1.2111,0) -- (1.2111,-.95);
\draw[blue,->] (3.2111,0) -- (3.2111,-.95);
\draw[blue,->] (2.5139,0) -- (2.5139,-.95);
\draw[blue,->] (3.8167,0) -- (3.8167,-.95);
\draw[very thick,red,->] (-4,-1) -- (5,-1);
\draw (-4,-1) node[anchor=north] {$0$};
\draw (-2,-1) node[anchor=north] {$1$};
\draw (0,-1) node[anchor=north] {$2$};
\draw (2,-1) node[anchor=north] {$3$};
\draw (4,-1) node[anchor=north] {$4$};
\draw (4.7,-1) node[anchor=north] {$i$};
\end{tikzpicture}
\caption{Schematic of the admissible exponents $i \in K_{N_0}$ (cf.~\eqref{def_kn}) in the generalized power series \eqref{expansion} of the solution $u$ up to $i \le 4$.}
\label{fig:power}
\end{figure}
Actually, we are able to prove further regularity properties for the coefficients $u_i(t)$ and the correction $R_{N_0} = R_{N_0}(x,t)$ and also interior-regularity bounds, the presentation of which we postpone to later sections (cf. Lemma~\ref{lem:est_coeff} and Theorem~\ref{th:regularity}). Note that in \cite{ggko.2014} it was only proven that  $u(t,x) = u_0(t) + u_\beta(t) x^\beta + o(x^\beta)$ as $x \searrow 0$ almost everywhere, where $u_0 = u_0(t)$ is bounded and continuous with $u_0(t) \to 0$ as $t \to \infty$ and $t^{\beta - \frac 1 2} u_\beta = t^{\beta - \frac 1 2} u_\beta(t)$ is square integrable with $u_\beta(t) = o\left(t^{\frac 1 2 - \beta}\right)$ as $t \to \infty$ for a subsequence.

\subsection{Discussion}
The result captures the regularizing effect of the degenerate-parabolic equation \eqref{pde}. Unlike in the case of the porous-medium equation
\[
\partial_t h - \partial_z^2 h^m = 0 \quad \mbox{for } \, t > 0 \, \mbox{ and } \, z \in \{h > 0\},
\]
with a constant $m > 1$ -- the second-order counterpart of \eqref{pde} for which a comparison principle holds and solutions become smooth for positive times \cite{a.1988,dh.1998,k.1999} -- here the solution only becomes smooth in the generalized sense \eqref{expansion}.

\medskip

We emphasize that in the case of partial wetting, that is, without loss of generality $\verti{\partial_z h} = 1$ at $z = Z_0(t)$, Kn\"upfer found that generically such singular expansions do appear as well: The solutions turn out to have a polynomial expansion in $x$ and $x \ln x$ for $n = 2$ \cite{k.2011}, and $x$ and $x^{3-n}$ for $n \in \left(0,\frac{14}{5}\right) \setminus \{1,2,\frac 5 2, \frac 8 3, \frac{11}{4}\}$ \cite{k.2012,k.2012.err} (cf.~also \cite{bgk.2016} for a discussion of source-type self-similar solutions with nonzero dynamic contact angle). In these cases, however, Kn\"upfer assumes such an expansion already for the initial data which also need to fulfill additional compatibility conditions. Hence in this case there is no proof available that the solution acquires additional regularity at the contact line for positive times, except for a smoothing effect in \cite[Cor.~4.3]{k.2012,k.2012.err} if one waits sufficiently long. This is due to the different techniques in \cite{k.2011,k.2012} relying on the Mellin transform and a suitable subtraction of the singular expansion multiplied with a cut off at $x = \infty$: The resulting estimates contain terms with different scaling in $x$ and consequently have no distinct scaling in time $t$. Thus it seems impossible to introduce time weights in order to capture the smoothing effect of the (non-)linear equation in this setting. On the other hand, it is well-known in the theory of non-degenerate parabolic equations (in domains with smooth boundary for instance) that solutions typically become smooth for positive times without having to assume well-prepared initial data.

\medskip

This (generalized) smoothing effect is the essentially new insight of Theorem~\ref{th:main}. The fact that we need control of an increasing number $k+6$ of logarithmic derivatives $D$ the larger $N_0$, essentially amounts to having sufficient interior regularity of the initial data. Since the degeneracy of \eqref{pde} is immaterial away from $x = 0$, the interior regularity of problem~\eqref{cauchy} can be treated by standard theory.

\medskip

We further remark that $\eps$ has to be chosen dependent on $N_0$. With our method it appears to be unavoidable to assume this dependence\footnote{As we will observe in Section~\ref{sec:linear}, the coercivity constant of $p(D-N_0)$ with respect to the inner product $(\cdot,\cdot)_\alpha$ vanishes as $\alpha \nearrow N_0$ or $\alpha \searrow N_0-1$, and as $N_0 \to \infty$ the admissible exponents in the interval $(N_0-1,N_0)$ become dense. Thus the constant in the maximal-regularity estimate \eqref{mr_main} blows up as $N_0 \to \infty$.}, so that we cannot determine whether the series $\sum_{i \in K_\infty} u_i(t) x^i$ converges.

\medskip

We also notice that there are connections of this work with the theory of elliptic boundary problems in domains with isolated point singularities at the boundary, for which it is known that singular expansions of solutions occur \cite{kmr.1997}. This is not surprising, as the underlying fluid model, the Stokes (or Navier-Stokes) system, has to be solved on a moving infinite-cusp domain. In the partial wetting case, instead, a moving infinite wedge may be considered and in fact, Kn\"upfer's analysis in \cite{k.2011,k.2012} strongly relies on this analogy.
\subsection{Transformation into the original set of variables}
In \cite[Rm.~3.2, Rm.~3.3, App.~A]{ggko.2014} it was explained that analogous leading-order expansions also hold for the function $h$ and the velocity $V$ (the vertically-averaged horizontal velocity within lubrication theory, cf.~\eqref{continuity_eq}). Indeed, in the expression $V = h \partial_z^3 h$ for the velocity (cf.~\eqref{def_vel}), we may employ the hodograph transform \eqref{hodograph}, i.e., $h = x^{\frac 3 2}$ as well as $\partial_z = F \partial_x$ (cf.~\eqref{tfe_trafo_0}\&\eqref{trafo}), so that
\begin{subequations}\label{vel_mult}
\begin{equation}\label{velocity}
V = \frac 3 2 x^{\frac 3 2} F \partial_x F \partial_x F x^{\frac 1 2} = \tilde \MM(F,F,F)
\end{equation}
with
\begin{equation}\label{multi_vel}
\tilde \MM(F_1,F_2,F_3) := \frac 3 2 F_1 \left(D - \frac 1 2\right) F_2 \left(D + \frac 1 2\right) F_3.
\end{equation}
\end{subequations}
Equations~\eqref{vel_mult} can be used to derive
\begin{subequations}\label{vel_non}
\begin{equation}
V = A(t) + B(t,x) + C(t,x)
\end{equation}
with
\begin{align}
A &= - \frac 3 8 (1+u_0)^3, \\
B &= 3 (1+u_0)^2 \tilde\MM_\mathrm{sym}(u-u_0,1,1) = \frac 3 2 (1+u_0)^2 \tilde p(D) (1+u_0) \quad \mbox{where } \, \tilde p(\zeta) = \left(\zeta+\beta+\frac 1 2\right) (\zeta-\beta), \\
C &= 3 (1+u_0) \tilde\MM_\mathrm{sym}(u-u_0,u-u_0,1) + \tilde\MM_\mathrm{sym}(u-u_0,u-u_0,u-u_0).
\end{align}
\end{subequations}
With help of expansion~\eqref{expansion}, equations~\eqref{vel_non} upgrade to
\begin{equation}\label{expansion_v}
V(t,Z(t,x)) = - \frac 3 8 \left(1+u_0(t)\right)^3 + \sum_{\substack{i \in K_{N_0}\\ 1 \le i < N_0}} V_i(t) x^i + O\left(x^{N_0}\right) \quad \mbox{as } \, x \searrow 0 \, \mbox{ and } \, t > 0,
\end{equation}
where the coefficients $V_i(t)$ and the correction $O\left(x^{N_0}\right)$ fulfill decay estimates as in \eqref{decay_as}. From \eqref{expansion_v} we can also read off the velocity $V_0(t)$ and the position $Z_0(t)$ of the contact line as
\[
V_0(t) = \frac 3 8 \left(1+u_0(t)\right)^3 \qquad \mbox{and} \qquad Z_0(t) = z_0 - \frac 3 8 \int_0^t \left(1 + u_0\left(t^\prime\right)\right)^3 \, \d t^\prime,
\]
where $z_0 \in \R$ is a free parameter.

\medskip

Finally, we can derive an expression for the expansion of the film height $h$ and the velocity $V$ in the vicinity of the contact line in terms of the original variables $t$ and $z$: Employing \eqref{expansion} in \eqref{trafo} in conjunction with \eqref{trafo_2}, we conclude
\[
Z_x = \left(1+u_0(t)\right)^{-1} + \sum_{\substack{i \in K_{N_0}\\ \beta \le i < N_0}} (Z_x)_i(t) x^i + O\left(x^{N_0}\right) \quad \mbox{as } \, x \searrow 0 \, \mbox{ and } \, t > 0,
\]
where the coefficients and remainder fulfill estimates analogous to \eqref{decay_as}. Integration of this expansion gives
\begin{eqnarray}\nonumber
\tilde x &:=& Z(t,x) - Z_0(t) \\
&=& x \left(\left(1+u_0(t)\right)^{-1} + \sum_{\substack{i \in K_{N_0}\\ \beta \le i < N_0}} (1+i)^{-1} (Z_x)_i(t) x^i + O\left(x^{N_0}\right)\right) \quad \mbox{as } \, x \searrow 0 \, \mbox{ and } \, t > 0. \label{def_tx}
\end{eqnarray}
Inversion yields
\begin{equation}\label{exp_x}
x = \left(1+u_0(t)\right) \tilde x \left(1 + \sum_{\substack{i \in K_{N_0}\\ \beta \le i < N_0}} c_i(t) \tilde x^i + O\left(\tilde x^{N_0}\right)\right) \quad \mbox{as } \, \tilde x \searrow 0 \, \mbox{ and } \, t > 0,
\end{equation}
with coefficients $c_i(t)$ and remainder $O\left(\tilde x^{N_0}\right)$ obeying estimates analogous to those in \eqref{decay_as}. Utilizing expansion~\eqref{exp_x} in the hodograph transform \eqref{hodograph}, we end up with
\[
h(t,z) = \tilde x^{\frac 3 2} \left(1 + \sum_{i \in K_N} \tilde u_i(t) \tilde x^i + O\left(\tilde x^{N_0}\right)\right) \quad \mbox{as } \, \tilde x \searrow 0 \, \mbox{ and } \, t > 0,
\]
where the $\tilde u_i(t)$ are (at least) continuous in time $t$, $\tilde x := z - Z_0(t)$ (cf.~\eqref{def_tx}), and decay estimates as in \eqref{decay_as} for the $\tilde u_i(t)$ and the remainder $O\left(\tilde x^{N_0}\right)$ hold true. We leave it up to the reader to directly relate the coefficients $\tilde u_i(t)$ to the $u_i(t)$ appearing in \eqref{expansion}. Furthermore, we can also derive an expansion for the velocity $V = V(t,z)$ by using \eqref{exp_x} in \eqref{expansion_v}:
\[
V(t,z) = - \frac 3 8 \left(1+u_0(t)\right)^3 + \sum_{\substack{i \in K_{N_0}\\ 1 \le i < N_0}} \tilde V_i(t) \tilde x^i + O\left(\tilde x^{N_0}\right) \quad \mbox{as } \, \tilde x \searrow 0 \, \mbox{ and } \, t > 0,
\]
where again $\tilde x := z - Z_0(t)$  and the $\tilde V_i(t)$ and the remainder term $O\left(\tilde x^{N_0}\right)$ meet decay estimates as in \eqref{decay_as}. We remark that it is open, whether a similar expansion for the velocity occurs for the (Navier-)Stokes equations on a moving infinite-cusp domain.

\subsection{Outline}
The paper consists of two parts: the linear theory (discussed in Section~\ref{sec:linear}) and the nonlinear theory (contained in Section~\ref{sec:nonlinear}). After recalling some of the basic notions on coercivity and maximal regularity in Section~\ref{ssec:coerc}, Section~\ref{ssec:formal} deals with the formal structure of the linear degenerate-parabolic equation \eqref{lin_cauchy}. Here we demonstrate that by applying an appropriate combination of \emph{scaling-invariant} operators $D-\gamma$ with $\gamma \in \R$ to the linear equation \eqref{lin_pde}, we are able to derive maximal-regularity estimates for \eqref{lin_cauchy} that control the singular expansion of $u$ at $x = 0$ to arbitrary orders (cf.~Section~\ref{ssec:heur_par}). Since all higher-order equations for $u$ are \emph{scaling-invariant}, the corresponding maximal-regularity estimates have a distinct scaling in $x$. Thus, when multiplied with an appropriate time weight, we can combine them to a quasi scaling-invariant maximal-regularity estimate for $u$ in terms of the initial data $u^{(0)}$ in the \emph{quasi-minimal} norm \eqref{norm_initial} and the right-hand side $f$. These arguments are made rigorous in Sections~\ref{ssec:func} and \ref{ssec:rigorous} (cf.~Proposition~\ref{prop:main_lin}) without going into all details. In particular Section~\ref{ssec:rigorous} is not essential for the understanding of the main ideas.

\medskip

Finally, in Section~\ref{sec:nonlinear} we prove our main regularity result, Theorem~\ref{th:regularity}, from which Theorem~\ref{th:main} follows as a special case. The proof strategy is standard and requires two ingredients: maximal-regularity estimates for the linearized problem \eqref{lin_cauchy} and the factorization of the nonlinearity $\NN(u)$ (cf.~\eqref{nonlinearity}) given in Proposition~\ref{prop:nonlinear_est}. The latter is in fact the non-trivial part of the proof (cf.~Section~\ref{ssec:nonest}) and requires detailed estimates that rely on the symmetry properties of the multi-linear form $\MM$ (cf.~\eqref{5linear}).

\section{The linear problem\label{sec:linear}}
\subsection{Coercivity and parabolic maximal regularity\label{ssec:coerc}}
Again, we briefly repeat some of the linear theory in \cite{ggko.2014}. Consider the linear problem
\begin{subequations}\label{lin_cauchy_alt}
\begin{align}
x \partial_t u + P(D) u &= f \quad \mbox{for } \, t, x > 0, \label{lin_pde_alt}\\
u_{|t = 0} &= u^{(0)},
\end{align}
\end{subequations}
which is structurally the same as \eqref{lin_cauchy} but with a general fourth-order polynomial $P(\zeta)$. We assume that the zeros $\gamma_1\le \cdots \le \gamma_4$ of $P(\zeta)$ are real. Then we know from \cite[Prop.~5.3]{ggko.2014} that there is a range of weights $\alpha$ -- which we may call \emph{coercivity range} -- such that $P(D)$ is formally coercive, i.e.,~$(u,P(D)u)_\alpha \gtrsim_\alpha \verti{u}_{2,\alpha}^2$ for all $u \in C_0^\infty((0,\infty))$. A sufficient criterion (which can be elementarily computed) is
\begin{equation}\label{coerc_cond}
\alpha \in (-\infty, \gamma_1) \cap (\gamma_2,\gamma_3) \cap (\gamma_4,\infty) \quad \mbox{and} \quad (\alpha-m(\gamma))^2 \le \frac{\sigma^2(\gamma)}{3},
\end{equation}
where $m(\gamma) := \frac 1 4 \sum_{j = 1}^4 \gamma_j$ (mean of the zeros $\gamma_j$) and $\sigma^2(\gamma) := \frac 1 4 \sum_{j = 1}^4 (\gamma_j - m(\gamma))^2$ (variance of the zeros $\gamma_j$). In the particular case of $P(D) = p(D)$, \eqref{coerc_cond} yields that the coercivity range contains the interval $(-1,0)$.

\medskip

Now suppose that $\alpha \in \R$ is in the coercivity range of $P(D)$. Then at least formally by quite elementary arguments (cf.~\cite[Sec.~2, Sec.~7.1]{ggko.2014}), we can derive a differential version of a maximal-regularity estimate for \eqref{lin_pde_alt} that reads
\begin{equation}\label{maxreg_diff}
\frac{\d}{\d t} \verti{u}_{\ell+2,\alpha-\frac 1 2}^2 + \verti{\partial_t u}_{\ell,\alpha-1}^2 + \verti{u}_{\ell+4,\alpha}^2 \lesssim_{\ell,\alpha} \verti{f}_{\ell,\alpha}^2, \quad \mbox{where } \, \ell \in \N_0.
\end{equation}
Indeed, in \eqref{maxreg_diff} derivatives $f, \cdots, D^\ell f$ control $u, \cdots, D^{\ell+4} u$ in the same norm, which is the maximal control in space one can expect as \eqref{lin_pde_alt} is fourth-order in $D$. The additional control of the time derivative $\partial_t u, \cdots, D^\ell \partial_t u$ (with reduced weight due to the degeneracy in \eqref{lin_cauchy_alt}) can be obtained by using control of $u, \cdots, D^{\ell+4} u$ and the fact that \eqref{lin_pde_alt} is fulfilled. This also yields control of the trace $\frac{\d}{\d t} \verti{u}_{\ell+2,\alpha-\frac 1 2}^2$ by interpolation. We refer to Section~\ref{ssec:rigorous} for more details.

\medskip

By multiplying \eqref{maxreg_diff} with a time weight $t^\sigma$ (where $\sigma \ge 0$), we obtain the integrated version of \eqref{maxreg_diff}, i.e.,
\begin{equation}\label{maxreg_int}
\begin{aligned}
&\sup_{t \ge 0} t^{2 \sigma} \verti{u}_{\ell+2,\alpha-\frac 1 2}^2 + \int_0^\infty t^{2 \sigma} \left(\verti{\partial_t u}_{\ell,\alpha-1}^2 + \verti{u}_{\ell+4,\alpha}^2\right) \d t \\
& \quad \lesssim_{\ell,\alpha} \delta_{\sigma,0} \verti{u^{(0)}}_{\ell+2,\alpha-\frac 1 2}^2 + \int_0^\infty t^{2 \sigma} \verti{f}_{\ell,\alpha}^2 \d t + 2 \sigma \int_0^\infty t^{2 \sigma - 1} \verti{u}_{\ell+2,\alpha-\frac 1 2}^2 \d t, \quad \mbox{where } \, \ell \in \N_0.
\end{aligned}
\end{equation}
The purpose of introducing time weights is two-fold: On the one hand they will enable us to prove the decay estimates as $t \to \infty$ in \eqref{decay_as} for the coefficients $u_i$ and the remainder $R_{N_0}$. While this would be irrelevant on a finite time interval, they also make it possible to prove regularity immediately after time $t = 0$, whereas without them the time after which regularity is obtained is unknown (cf.~\cite[Cor.~4.3]{k.2012,k.2012.err} for a similar case).

\medskip

Estimate~\eqref{maxreg_int} will be the basis of all linear estimates in the sequel.

\subsection{The formal structure of the linear equation\label{ssec:formal}}
As pointed out in \cite[Sec.~2]{ggko.2014}, just applying maximal regularity of the form \eqref{maxreg_int} with $\sigma = 0$ to the linear equation \eqref{lin_pde} is not sufficient in order to obtain well-posedness of the corresponding nonlinear problem \eqref{cauchy}. This is so, since only negative weights $\alpha$ are admissible (viz.~in the coercivity range of $p(D)$) and hence no control of the boundary value $u_0$ or the sup-norm $\sup_{t,x > 0} \verti{u(t,x)}$ can be achieved. On the other hand, as products of up to five factors in $\{u, D u, \cdots, D^4 u\}$ appear in the nonlinearity $\NN(u)$ (cf.~\eqref{5linear}~and~\eqref{nonlinearity}), the control of $\sup_{t,x > 0} \verti{u(t,x)}$ appears to be necessary for proving well-posedness by a contraction argument. In order to circumvent this problem, it was convenient to apply $p(D-1)$ to the linear equation \eqref{lin_pde} that, by using the commutation relation $D x = x (D+1)$, transforms into
\begin{equation}\label{lin_v1}
x \partial_t v^{(1)} + p(D-1) v^{(1)} = g^{(1)} \quad \mbox{for } \, t, x > 0,
\end{equation}
where $v^{(1)} := p(D) u$ and $g^{(1)} := p(D-1) f$. Since the coercivity range has translated to the positive interval $(0,1)$ (cf.~e.g.~\eqref{coerc_cond}), one obtains better control on the regularity of $v$ at the boundary $x = 0$. This is not surprising as in view of \eqref{expansion} we expect $u(t,x) = u_0(t) + u_\beta(t) x^\beta + O(x)$ as $x \searrow 0$ and $t > 0$ and the powers $x^0$ and $x^\beta$ are in the kernel of $p(D)$, whence $v^{(1)}(t,x) = O(x)$ as $x \searrow 0$ and $t > 0$. Furthermore, by compatibility $f(t,x) = O(x)$ as $x \searrow 0$ and $t > 0$ and therefore also $g^{(1)}(t,x) = O(x)$ as $x \searrow 0$ and $t > 0$. As a second step one can then retrieve regularity information on $u$ from regularity information on $v^{(1)}$ using elliptic estimates for the operator $p(D)$ (cf.~\cite[Sec.~2, Lem.~7.2]{ggko.2014}).

\medskip

Before reviewing the arguments, we point out the limitations of the ansatz: As a natural second step, we apply the operator $p(D-2)$ to equation \eqref{lin_v1} and obtain
\begin{equation}\label{lin_v2}
x \partial_t v^{(2)} + p(D-2) v^{(2)} = g^{(2)} \quad \mbox{for } \, t, x > 0,
\end{equation}
with $v^{(2)} := p(D-1) v^{(1)}$ and $g^{(2)} := p(D-2) g^{(1)}$. Following the argumentation above and noting that the coercivity range of the operator $p(D-2)$ contains the interval $(1,2)$, we apply the maximal-regularity estimate \eqref{maxreg_int} to \eqref{lin_v2} and seemingly obtain even better control on the boundary regularity of $v^{(2)}$ which formally suggests $v(t,x) = O\left(x^{2-\eps}\right)$ as $x \searrow 0$ and $t > 0$, where $\eps > 0$ is arbitrarily small. Apparently, such a claim is too strong as generically also the term $x^{2 \beta}$ appears in the expansion of $u$ (cf.~Section~\ref{ssec:regularity}), this term is not in the kernel of $p(D-1) p(D)$, and thus in general $v(t,x) = O\left(x^{2\beta}\right)$ as $x \searrow 0$ and $t > 0$.

\medskip

In order to work around this problem, we set
\[
I_2 := \left\{n_1 + \beta n_2: \, (n_1,n_2) \in \N_0^2, \; 1 < n_1 + \beta n_2 < 2\right\} \setminus \{1 + \beta\} = \{2 \beta, 3 \beta\}
\]
and apply the operator $\prod_{i \in I_2} (D-i) = (D-2\beta) (D-3\beta)$ to \eqref{lin_v2}. Setting $w^{(2)} := \left(\prod_{i \in I_2} (D-i)\right) v^{(2)}$ and $r^{(2)} := \left(\prod_{i \in I_2} (D-i)\right) g^{(2)}$, equation~\eqref{lin_v2} turns into
\begin{equation}\label{lin_w2}
x \partial_t w^{(2)} + p(D-2) w^{(2)} = r^{(2)} + x q_2(D) \partial_t v^{(2)} \quad \mbox{for } \, t, x > 0,
\end{equation}
where $q_2(D) = \prod_{i \in I_2} (D-i) - \prod_{i \in I_2} (D-i+1) = - 2 D + 5 \beta - 1$ is a polynomial of degree $\verti{I_2}-1 = 1$. The additional term $x q_2(D) \partial_t v^{(2)}$ appears as the commutator of $\prod_{i \in I_2} (D-i)$ and the multiplication with $x$ does not vanish. Nevertheless, \eqref{lin_w2} is structurally advantageous compared to \eqref{lin_v2} as now we may indeed expect $w^{(2)}(t,x) = O(x^2)$, $r^{(2)}(t,x) =O(x^2)$, and $x q_2(D) \partial_t v^{(2)}(t,x) = o(x^2)$ as $x \searrow 0$ and $t > 0$. Hence applying the maximal-regularity estimate \eqref{maxreg_int} with $\alpha \in (1,2)$, we obtain better control on $w^{(2)}$ (assuming $\sigma > 0$):
\begin{equation}\label{mr_w2}
\begin{aligned}
&\sup_{t \ge 0} t^{2 \sigma} \verti{w^{(2)}}_{\ell+2,\alpha-\frac 1 2}^2 + \int_0^\infty t^{2 \sigma} \left(\verti{\partial_t w^{(2)}}_{\ell,\alpha-1}^2 + \verti{w^{(2)}}_{\ell+4,\alpha}^2\right) \d t \\
& \quad \lesssim_{\ell,\alpha} \int_0^\infty t^{2 \sigma} \verti{r^{(2)}}_{\ell,\alpha}^2 \d t + 2 \sigma \int_0^\infty t^{2 \sigma - 1} \verti{w^{(2)}}_{\ell+2,\alpha-\frac 1 2}^2 \d t + \int_0^\infty t^{2 \sigma} \verti{q_2(D) \partial_t v^{(2)}}_{\ell, \alpha-1}^2 \d t.
\end{aligned}
\end{equation}

Yet, the additional term $\int_0^\infty t^{2 \sigma} \verti{q_2(D) \partial_t v^{(2)}}_{\ell, \alpha-1}^2 \d t$ in \eqref{mr_w2} has to be treated: Since $\alpha-1 \in (0,1)$ is in the coercivity range of $p(D-1)$, we can apply maximal regularity of the form \eqref{maxreg_int} to the \emph{time-differentiated} version of \eqref{lin_v1}, i.e.,
\begin{equation}\label{lin_tv1}
x \partial_t^2 v^{(1)} + p(D-1) \partial_t v^{(1)} = \partial_t g^{(1)} \quad \mbox{for } \, t, x > 0,
\end{equation}
which leaves us with
\begin{equation}\label{mr_tv1}
\begin{aligned}
& \sup_{t \ge 0} t^{2 \sigma} \verti{\partial_t v^{(1)}}_{\ell^\prime+2,\alpha-\frac 3 2}^2 + \int_0^\infty t^{2 \sigma} \left(\verti{\partial_t^2 v^{(1)}}_{\ell^\prime,\alpha-2}^2 + \verti{\partial_t v^{(1)}}_{\ell^\prime+4,\alpha-1}^2\right) \d t \\
& \quad \lesssim_{\ell^\prime,\alpha} \int_0^\infty t^{2 \sigma} \verti{\partial_t g^{(1)}}_{\ell^\prime,\alpha-1}^2 \d t + 2 \sigma \int_0^\infty t^{2 \sigma - 1} \verti{\partial_t v^{(1)}}_{\ell^\prime+2,\alpha-\frac 3 2}^2 \d t.
\end{aligned}
\end{equation}
Assuming $\ell^\prime \ge \ell + \verti{I_2} - 5 = \ell - 3$ and noting that trivially $\verti{q_2(D) \partial_t v^{(2)}}_{\ell, \alpha-1}^2 \lesssim \verti{\partial_t v^{(1)}}_{\ell^\prime+4,\alpha}^2$, the combination of \eqref{mr_w2} and \eqref{mr_tv1} yields
\begin{equation}\label{mr_w2_tv1}
\begin{aligned}
&\sup_{t \ge 0} t^{2 \sigma} \left(\verti{w^{(2)}}_{\ell+2,\alpha-\frac 1 2}^2 + \verti{\partial_t v^{(1)}}_{\ell^\prime+2,\alpha-\frac 3 2}^2\right) \\
&+ \int_0^\infty t^{2 \sigma} \left(\verti{\partial_t w^{(2)}}_{\ell,\alpha-1}^2 + \verti{w^{(2)}}_{\ell+4,\alpha}^2 + \verti{\partial_t^2 v^{(1)}}_{\ell^\prime,\alpha-2}^2 + \verti{\partial_t v^{(1)}}_{\ell^\prime+4,\alpha-1}^2\right) \d t \\
& \quad \lesssim_{\ell,\ell^\prime,\alpha} \int_0^\infty t^{2 \sigma} \left(\verti{r^{(2)}}_{\ell,\alpha}^2 + \verti{\partial_t g^{(1)}}_{\ell^\prime,\alpha-1}^2\right) \d t \\
& \qquad\qquad + 2 \sigma \int_0^\infty t^{2 \sigma - 1} \left(\verti{w^{(2)}}_{\ell+2,\alpha-\frac 1 2}^2 + \verti{\partial_t v^{(1)}}_{\ell^\prime+2,\alpha-\frac 3 2}^2 \d t\right) \d t.
\end{aligned}
\end{equation}
Now the solutions $w^{(2)}$ and $v^{(1)}$ are estimated in sufficiently strong norms by $\partial_t g^{(1)}$ and $r^{(2)}$ in respective norms and the integral
\begin{equation}\label{remainder1}
\int_0^\infty t^{2 \sigma - 1} \left(\verti{w^{(2)}}_{\ell+2,\alpha-\frac 1 2}^2 + \verti{\partial_t v^{(1)}}_{\ell^\prime+2,\alpha-\frac 1 2}^2 \right) \d t.
\end{equation}
The norms in \eqref{remainder1} have reduced spatial and temporal weights and thus it is possible to absorb these terms by lower-order estimates\footnote{Here ``lower order" is meant in the sense of using norms with lowered weights and the same number of time derivatives.}. We will detail the arguments in Section~\ref{ssec:heur_par}. Notably it seems unavoidable to combine higher-regularity estimates in space with higher-regularity estimates in time as opposed to the case of the thin-film equation \eqref{tfe_general} with $n = 1$, that is, the lubrication approximation of the Hele-Shaw cell \cite[Sec.~8, Sec.~9]{gko.2008}. This has already been observed by Kn\"upfer in \cite{k.2011,k.2012} for the partial wetting case using rather different techniques.

\medskip

Before addressing the issue of dealing with the lower-order terms in \eqref{mr_w2_tv1}, we will first systematize our observations: In order to obtain better control on the solution, we apply the operator $p(D-3)$ to equation \eqref{lin_w2}. Thus we arrive at
\begin{equation}\label{lin_v3}
x \partial_t v^{(3)} + p(D-3) v^{(3)} = g^{(3)} + x q_2(D) p(D-2) \partial_t v^{(2)} \quad \mbox{for } \, t, x > 0,
\end{equation}
where we have set $v^{(3)} := p(D-2) w^{(2)}$ and $g^{(3)} := p(D-3) r^{(2)}$. Again, we cannot expect to have $v^{(3)}(t,x) = O(x^3)$ as $x \searrow 0$ and $t > 0$ as the set
\[
I_3 := \left\{n_1 + \beta n_2: \, (n_1,n_2) \in \N_0^2, \; 2 < n_1 + \beta n_2 < 3\right\} \setminus \{2 + \beta\}
\]
is non-empty (cf.~Fig.~\ref{fig:power}). Applying $\prod_{i \in I_3} (D-i)$ to \eqref{lin_w2}, we obtain
\begin{equation}\label{lin_w3_0}
\begin{aligned}
& x \partial_t w^{(3)} + p(D-3) w^{(3)} \\
& \quad = r^{(3)} + x \tilde q_3(D) \partial_t v^{(3)} + x q_2(D) p(D-2) \left(\prod_{i \in I_3} (D-i+1)\right) \partial_t v^{(2)} \quad \mbox{for } \, t, x > 0,
\end{aligned}
\end{equation}
where again $w^{(3)} := \left(\prod_{i \in I_3} (D-i)\right) v^{(3)}$, $r^{(3)} := \left(\prod_{i \in I_3} (D-i)\right) g^{(3)}$, and $\tilde q_3(D)= \prod_{i \in I_3} (D-i) - \prod_{i \in I_3} (D-i+1)$ is a polynomial in $D$ of degree $\verti{I_3}-1$ that originates from the commutator of $x$ and $\prod_{i \in I_3} (D-i)$. Then we make the rather trivial observation $I_2 \subset I_3 - 1 := \{i-1: \, i \in I_3\}$ (cf.~Fig.~\ref{fig:power}). By exploiting
\begin{align*}
p(D-2) \left(\prod_{i \in I_3} (D-i+1)\right) \partial_t v^{(2)} &= \left(\prod_{i \in (I_3-1) \setminus I_2} (D-i)\right) p(D-2) \underbrace{\left(\prod_{i \in I_2} (D-i)\right) v^{(2)}}_{= w^{(2)}} \\
&= \left(\prod_{i \in (I_3-1) \setminus I_2} (D-i)\right) v^{(3)}
\end{align*}
and setting
\[
q_3(D) := \tilde q_3(D) + q_2(D) \prod_{i \in (I_3-1) \setminus I_2} (D-i),
\]
which is a polynomial of degree less or equal to $\verti{I_3}-1$, we can rewrite the last two terms in \eqref{lin_w3_0} and obtain:
\begin{equation}\label{lin_w3}
x \partial_t w^{(3)} + p(D-3) w^{(3)} = r^{(3)} + x q_3(D) \partial_t v^{(3)} \quad \mbox{for } \, t, x > 0.
\end{equation}
As \eqref{lin_w3} is structurally the same as \eqref{lin_w2}, this procedure can be iterated, and we arrive at the following set of equations
\begin{equation}\label{lin_wnm}
(x \partial_t  + p(D-n)) \partial_t^m w^{(n)} = \partial_t^m r^{(n)} + x q_n(D) \partial_t^{m+1} v^{(n)} \quad \mbox{for } \, t, x > 0,
\end{equation}
where we define the sets of indices
\begin{subequations}\label{def_jnin}
\begin{align}
I_n &:= \left\{n_1 + \beta n_2: \, (n_1,n_2) \in \N_0^2, \; n-1 < n_1 + \beta n_2 < n\right\} \setminus \{n-1+\beta\},\\
J_n &:= \left\{n_1 + \beta n_2: \, (n_1,n_2) \in \N_0^2, \; 0 < n_1 + \beta n_2 < n\right\} \setminus \left(\N \cup (\N_0 + \beta)\right) = \cup_{n^\prime = 1}^n I_{n^\prime}, \label{def_jn}
\end{align}
\end{subequations}
introduce the functions
\begin{subequations}\label{def_vwr}
\begin{align}
v^{(n)} &:= \left(\prod_{n^\prime = 0}^{n-1} p(D-n^\prime)\right) \left(\prod_{i \in J_{n-1}} (D-i)\right) u, \label{def_vn}\\
w^{(n)} &:= \left(\prod_{n^\prime = 0}^{n-1} p(D-n^\prime)\right) \left(\prod_{i \in J_n} (D-i)\right) u = \left(\prod_{i \in I_n} (D-i)\right) v^{(n)}, \label{def_wn}\\
r^{(n)} &:= \left(\prod_{n^\prime = 1}^n p(D-n^\prime)\right) \left(\prod_{i \in J_n} (D-i)\right) f, \label{def_rn}
\end{align}
\end{subequations}
and denote by $q_n(D)$ a polynomial of degree less or equal to $\verti{I_n}-1$. The numbers $n \in \N$ and $m \in \N_0$ are arbitrary.

\subsection{Heuristics for parabolic maximal regularity\label{ssec:heur_par}}
In this section we systematize the ideas of the previous section, leading to estimate~\eqref{mr_w2_tv1}. Throughout the section, all estimates may depend on $N$, $n$, $m$, $\alpha$, or $\delta$. Applying the maximal-regularity estimate \eqref{maxreg_int} to equation~\eqref{lin_w3}, we obtain
\begin{equation}\label{mr_wnm}
\begin{aligned}
&\sup_{t \ge 0} t^{2 (\alpha+n+m) - 3} \verti{\partial_t^m w^{(n)}}_{k(n,m,\alpha^\prime) + 2,\alpha^\prime + n - \frac 3 2}^2 \\
&+ \int_0^\infty t^{2 (\alpha+n+m) - 3} \left(\verti{\partial_t^{m+1} w^{(n)}}_{k(n,m,\alpha^\prime),\alpha^\prime+n-2}^2 + \verti{\partial_t^m w^{(n)}}_{k(n,m,\alpha^\prime)+4,\alpha^\prime+n-1}^2\right) \d t \\
& \quad \lesssim \delta_{2 (\alpha+n+m-1),1} \verti{\partial_t^m w^{(n)}_{| t = 0}}_{k(n,m,\alpha^\prime) + 2,\alpha^\prime + n - \frac 3 2}^2 + \int_0^\infty t^{2 (\alpha+n+m) - 3} \verti{\partial_t^m r^{(n)}}_{k(n,m,\alpha^\prime),\alpha^\prime+n-1}^2 \d t \\
& \qquad + (1-\delta_{n,1}) \int_0^\infty t^{2 (\alpha+n+m) - 3} \verti{\partial_t^{m+1} v^{(n)}}_{k(n,m,\alpha^\prime)+\verti{I_n}-1,\alpha^\prime+n-2}^2 \d t \\
& \qquad + (2 (\alpha+n+m) - 3) \int_0^\infty t^{2 (\alpha+n+m) - 4} \verti{\partial_t^m w^{(n)}}_{k(n,m,\alpha^\prime)+2,\alpha^\prime+n-\frac 3 2}^2 \d t,
\end{aligned}
\end{equation}
where we assume and use the following:
\begin{itemize}
\item[$\bullet$] We take a finite number of weights $\alpha \in [0,1]$ and we choose $\delta > 0$ sufficiently small such that
\begin{itemize}
\item[$\star$] $\alpha^\prime := \alpha \pm \delta \in (0,1)$ if $\alpha \in (0,1)$,
\item[$\star$] $\delta \in (0,1)$ and thus also $1-\delta \in (0,1)$.
\end{itemize}
Thus $\alpha^\prime$ is in the coercivity range of $p(D-1)$ (i.e.,~$\alpha+n-1 \in (n-1,n)$ is in the coercivity range of $p(D-n)$).
\item[$\bullet$] $\delta > 0$ has to be chosen small enough such that if $\alpha_1 < \alpha_2$, then also $\alpha_1 +\delta < \alpha_2 -\delta$. Further smallness conditions on $\delta$ will be specified when necessary.
\item[$\bullet$] We need to assume $2(\alpha+n+m)-3 \ge 0$ so that all time weights are integrable at $t = 0$ (the last term in \eqref{mr_wnm} vanishes for $2(\alpha+n+m)-3 = 0$).
\item[$\bullet$] The indices $k(n,m,\alpha) \in \N_0$, determining the number of $D$-derivatives in the norms appearing in \eqref{mr_wnm}, will be chosen later.
\end{itemize}
The specific choice of the weights $\alpha$ will be explained further below. Choosing proper weights turns out to be essential for obtaining control on the coefficients $u_i(t)$ of the generalized power series \eqref{expansion} of $u$ and for being able to absorb the last line of \eqref{mr_wnm}.

\medskip

Note that estimate~\eqref{mr_wnm} itself is insufficient as the solution $u$ still appears on the estimate's right-hand side in the last two lines.

\subsubsection{Absorption of remnant terms I}
We will first concentrate on absorbing the second but last line of \eqref{mr_wnm} by exploiting the additional time regularity: Therefore we start with estimate~\eqref{mr_wnm} with $n = N$ and $m = 0$. For $N > 1$, the term
\[
\int_0^\infty t^{2 (\alpha+N) - 3} \verti{\partial_t v^{(N)}}_{k(N,0,\alpha^\prime)+\verti{I_N}-1,\alpha^\prime+N-2}^2 \d t
\]
has to be absorbed. This term can be estimated by the left-hand side of \eqref{mr_wnm} for $n = N-1$ and $m = 1$, provided that the indices $k(n,m,\alpha^\prime)$ obey $k(N-1,1,\alpha^\prime) \ge k(N,0,\alpha^\prime) + \verti{I_N} - 1$. Then indeed
\[
\verti{\partial_t v^{(N)}}_{k(N,0,\alpha^\prime)+\verti{I_N}-1,\alpha^\prime+N-2} \lesssim \verti{\partial_t w^{(N-1)}}_{k(N-1,1,\alpha^\prime)+4,\alpha^\prime+N-2}.
\]
Next, supposed that $N > 2$, the term
\[
\int_0^\infty t^{2 (\alpha+N) - 3} \verti{\partial_t^2 v^{(N-1)}}_{k(N-1,1,\alpha^\prime)+\verti{I_{N-1}}-1,\alpha^\prime+N-3}^2 \d t
\]
has to be absorbed, which can be achieved by combining it with estimate~\eqref{mr_wnm} for $n = N - 2$ and $m = 2$, provided that $k(N-2,2,\alpha^\prime) \ge k(N-1,1,\alpha^\prime) + \verti{I_{N-1}} - 1$. Apparently, this procedure can be iterated (cf.~Fig.~\ref{fig:absorb1})
\begin{figure}[t]
\centering
\begin{tikzpicture}[scale=1]
\draw[very thick,->] (.5,1) -- (.5,5);
\draw[very thick,->] (1,.5) -- (5,.5);
\draw [gray] (.4,1) -- (.6,1);
\draw [gray] (1,.4) -- (1,.6);
\draw [gray] (.4,2) -- (.6,2);
\draw [gray] (2,.4) -- (2,.6);
\draw [gray] (.4,3) -- (.6,3);
\draw [gray] (3,.4) -- (3,.6);
\draw [gray] (.4,4) -- (.6,4);
\draw [gray] (4,.4) -- (4,.6);
\draw (.3,1) node[anchor=east] {$1$};
\draw (1,.3) node[anchor=north] {$0$};
\draw (.3,2) node[anchor=east] {$2$};
\draw (2,.3) node[anchor=north] {$1$};
\draw (.3,3) node[anchor=east] {$3$};
\draw (3,.3) node[anchor=north] {$2$};
\draw (.3,4) node[anchor=east] {$4$};
\draw (4,.3) node[anchor=north] {$3$};
\draw [gray] (.9,1) -- (1.1,1);
\draw [gray] (1,.9) -- (1,1.1);
\draw [gray] (.9,2) -- (1.1,2);
\draw [gray] (1,1.9) -- (1,2.1);
\draw [gray] (.9,3) -- (1.1,3);
\draw [gray] (1,2.9) -- (1,3.1);
\draw [gray] (.9,4) -- (1.1,4);
\draw [gray] (1,3.9) -- (1,4.1);
\draw [gray] (1.9,1) -- (2.1,1);
\draw [gray] (2,.9) -- (2,1.1);
\draw [gray] (1.9,2) -- (2.1,2);
\draw [gray] (2,1.9) -- (2,2.1);
\draw [gray] (1.9,3) -- (2.1,3);
\draw [gray] (2,2.9) -- (2,3.1);
\draw [gray] (1.9,4) -- (2.1,4);
\draw [gray] (2,3.9) -- (2,4.1);
\draw [gray] (2.9,1) -- (3.1,1);
\draw [gray] (3,.9) -- (3,1.1);
\draw [gray] (2.9,2) -- (3.1,2);
\draw [gray] (3,1.9) -- (3,2.1);
\draw [gray] (2.9,3) -- (3.1,3);
\draw [gray] (3,2.9) -- (3,3.1);
\draw [gray] (2.9,4) -- (3.1,4);
\draw [gray] (3,3.9) -- (3,4.1);
\draw [gray] (3.9,1) -- (4.1,1);
\draw [gray] (4,.9) -- (4,1.1);
\draw [gray] (3.9,2) -- (4.1,2);
\draw [gray] (4,1.9) -- (4,2.1);
\draw [gray] (3.9,3) -- (4.1,3);
\draw [gray] (4,2.9) -- (4,3.1);
\draw [gray] (3.9,4) -- (4.1,4);
\draw [gray] (4,3.9) -- (4,4.1);
\draw[blue,->] (1.1,3.9) -- (1.9,3.1);
\draw[blue,->] (2.1,2.9) -- (2.9,2.1);
\draw[blue,->] (3.1,1.9) -- (3.9,1.1);
\draw (.5,4.7) node[anchor=east] {$n$};
\draw (4.7,.5) node[anchor=north] {$m$};
\end{tikzpicture}
\caption{Schematic: absorption of remnant terms I. Each node $+$ corresponds to an estimate of the form \eqref{mr_wnm}. The displayed arrows visualize the absorption mechanism for $N = 4$, that is, the remnant at the base of the arrow (forming the second but last line in \eqref{mr_wnm}) is absorbed by the corresponding estimate \eqref{mr_wnm} at the tip of the arrow under the assumption that \eqref{same_alpha} holds true.}
\label{fig:absorb1}
\end{figure}
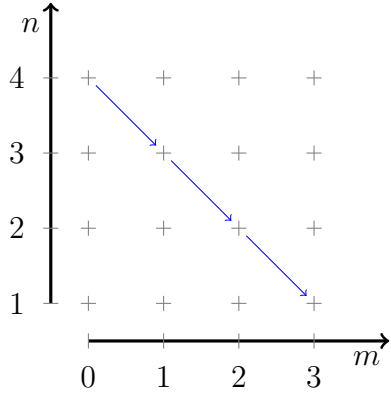
and by merely summing \eqref{mr_wnm} for indices
\[
(n,m) \in \left\{(N,0), (N-1,1), \cdots, (1,N-1)\right\},
\]
we end up with
\begin{equation}\label{wnm_sum}
\begin{aligned}
&\sup_{t \ge 0} t^{2 (\alpha+N) - 3} \sum_{m = 0}^{N-1} \verti{\partial_t^m w^{(N-m)}}_{k(N-m,m,\alpha^\prime) + 2,\alpha^\prime + N-m - \frac 3 2}^2 \\
&+ \int_0^\infty t^{2 (\alpha+N) - 3} \sum_{m = 0}^{N-1} \verti{\partial_t^{m+1} w^{(N-m)}}_{k(N-m,m,\alpha^\prime),\alpha^\prime+N-m-2}^2 \d t \\
&+ \int_0^\infty t^{2 (\alpha+N) - 3} \sum_{m = 0}^{N-1} \verti{\partial_t^m w^{(N-m)}}_{k(N-m,m,\alpha^\prime)+4,\alpha^\prime+N-m-1}^2 \d t \\
& \quad \lesssim \delta_{2 \alpha,1} \delta_{N,1} \verti{w^{(1)}_{| t = 0}}_{k\left(1,0,\frac 1 2 \pm \delta\right) + 2,\pm \delta}^2 + \int_0^\infty t^{2 (\alpha+N) - 3} \sum_{m = 0}^{N-1} \verti{\partial_t^m r^{(N-m)}}_{k(N-m,m,\alpha^\prime),\alpha^\prime+N-m-1}^2 \d t \\
& \qquad + (2 \alpha + 2 N - 3) \int_0^\infty t^{2 (\alpha+N) - 4} \sum_{m = 0}^{N-1} \verti{\partial_t^m w^{(N-m)}}_{k(N-m,m,\alpha^\prime)+2,\alpha^\prime+N-m-\frac 3 2}^2 \d t,
\end{aligned}
\end{equation}
provided that we have
\[
k(N-m-1,m+1,\alpha^\prime) \ge k(N-m,m,\alpha^\prime) + \verti{I_{N-m}} - 1 \quad \mbox{for } \, m = 0, \cdots, N-2
\]
and
\begin{equation}\label{inequ_an}
2 \alpha + 2 N - 3 \ge 0.
\end{equation}
For later purpose we require the stronger assumption
\begin{equation}\label{same_alpha}
k(N-m-1,m+1,\alpha^\prime) \ge k(N-m,m,\alpha^\prime) + \verti{I_{N-m}} \quad \mbox{for } \, m = 0, \cdots, N-2.
\end{equation}
%

\subsubsection{Absorption of remnant terms II}
Estimate~\eqref{wnm_sum} is still insufficient as yet the solution appears on the right-hand side of the estimate (forming the last line). The absorption of this last line is indeed more complicated and demands to specifically choose the weights $\alpha$. In order to understand the choices made, we first make additional comments on elliptic maximal regularity: In fact, one can get control on $u$ from control on the functions $w^{(n)}$.
\begin{proposition}\label{prop:elliptic}
Suppose that $k \in \N_0$, $\varrho \in \R \setminus K_\infty$, and $u: (0,\infty) \to \R$ is smooth satisfying
\begin{equation}\label{ell_assume}
D^\ell u(x) = \sum_{i \in K_\varrho} u_i i^\ell x^i + o(x^\varrho) \quad \mbox{as } \, x \searrow 0 
\end{equation}
for all $\ell = 0, \cdots, k + \verti{K_\varrho}$ (cf.~\eqref{def_kn} for the definition of $K_\varrho$). Then for any polynomial $Q(\zeta) = \prod_{\ell = 1}^m (\zeta - \zeta_\ell)$ with real zeros $\zeta_1, \cdots, \zeta_m$ such that $K_\varrho \subset \{\zeta_1,\cdots,\zeta_m\}$ and $\varrho \notin \{\zeta_1,\cdots,\zeta_m\}$, we have
\begin{equation}\label{elliptic_mr}
\verti{u - \sum_{i \in K_\varrho} u_i x^i}_{k+m,\varrho} \lesssim_{k,\varrho} \verti{Q(D) u}_{k,\varrho}.
\end{equation}
\end{proposition}
A proof in a similar case can be found in \cite[Lem.~7.2]{ggko.2014}. It is an immediate consequence of a version of Hardy's inequality:
\begin{lemma}\label{lem:hardy}
For any $w \in C^\infty((0,\infty))$, $\gamma, \varrho \in \R$ with $\gamma \ne \varrho$, $\verti{w}_{1,\varrho} < \infty$, and $w(x) = o(x^\varrho)$ as $x \searrow 0$ ($x \nearrow \infty$) if $\gamma < \varrho$ ($\gamma > \varrho$), we have
\begin{equation}\label{hardy}
\verti{w}_{1,\varrho} \lesssim_{\gamma,\varrho} \verti{(D-\gamma) w}_\varrho.
\end{equation}
\end{lemma}
For a particular admissible exponent $i \in K_\infty \setminus \{0\}$, we may choose the weights $i \pm \delta$. Both $i + \delta$ and $i-\delta$ appear in the coercivity range of $p(D-\floor{i})$ or $p(D-\floor{i}-1)$ (these are the intervals $(\floor{i}-1,\floor{i})$ and $(\floor{i},\floor{i}+1)$). Having control on
\[
\int_0^\infty t^{2 i - 1} \verti{w^{(\floor{i})}}_{k(\floor{i},0,i-\floor{i}+1-\delta)+4,i - \delta}^2 \d t \quad \mbox{if } \, i \in \N_0
\]
or
\[
\int_0^\infty t^{2 i + 1} \verti{w^{(\floor{i}+1)}}_{k(\floor{i}+1,0,i-\floor{i}\pm\delta)+4,i \pm \delta}^2 \d t \quad \mbox{else},
\]
we obtain control on
\[
\int_0^\infty t^{2 i - 1} \verti{u - \sum_{j \in K_i} u_j x^j}_{k(\floor{i},0,i-\floor{i}+1-\delta)+4+ 4 \floor{i} + \verti{J_{\floor{i}}},i - \delta}^2 \d t \quad \mbox{for } \, i \in \N_0
\]
or else
\[
\int_0^\infty t^{2 i - 1} \verti{u - \sum_{j \in K_i} u_j x^j}_{k(\floor{i}+1,0,i-\floor{i}-\delta)+4+ 4 \left(\floor{i}+1\right) + \verti{J_{\floor{i}+1}},i - \delta}^2 \d t,
\]
and
\[
\int_0^\infty t^{2 i - 1} \verti{u - \sum_{j \in K_i} u_j x^j - u_i x^i}_{k(\floor{i}+1,0,i-\floor{i}+\delta)+4+ 4 \left(\floor{i}+1\right) + \verti{J_{\floor{i}+1}},i + \delta}^2 \d t,
\]
respectively, by using elliptic maximal regularity given by \eqref{elliptic_mr}. It is quite apparent that, by applying the triangle inequality, the last three terms yield control on the coefficient $u_i(t)$ of the form: $\int_0^\infty t^{2 i - 1} \verti{u_i(t)}^2 \d t$. As we have already noted in the introduction, for $\alpha = \frac 1 2$ and $N = 1$ the first term in estimate~\eqref{wnm_sum} also yields control on $\sup_{t \ge 0} \verti{u_0(t)}^2$ and analogous estimates of higher-order coefficients are possible. We will postpone the details on how to extract further control on the coefficients to Section~\ref{ssec:func} (cf.~Lemma~\ref{lem:est_coeff})\footnote{Indeed, sufficiently strong estimates on them are essential for proving appropriate estimates for the nonlinearity (cf.~Proposition~\ref{prop:nonlinear_est}).}. For the moment we just note that for controlling the singular expansion of $u$, it is convenient to use spatial weights $i \pm \delta$ with $i \in K_\infty \setminus \{0\}$ in our norm.

\medskip

Without further ado, we start with the choice of weights by recalling that
\begin{equation}\label{index_sub}
I_{N-1} \subset I_N - 1 = \{i-1: \, i \in I_N\} \quad \mbox{(cf.~Fig.~\ref{fig:power})}.
\end{equation}
We fix $N_0 \in \N$ and are aiming at controlling expansion~\eqref{expansion} up to order $O\left(x^{N_0}\right)$. In view of \eqref{index_sub} it is reasonable to view the set of indices $\AAA$ as a subset of $[0,1] \times \{1,\cdots,N_0\}$. For each $(\alpha,N) \in \AAA$ we may use estimate~\eqref{wnm_sum} where $\alpha^\prime = \alpha \pm \delta \in (0,1)$. We distinguish between two classes of weights:
\begin{itemize}
\item[(a)] By the above considerations on estimating the coefficients and in view of the fact that \eqref{index_sub} holds true, we start by including $(\alpha,N)$ with
\begin{itemize}
\item[$\star$] $\alpha \in \left(I_{N_0} - N_0 + 1\right) \cup \{\beta\}$ and $N = 2, \cdots, N_0$,
\item[$\star$] $\alpha \in \left(\left(I_{N_0} - N_0 + 1\right) \cup \{\beta\}\right) \cap \left(\frac 1 2,1\right)$ and $N = 1$,
\item[$\star$] as well as $(\alpha,N) = (0,N)$ with $N = 2,\cdots,N_0$
\item[$\star$] and $(\alpha,N) = (1,N)$ with $N = 1,\cdots,N_0-1$,
\end{itemize}
in the set $\AAA$. Thus we are already able to control all coefficients $u_i(t)$ with $N_0 > i  > \frac 1 2$. 
\item[(b)] In all these cases, we need to be able to absorb the respective remnant terms forming the last line of \eqref{wnm_sum}. Since the weight in the norm is shifted by $-\frac 1 2$, this requires to include $(\alpha,N)$ with
\begin{itemize}
\item[$\star$] $\alpha \in \left(\left(I_{N_0} - N_0+\frac 1 2\right) \cup \{\beta - \frac 1 2\}\right) \cap \left(0,\frac 1 2\right)$ and $N = 2, \cdots, N_0$,
\item[$\star$] $\alpha \in \left(I_{N_0} - N_0+\frac 3 2\right) \cap \left(\frac 1 2,1\right)$ and $N = 1,\cdots,N_0-1$,
\item[$\star$] as well as $\left(\alpha,N\right) = \left(\frac 1 2,N\right)$ with $N = 1,\cdots,N_0$,
\end{itemize}
in $\AAA$.
\end{itemize}
Now that we have chosen the set of weights $\AAA$, we need to ensure that the absorption mechanism for the last line of \eqref{wnm_sum} works. This requires some additional conditions on the number of derivatives $k(n,m,\alpha^\prime)$. We note that since $\beta$ is irrational, by choosing $\delta > 0$ sufficiently small we have $\alpha^\prime \ne \frac 1 2$ for all $(\alpha,N) \in \AAA$. For $(\alpha,N) \in \AAA$ such that $\alpha + N \ge 2$, we need to distinguish between three cases (cf.~Fig.~\ref{fig:absorb2}):
\begin{itemize}
\item[(a)] If $\alpha^\prime > \frac 1 2$, we can absorb the remnant in the last line of \eqref{wnm_sum} through
\begin{align*}
& \int_0^\infty t^{2 (\alpha+N) - 4} \verti{\partial_t^m w^{(N-m)}}_{k(N-m,m,\alpha^\prime)+2,\alpha^\prime+N-m-\frac 3 2}^2 \d t \\
& \quad \lesssim \int_0^\infty t^{2 \left(\alpha - \frac 1 2 + N\right) - 3} \verti{\partial_t^m w^{(N-m)}}_{k\left(N-m,m,\alpha^\prime-\frac 1 2\right)+4,\alpha^\prime-\frac 1 2 + N - m - 1}^2 \d t,
\end{align*}
where the second line of the inequality appears on the left-hand side of \eqref{wnm_sum} with $\alpha$ replaced by $\alpha - \frac 1 2$. This requires that the indices obey
\begin{equation}\label{alpha-12}
k\left(N-m,m,\alpha^\prime-\frac 1 2\right) \ge k\left(N-m,m,\alpha^\prime\right) - 2 \quad \mbox{if } \, \alpha^\prime \in \left(\frac 1 2, 1\right).
\end{equation}
Indeed one may verify that by construction in all such cases $\left(\alpha-\frac 1 2,N\right) \in \AAA$.
\item[(b)] If $\alpha^\prime < \frac 1 2$ and $N-m \ge 2$, we can absorb the remnant term in \eqref{wnm_sum} through
\begin{align*}
& \int_0^\infty t^{2 (\alpha+N) - 4} \verti{\partial_t^m w^{(N-m)}}_{k(N-m,m,\alpha^\prime)+2,\alpha^\prime+N-m-\frac 3 2}^2 \d t \\
& \quad \lesssim \int_0^\infty t^{2 \left((\alpha + \frac 1 2) + (N - 1)\right) - 3} \verti{\partial_t^m w^{((N-1)-m)}}_{k\left((N-1)-m,m,\alpha^\prime+\frac 1 2\right)+4,\alpha^\prime+\frac 1 2 + (N-1) - m - 1}^2 \d t,
\end{align*}
where the last line of the estimate appears on the left-hand side of \eqref{wnm_sum} with $\alpha$ replaced by $\alpha + \frac 1 2$ and $N$ replaced by $N-1$. In view of \eqref{def_vwr}, this requires the constraint
\begin{equation}\label{alpha+12}
k\left(N-m-1,m,\alpha^\prime+\frac 1 2\right) \ge k\left(N-m,m,\alpha^\prime\right) + 2 + \verti{I_{N-m}} \quad \mbox{if } \, \alpha^\prime \in \left(0, \frac 1 2\right).
\end{equation}
Again, by construction we have $\left(\alpha+\frac 1 2, N-1\right) \in \AAA$.
\item[(c)] If $\alpha^\prime < \frac 1 2$ and $N-m = 1$, necessarily $m \ge 1$ and we can estimate the remnant in \eqref{wnm_sum} by
\begin{align*}
& \int_0^\infty t^{2 (\alpha+N) - 4} \verti{\partial_t^m w^{(1)}}_{k(1,m,\alpha^\prime)+2,\alpha^\prime-\frac 1 2}^2 \d t \\
& \quad \lesssim \int_0^\infty t^{2 \left((\alpha + \frac 1 2) + N - 1\right) - 3} \verti{\partial_t^{(m-1)+1} w^{(1)}}_{k\left(1,m-1,\alpha^\prime+\frac 1 2\right),\alpha^\prime + \frac 1 2 - 1}^2 \d t.
\end{align*}
Here, the second line in the estimate is controlled in \eqref{wnm_sum} with $\alpha + \frac 1 2$ instead of $\alpha$ and $m$ replaced by $m-1$. Yet, the absorption only works if the indices obey
\begin{equation}\label{alpha+12_alt}
k\left(1,m-1,\alpha^\prime+\frac 1 2\right) \ge k\left(1,m,\alpha^\prime\right) + 2 \quad \mbox{if } \, \alpha^\prime \in \left(0, \frac 1 2\right).
\end{equation}
\end{itemize}
%

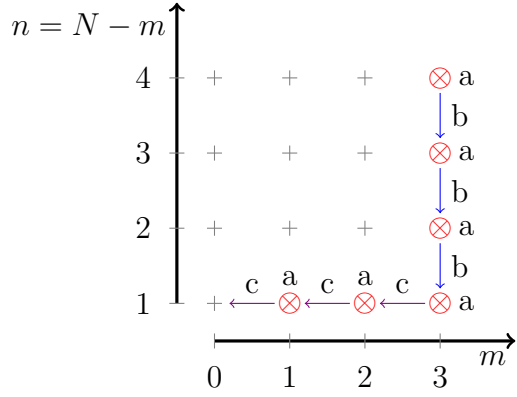
\begin{figure}[htp]
\centering
\begin{tikzpicture}[scale=1]
\draw[very thick,->] (.5,1) -- (.5,5);
\draw[very thick,->] (1,.5) -- (5,.5);
\draw [gray] (.4,1) -- (.6,1);
\draw [gray] (1,.4) -- (1,.6);
\draw [gray] (.4,2) -- (.6,2);
\draw [gray] (2,.4) -- (2,.6);
\draw [gray] (.4,3) -- (.6,3);
\draw [gray] (3,.4) -- (3,.6);
\draw [gray] (.4,4) -- (.6,4);
\draw [gray] (4,.4) -- (4,.6);
\draw (.3,1) node[anchor=east] {$1$};
\draw (1,.3) node[anchor=north] {$0$};
\draw (.3,2) node[anchor=east] {$2$};
\draw (2,.3) node[anchor=north] {$1$};
\draw (.3,3) node[anchor=east] {$3$};
\draw (3,.3) node[anchor=north] {$2$};
\draw (.3,4) node[anchor=east] {$4$};
\draw (4,.3) node[anchor=north] {$3$};
\draw [gray] (.9,1) -- (1.1,1);
\draw [gray] (1,.9) -- (1,1.1);
\draw [gray] (.9,2) -- (1.1,2);
\draw [gray] (1,1.9) -- (1,2.1);
\draw [gray] (.9,3) -- (1.1,3);
\draw [gray] (1,2.9) -- (1,3.1);
\draw [gray] (.9,4) -- (1.1,4);
\draw [gray] (1,3.9) -- (1,4.1);
\draw [red] (2,1) node {$\otimes$};
\draw [gray] (1.9,2) -- (2.1,2);
\draw [gray] (2,1.9) -- (2,2.1);
\draw [gray] (1.9,3) -- (2.1,3);
\draw [gray] (2,2.9) -- (2,3.1);
\draw [gray] (1.9,4) -- (2.1,4);
\draw [gray] (2,3.9) -- (2,4.1);
\draw [red] (3,1) node {$\otimes$};
\draw [gray] (2.9,2) -- (3.1,2);
\draw [gray] (3,1.9) -- (3,2.1);
\draw [gray] (2.9,3) -- (3.1,3);
\draw [gray] (3,2.9) -- (3,3.1);
\draw [gray] (2.9,4) -- (3.1,4);
\draw [gray] (3,3.9) -- (3,4.1);
\draw [red] (4,1) node {$\otimes$};
\draw [red] (4,2) node {$\otimes$};
\draw [red] (4,3) node {$\otimes$};
\draw [red] (4,4) node {$\otimes$};
\draw[blue,->] (4,3.8) -- (4,3.2);
\draw[blue,->] (4,2.8) -- (4,2.2);
\draw[blue,->] (4,1.8) -- (4,1.2);
\draw[violet,->] (3.8,1) -- (3.2,1);
\draw[violet,->] (2.8,1) -- (2.2,1);
\draw[violet,->] (1.8,1) -- (1.2,1);
\draw (4.1,4) node[anchor=west] {a};
\draw (4,3.5) node[anchor=west] {b};
\draw (4.1,3) node[anchor=west] {a};
\draw (4,2.5) node[anchor=west] {b};
\draw (4.1,2) node[anchor=west] {a};
\draw (4,1.5) node[anchor=west] {b};
\draw (4.1,1) node[anchor=west] {a};
\draw (3.5,1) node[anchor=south] {c};
\draw (3,1.1) node[anchor=south] {a};
\draw (2.5,1) node[anchor=south] {c};
\draw (2,1.1) node[anchor=south] {a};
\draw (1.5,1) node[anchor=south] {c};
\draw (.5,4.7) node[anchor=east] {$n =  N-m$};
\draw (4.7,.5) node[anchor=north] {$m$};
\end{tikzpicture}
\caption{Schematic: absorption of remnant terms II. Each node $+$ corresponds to an estimate of the form \eqref{wnm_sum} (with $n = N-m$). The displayed arrows visualize the absorption mechanism starting with $N = 7$, $m = 3$, and $\alpha^\prime > \frac 1 2$. The symbol $\otimes$ denotes a shift of $\alpha$ by $- \frac 1 2$ (keeping $n$ and $m$ fixed). The remnant at the base of the arrow (forming the last line in \eqref{wnm_sum}) is absorbed by the corresponding estimate \eqref{wnm_sum} at the tip of the arrow under the assumption that \eqref{alpha-12}, \eqref{alpha+12}, and \eqref{alpha+12_alt}, respectively, holds true.}
\label{fig:absorb2}
\end{figure}
The above argumentation shows that we can restrict our considerations to the cases in which $\alpha \in \left(\frac 1 2, 1\right)$, $N = 1$, and $m = 0$ since in the case $\alpha = \frac 1 2$, the remnant term in \eqref{wnm_sum} disappears. The remaining terms can be treated by applying an anisotropic version of Hardy's inequality:
\begin{lemma}\label{lem:anisotropic}
Suppose $v: \, (0,\infty)^2 \to \R$ is smooth, $\ell \in \N_0$, and $\alpha \in \left(\frac 1 2, 1\right)$. Then
\begin{equation}\label{anisotropic}
\int_0^\infty t^{2 \alpha - 2} \verti{v}_{\ell,\alpha-\frac 1 2\pm \delta}^2 \d t \lesssim \int_0^\infty \verti{\partial_t v}_{\ell,-\frac 1 2\pm \delta}^2 \d t + \int_0^\infty \verti{v}_{\ell+1,\frac 1 2\pm \delta}^2
\end{equation}
\end{lemma}
A proof can be found in \cite[Lem.~7.5]{ggko.2014}. The remnant in \eqref{wnm_sum} for $(N,m) = (1,0)$ is of the form
\[
\int_0^\infty t^{2 \alpha - 2} \verti{w^{(1)}}_{k(1,0,\alpha^\prime)+2,\alpha^\prime-\frac 1 2}^2 \d t
\]
and can be absorbed by estimate~\eqref{wnm_sum} with $\alpha = \frac 1 2$ and $N = 1$, that is,

\begin{equation}\label{aniso_appl}
\int_0^\infty t^{2 \alpha - 2} \verti{w^{(1)}}_{k\left(1,0,\alpha^\prime\right)+2,\alpha^\prime - \frac 1 2}^2 \, \d t \lesssim \int_0^\infty \verti{\partial_t w^{(1)}}_{k,-\frac 1 2 \pm \delta}^2 \, \d t + \int_0^\infty \verti{w^{(1)}}_{k+4,\frac 1 2 \pm \delta}^2 \, \d t,
\end{equation}
where we write $k := k\left(1,0,\frac 1 2 \pm \delta\right)$ (assuming that the values for $+$ and $-$ coincide), provided we have
\begin{equation}\label{alpha_12_1}
k \ge k\left(1,0,\alpha^\prime\right) + 2 \quad \mbox{if } \, \alpha^\prime \in \left(\frac 1 2, 1\right) \, \mbox{ and } \, \alpha \ne \frac 1 2.
\end{equation}

\bigskip

By summing over all estimates~\eqref{wnm_sum} with $(\alpha,N) \in \AAA$ and $\alpha^\prime = \alpha \pm \delta \in (0,1)$, and fulfilling conditions~\eqref{same_alpha}, \eqref{alpha-12}, \eqref{alpha+12}, \eqref{alpha+12_alt}, and \eqref{alpha_12_1}, we obtain
\begin{equation}\label{mr_star}
\vertiii{u}_* \lesssim \vertiii{u^{(0)}}_{*,0} + \vertiii{f}_{*,1},
\end{equation}
where
\begin{equation}\label{norm_0star}
\vertiii{u^{(0)}}_{*,0}^2 := \verti{w^{(0,0)}}_{k+2,\frac 1 2 - \delta}^2 + \verti{w^{(0,0)}}_{k+2,\frac 1 2 + \delta}^2 \quad \mbox{with } \, w^{(0,0)} := p(D) u^{(0)}
\end{equation}
is the norm for the initial data,
\begin{equation}\label{norm_sol_star}
\begin{aligned}
\vertiii{u}_*^2 := \, & \sum_{\substack{(\alpha,N) \in \AAA \\ \alpha^\prime = \alpha \pm \delta \in (0,1)}} \sup_{t \ge 0} t^{2 (\alpha+N) - 3} \sum_{m = 0}^{N-1} \verti{\partial_t^m w^{(N-m)}}_{k(N-m,m,\alpha^\prime) + 2,\alpha^\prime + N-m - \frac 3 2}^2 \\
&+ \sum_{\substack{(\alpha,N) \in \AAA \\ \alpha^\prime = \alpha \pm \delta \in (0,1)}} \int_0^\infty t^{2 (\alpha+N) - 3} \sum_{m = 0}^{N-1} \verti{\partial_t^{m+1} w^{(N-m)}}_{k(N-m,m,\alpha^\prime),\alpha^\prime+N-m-2}^2 \d t \\
&+ \sum_{\substack{(\alpha,N) \in \AAA \\ \alpha^\prime = \alpha \pm \delta \in (0,1)}} \int_0^\infty t^{2 (\alpha+N) - 3} \sum_{m = 0}^{N-1} \verti{\partial_t^m w^{(N-m)}}_{k(N-m,m,\alpha^\prime)+4,\alpha^\prime+N-m-1}^2 \d t
\end{aligned}
\end{equation}
is the norm for the solution $u$ (with $w^{(n)}$ defined in \eqref{def_wn}), and
\begin{equation}\label{norm_rhs_star}
\vertiii{f}_{*,1}^2 := \sum_{\substack{(\alpha,N) \in \AAA \\ \alpha^\prime = \alpha \pm \delta \in (0,1)}} \int_0^\infty t^{2 (\alpha+N) - 3} \sum_{m = 0}^{N-1} \verti{\partial_t^m r^{(N-m)}}_{k(N-m,m,\alpha^\prime),\alpha^\prime+N-m-1}^2 \d t
\end{equation}
is the norm for the right-hand side $f$ (with $r^{(n)}$ defined in \eqref{def_rn}).

\medskip

Now we make a further assumption, that is, we assume
\begin{equation}\label{alpha_equiv}
\begin{aligned}
& k(n,m,\alpha^\prime) \, \mbox{ is constant for } \, \alpha^\prime \in \left(0, \frac 1 2\right), \, \mbox{ and } \, \alpha^\prime \in \left(\frac 1 2, 1\right) \, \mbox{ respectively}, \\
&\mbox{except for half integers } \, \alpha \in \left\{0,\frac 1 2, 1\right\}.
\end{aligned}
\end{equation}
One may verify that thus still conditions~\eqref{same_alpha}, \eqref{alpha-12}, \eqref{alpha+12}, \eqref{alpha+12_alt}, and \eqref{alpha_12_1} can be satisfied (see below).

\subsubsection{Maximal regularity for the linear equation}
Further applying elliptic maximal regularity given by Proposition~\ref{prop:elliptic}, we infer that the norms $\vertiii{\cdot}_0$ and $\vertiii{\cdot}_{*,0}$, $\vertiii{\cdot}$ and $\vertiii{\cdot}_*$, as well as $\vertiii{\cdot}_1$ and $\vertiii{\cdot}_{*,1}$, respectively, are equivalent, where $\vertiii{\cdot}_0$ is given by \eqref{norm_initial}, i.e.,
\[
\vertiii{u^{(0)}}_0^2 := \verti{u^{(0)}}_{k+6,-\delta}^2 + \verti{u^{(0)} - u^{(0)}_0}_{k+6,\delta}^2,
\]
the norm $\vertiii{\cdot}$ for the solution $u$ is given by
\begin{equation}\label{norm_sol}
\begin{aligned}
\vertiii{u}^2 := \, & \sum_{\substack{(\alpha,N) \in \AAA \\ \alpha^\prime = \alpha \pm \delta \in (0,1)}} \sup_{t \ge 0} t^{2 (\alpha+N) - 3} \sum_{m = 0}^{N-1} \verti{\partial_t^m u - \sum_{i < \alpha^\prime + N-m - \frac 3 2} \frac{\d^m u_i}{\d t^m} x^i}_{\ell(N-m,m,\alpha^\prime) + 2,\alpha^\prime + N-m - \frac 3 2}^2 \\
&+ \sum_{\substack{(\alpha,N) \in \AAA \\ \alpha^\prime = \alpha \pm \delta \in (0,1)}} \int_0^\infty t^{2 (\alpha+N) - 3} \sum_{m = 0}^{N-1} \verti{\partial_t^{m+1} u - \sum_{i < \alpha^\prime + N-m - 2} \frac{\d^{m+1} u_i}{\d t^{m+1}} x^i}_{\ell(N-m,m,\alpha^\prime),\alpha^\prime+N-m-2}^2 \d t \\
&+ \sum_{\substack{(\alpha,N) \in \AAA \\ \alpha^\prime = \alpha \pm \delta \in (0,1)}} \int_0^\infty t^{2 (\alpha+N) - 3} \sum_{m = 0}^{N-1} \verti{\partial_t^m u - \sum_{i < \alpha^\prime + N-m - 1} \frac{\d^m u_i}{\d t^m} x^i}_{\ell(N-m,m,\alpha^\prime)+4,\alpha^\prime+N-m-1}^2 \d t
\end{aligned}
\end{equation}
with (cf.~\eqref{def_jnin} and \eqref{def_vwr})
\begin{equation}\label{def_lnm}
\ell(n,m,\alpha^\prime) := k(n,m,\alpha^\prime) + \verti{J_n} + 4 n,
\end{equation}
and the norm $\vertiii{\cdot}_1$ for the right-hand side $f$ reads
\begin{equation}\label{norm_rhs}
\vertiii{f}_1^2 := \sum_{\substack{(\alpha,N) \in \AAA \\ \alpha^\prime = \alpha \pm \delta \in (0,1)}} \int_0^\infty t^{2 (\alpha+N) - 3} \sum_{m = 0}^{N-1} \verti{\partial_t^m f - \sum_{\beta < i < \alpha^\prime + N-m - 1} \frac{\d^m f_i}{\d t^m} x^i}_{\ell(N-m,m,\alpha^\prime),\alpha^\prime+N-m-1}^2 \d t.
\end{equation}
Consequently, \eqref{mr_star} turns into the maximal-regularity estimate
\begin{equation}\label{mr_main}
\vertiii{u} \lesssim \vertiii{u^{(0)}}_0 + \vertiii{f}_1.
\end{equation}
For convenience, we summarize the conditions on the numbers $\ell(n,m,\alpha^\prime)$: $\ell(n,m,\alpha^\prime)$ is constant for all $\alpha^\prime \in \left(0,\frac 1 2\right)$ and $\alpha^\prime \in \left(\frac 1 2, 1\right)$, respectively, except for $\alpha \in \left\{0,\frac 1 2, 1\right\}$. Furthermore, the following inequalities (through \eqref{def_lnm} equivalent to \eqref{same_alpha}, \eqref{alpha-12}, \eqref{alpha+12}, \eqref{alpha+12_alt}, and \eqref{alpha_12_1}) must hold:
\begin{subequations}\label{linear_cond}
\begin{align}
\ell(N-m-1,m+1,\alpha^\prime) &\ge \ell(N-m,m,\alpha^\prime) - 4 \quad \mbox{for } \, N-m \ge 2,\label{linear_cond1}\\
\ell\left(N-m,m,\alpha^\prime-\frac 1 2\right) &\ge \ell\left(N-m,m,\alpha^\prime\right) - 2 \quad \mbox{if } \, \alpha^\prime \in \left(\frac 1 2, 1\right) \, \mbox{ and } \, \alpha + N \ge 2,\label{linear_cond2}\\
\ell\left(N-m-1,m,\alpha^\prime+\frac 1 2\right) &\ge \ell\left(N-m,m,\alpha^\prime\right) - 2 \quad \mbox{if } \, \alpha^\prime \in \left(0, \frac 1 2\right), \; N-m \ge 2,\label{linear_cond3}\\
\ell\left(1,m-1,\alpha^\prime+\frac 1 2\right) &\ge \ell\left(1,m,\alpha^\prime\right) + 2 \quad \mbox{if } \, \alpha^\prime \in \left(0, \frac 1 2\right) \, \mbox{ and } \, m \ge 1,\label{linear_cond3_alt}\\
k &\ge \ell\left(1,0,\alpha^\prime\right) - 2 \quad \mbox{if } \, \alpha^\prime \in \left(\frac 1 2, 1\right) \, \mbox{ and } \, \alpha \ne \frac 1 2.\label{linear_cond4}
\end{align}
\end{subequations}
It is apparent that conditions~\eqref{linear_cond} can be fulfilled and non-negativity of $k\left(n,m,\alpha^\prime\right)$ can be ensured (cf.~\eqref{def_lnm}) if we explicitly choose (cf.~\eqref{def_jn} for the definition of $J_n$)
\begin{subequations}\label{explicit}
\begin{align}
\ell\left(n,m,\alpha^\prime\right) &:= 8 N_0 + \verti{J_{N_0}} - 2 \floor{2 \left(n+m+\alpha^\prime\right)} \quad \mbox{for } \, \alpha \notin \left\{0,\frac 1 2,1\right\}, \\
\ell\left(n,m,\alpha^\prime\right) &:= 8 N_0 + \verti{J_{N_0}} + 3 - 4 \left(n+m+\alpha\right) \quad \mbox{for } \, \alpha \in \left\{0,\frac 1 2,1\right\}, \;\, (n,m,\alpha) \ne \left(1,0,\frac 1 2\right), \\
k &:= 8 N_0 + \verti{J_{N_0}} - 5.
\end{align}
\end{subequations}
This choice is also compatible with the ``nonlinear" conditions \eqref{cond_l_non}, which are necessary for the treatment of the nonlinearity $\NN(u)$ in Section~\ref{sec:nonlinear}.
\subsection{Properties of the parabolic norms and definition of function spaces\label{ssec:func}}
In this subsection we summarize some of the properties of the parabolic norms $\vertiii{\cdot}$, $\vertiii{\cdot}_0$, and $\vertiii{\cdot}_1$ (cf.~\eqref{norm_initial}, \eqref{norm_sol}, \eqref{norm_rhs}).

\begin{lemma}\label{lem:est_coeff}
For given $N_0 \in \N$ and locally integrable $u, f: \, (0,\infty)^2 \to \R$ such that the generalized power series \eqref{expansion} exists to order $O(x^{N_0})$, the following estimates (with constants independent of $f$ and $u$) hold true:
\begin{subequations}\label{est_coeff}
\begin{align}
\int_0^\infty t^{2 i + 2 m - 1} \verti{\frac{\d^m u_i}{\d t^m}(t)}^2 \d t &\lesssim \vertiii{u}^2 \quad \mbox{for } \, i \in K_{N_0-m} \setminus \{0\} \, \mbox{ and } \, m \in \N_0, \label{est_coeff1}\\
\sup_{t \ge 0} t^{2 i + 2 m} \verti{\frac{\d^m u_i}{\d t^m}(t)}^2 &\lesssim \vertiii{u}^2 \quad \mbox{for } \, i \in K_{N_0-m-\frac 1 2} \, \mbox{ and } \, m \in \N_0, \label{est_coeff2}\\
\int_0^\infty t^{2 i + 2 m - 1} \verti{\frac{\d^m f_i}{\d t^m}(t)}^2 \d t &\lesssim \vertiii{f}_1^2 \quad \mbox{for } \, i \in K_{N_0-m} \setminus \{0,\beta\} \, \mbox{ and } \, m \in \N_0. \label{est_coeff3}
\end{align}
\end{subequations}
Furthermore, for any locally integrable $u^{(0)}: \, (0,\infty) \to \R$ such that $u^{(0)}_0 = \lim_{x \searrow 0} u^{(0)}(x)$ exists, we have $\verti{u^{(0)}_0} \lesssim \vertiii{u^{(0)}}_0$ (where the constant is independent of $u^{(0)}$).
\end{lemma}
\begin{proof}
The estimate for $u^{(0)}$ has already been proven in \cite[Lem.~4.3~(a)]{ggko.2014}. Estimates~\eqref{est_coeff} for the coefficients follow quite elementarily by the same reasoning as in the proof of \cite[Lem.~4.3~(b)]{ggko.2014}:

\medskip

For estimate~\eqref{est_coeff1} we take $(\alpha,N) \in \AAA$ with $\alpha + N - m - 1 = i$ and $i \notin \N_0$ and obtain\footnote{The assumption $i \notin \N_0$ is merely for notational simplicity. The reader may verify that the reasoning works in the same way for $i \in \N_0$.}:
\begin{align*}
\verti{\frac{\d^m u_i}{\d t^m}}^2 \lesssim \, & \int_{\frac 1 2}^2 \verti{\frac{\d^m u_i}{\d t^m}}^2 \d x \\
\lesssim \, &  \int_{\frac 1 2}^2 \left(\verti{\partial_t^m u - \sum_{j < \alpha + N - m - 1 - \delta} \frac{\d^m u_j}{\d t^m} x^j}^2 + \verti{\partial_t^m u - \sum_{j < \alpha + N - m - 1 - \delta} \frac{\d^m u_j}{\d t^m} x^j - \frac{\d^m u_i}{\d t^m} x^i}^2\right) \d x \\
\lesssim \, &  \verti{\partial_t^m u - \sum_{j < \alpha + N - m - 1 - \delta} \frac{\d^m u_j}{\d t^m} x^j}_{\ell(N-m,m,\alpha^\prime)+4,\alpha+N-m-1-\delta}^2 \\
&+ \verti{\partial_t^m u - \sum_{j < \alpha + N - m - 1 + \delta} \frac{\d^m u_j}{\d t^m} x^j}_{\ell(N-m,m,\alpha^\prime)+4,\alpha+N-m-1+\delta}^2
\end{align*}
Multiplying with the time weight $t^{2 (\alpha+N)-3}$ and integrating in time, we obtain \eqref{est_coeff1} (cf.~\eqref{norm_sol}). For proving estimate~\eqref{est_coeff2}, we take $(\alpha,N) \in \AAA$ with $\alpha + N - m - \frac 3 2 = i$ and the same reasoning as above (taking the $\sup$ in time instead of integrating) leads to estimate~\eqref{est_coeff2}. The proof of \eqref{est_coeff3} is the same as for \eqref{est_coeff1}.
\end{proof}
Estimates~\eqref{est_coeff} will turn out to be relevant for estimating the nonlinearity. However, they are also convenient in order to define appropriate spaces for our solution and the right-hand side. As in the case of standard Sobolev spaces there are two possible approaches:

\medskip

On the one hand, one may define for a locally integrable function $u: (0,\infty)^2 \to \R$ its distributional derivatives and take the infimum over all possible coefficients $u_i : (0,\infty) \to \R$ in the definition of the norm $\vertiii{u}$ in \eqref{norm_sol}. It is apparent from \eqref{norm_sol} that if $\vertiii{u} < \infty$, the coefficients $u_i$ are uniquely defined almost everywhere. This corresponds to the definition of the $W$-scale in the standard theory of Sobolev spaces.

\medskip

On the other hand, it was shown in \cite[Lem.~B.4]{ggko.2014} for $N_0 = 1$ that one can approximate any locally integrable $u: (0,\infty)^2 \to \R$ with $\vertiii{u} < \infty$ by a sequence $\left(u^{(\nu)}: (0,\infty)^2 \to \R\right)_{\nu \in \N}$ of more regular functions $u^{(\nu)}$ fulfilling:
\begin{itemize}
\item[(a)] $u^{(\nu)} \in C^\infty((0,\infty)^2) \cap C^0_0([0,\infty)^2)$;
\item[(b)] for every $t \in [0,\infty)$ we have $u^{(\nu)}(t,x) = u^{(\nu)}_0(t) + u^{(\nu)}_\beta(t) x^\beta$ for $x \ll_\nu 1$, where $u^{(\nu)}_0, u^{(\nu)}_{\beta}: (0,\infty) \to \R$ are smooth;
\item[(c)] $\vertiii{u-u^{(\nu)}} \to 0$ as $\nu \to \infty$.
\end{itemize}
Taking the closure of all $u: (0,\infty)^2 \to \R$ with (a), (b), and (c) with respect to $\vertiii{\cdot}$ (with $N_0 = 1$), we end up with the analogue of what one commonly refers to as the $H$-scale of Sobolev spaces. Lemma~B.4 of reference \cite{ggko.2014} then states that for $N_0 = 1$ indeed $H = W$ holds\footnote{This is also true for the norm $\vertiii{\cdot}_0$ as shown in \cite[Lem.~B.3]{ggko.2014}.}.

\medskip

Unlike in \cite{ggko.2014}, where it turned out to be more convenient to rely on the $W$-approach of Sobolev spaces, we will employ the $H$-approach in what follows:
\begin{definition}\label{def:spaces}
Suppose $N_0 \in \N$ and $\delta > 0$ is chosen sufficiently small. Furthermore suppose that conditions~\eqref{linear_cond} are fulfilled (cf.~\eqref{explicit} for a specific choice).
\begin{itemize}
\item[(a)] We define the space $U_0$ of initial data $u^{(0)}$ as the closure of all $u^{(0)} \in C^\infty((0,\infty)) \cap C^0_0([0,\infty))$ with $u^{(0)}(x) = u^{(0)}_0 + u^{(0)}_\beta x^\beta$ for $x \ll 1$ with respect to $\vertiii{\cdot}_0$ (cf.~\eqref{norm_initial}).
\item[(b)] The solution space $U$ is defined as the closure with respect to $\vertiii{\cdot}$ (cf.~\eqref{norm_sol}) of all $u \in C^\infty\left((0,\infty)^2\right) \cap C^0_0\left([0,\infty)^2\right)$ such that $u(t,x) = \sum_{i \in K_{N_0}} u_i(t) x^i$ for $x \ll 1$, where the $u_i: (0,\infty) \to \R$ are smooth functions of time.
\item[(c)] The space $F$ of right-hand sides $f$ is defined as the closure with respect to $\vertiii{\cdot}_1$ (cf.~\eqref{norm_rhs}) of all $f \in C^\infty\left((0,\infty)^2\right) \cap C^0_0\left([0,\infty)^2\right)$ such that $f(t,x) = \sum_{i \in K_{N_0} \setminus \{0,\beta\}} f_i(t) x^i$ for $x \ll 1$, where the $f_i: (0,\infty) \to \R$ are smooth.
\end{itemize}
\end{definition}
We mark that in view of Lemma~\ref{lem:est_coeff} the coefficients $u_i$, $u_0^{(0)}$, and $f_i$ of the generalized power series are also defined (at least almost everywhere in time $t$) for functions $u$, $u^{(0)}$, and $f$, respectively, for which $\vertiii{u}$, $\vertiii{u^{(0)}}_0$, and $\vertiii{f}_1$, respectively, is finite.

\medskip

Next we also provide another estimate that in particular guarantees control of the norm $\sup_{t \ge 0} \vertii{u}$, where
\begin{equation}\label{def_c0}
\vertii{u} := \sup_{x > 0} \verti{u(x)}
\end{equation}
denotes the $\infty$-norm in space $x$. By approximation this also implies continuity of $u$ and derivatives if $u \in U$.
\begin{lemma}\label{lem:c0control}
Suppose that $N_0 \in \N$, $m \in \{0,\cdots,N_0-1\}$, and $\ell \in \left\{0,\cdots,\ell\left(1,m,\frac 1 2\pm\delta\right)+1\right\}$. Then we have
\begin{equation}\label{est_c0norm}
\sup_{t \ge 0} t^{2 m} \vertii{\partial_t^m D^\ell u}^2 + \sup_{t \ge 0} t^{2 m} \vertii{\partial_t^m D^\ell (u-u_0)}^2 \lesssim \vertiii{u}^2 \quad \mbox{for all } \, u \in U,
\end{equation}
where the constant in the estimate is independent of $u$.
\end{lemma}
\begin{proof}
A proof for an analogous estimate is contained in \cite[est.~(8.5)]{ggko.2014} so that we only sketch the arguments here. First we may show that, passing to the logarithmic variable $s := \ln x$ and using a cut-off argument in combination with the standard embedding $H^1(\R) \hookrightarrow C^0(\R)$, the following estimates hold
\begin{equation}\label{est_interval}
\vertii{v-v_0}_{(0,1]} \lesssim \verti{v-v_0}_{1,\delta} \quad \mbox{and} \quad \vertii{v}_{[1,\infty)} \lesssim \verti{v}_{1,-\delta} \quad \mbox{for any locally integrable } \, v,
\end{equation}
where the constants are independent of $\delta$ and $\vertii{v}_A := \sup_{x \in A} \verti{v(x)}$ for any set $A \subset (0,\infty)$. estimate~\eqref{est_interval} in combination with Lemma~\ref{lem:est_coeff} shows
\begin{eqnarray*}
\sup_{t \ge 0} t^{2 m} \vertii{\partial_t^m D^\ell (u-u_0)}^2 &\lesssim& \sup_{t \ge 0} t^{2 m} \vertii{\partial_t^m D^\ell u - \delta_{\ell,0} \frac{\d^m u_0}{\d t^m}}_{(0,1]}^2 + \sup_{t \ge 0} t^{2 m} \verti{\frac{\d^m u_0}{\d t^m}}^2 \\
&& + \sup_{t \ge 0} t^{2 m} \vertii{\partial_t^m D^\ell u}_{[1,\infty)}^2 \\
&\stackrel{\eqref{est_coeff2}, \,\eqref{est_interval}}{\lesssim}& \sup_{t \ge 0} t^{2 m} \left(\verti{\partial_t^m u - \frac{\d^m u_0}{\d t^m}}_{\ell+1,\delta}^2 + \verti{\partial_t^m u}_{\ell+1,-\delta}^2\right) + \vertiii{u}^2\\
&\stackrel{\eqref{norm_sol}}{\lesssim}& \vertiii{u}^2.
\end{eqnarray*}
By the triangle inequality and again using estimate~\eqref{est_coeff2}, we obtain \eqref{est_c0norm}.
\end{proof}
\subsection{Rigorous treatment of the linear equation\label{ssec:rigorous}}
In this section we prove our main result for the linear equation:
\begin{proposition}\label{prop:main_lin}
Suppose $N_0 \in \N$ and $\delta > 0$ is chosen sufficiently small. Furthermore, suppose that conditions~\eqref{linear_cond} are fulfilled (cf.~\eqref{explicit} for explicitly chosen indices). Then for any $f \in F$ and $u^{(0)} \in U_0$ there exists exactly one solution $u = S\left[u^{(0)},f\right] \in U$ of the linear degenerate-parabolic problem \eqref{lin_cauchy}. This solution fulfills the maximal-regularity estimate \eqref{mr_main}.
\end{proposition}

\begin{proof}[of Proposition~\ref{prop:main_lin}]
The statement is the generalization of \cite[Prop.~7.6]{ggko.2014} for $N_0 = 1$ to arbitrary $N_0 \in \N$. Since our Banach spaces $U$ and $F$ are nested for increasing $N_0$, uniqueness is already clear and it remains to prove existence. As Proposition~\ref{prop:main_lin} for $N_0 = 1$ is already proven, there exists a unique solution $u$ of \eqref{lin_cauchy} lying in the space $U$ for $N_0 = 1$. By approximation (cf.~Definition~\ref{def:spaces}), we may without loss of generality assume that $f \in C^\infty\left((0,\infty)^2\right) \cap C_0^0\left([0,\infty)\right)$ and $u^{(0)} \in C^\infty\left((0,\infty)\right) \cap C_0^0\left([0,\infty)\right)$, with the expansions $u^{(0)}(x) = u^{(0)}_0 + u^{(0)}_\beta x^\beta$ for $x \ll 1$, and $f(t,x) = \sum_{i \in K_{N_0} \setminus \{0,\beta\}} f_i(t) x^i$ for $x \ll 1$, where the $f_i: (0,\infty) \to \R$ are smooth functions of time. By standard parabolic theory in the bulk and using \cite[Prop.~7.6]{ggko.2014} also $u \in C^\infty\left((0,\infty)^2\right) \cap C^0_0\left([0,\infty)^2\right)$ with $u(t,x) = u_0(t) + u_\beta(t) x^\beta + o\left(x^\beta\right)$ as $x \searrow 0$.

\subsubsection*{A qualitative argument for regularity}
Here we argue why the solution indeed has additional spatial regularity. As we have observed before, this requires higher regularity in time. At the basis of our reasoning is the existence and uniqueness result for the linear problem \eqref{lin_cauchy} given by \cite[Prop.~7.6]{ggko.2014}. In particular we have finiteness of
\[
\int_0^\infty t^{2 \beta - 1} \verti{\partial_t w^{(1)}}_{k(1,0,\beta\pm\delta),\beta - 1 \pm \delta}^2 \d t \stackrel{\eqref{def_wn}}{\sim} \int_0^\infty t^{2 \beta - 1} \verti{\partial_t u}_{\ell(1,0,\beta\pm\delta),\beta - 1 \pm \delta}^2 \d t
\]
and thus we may use finiteness of the norms $\verti{\partial_t u}_{\ell(1,0,\beta\pm\delta),\beta - 1 \pm \delta}^2$ for some time $t = \tau > 0$ to solve
\[
x \partial_t (\partial_t u) + p(D) (\partial_t u) = \partial_t f \quad \mbox{for } \, t > \tau \, \mbox{ and } \, x > 0
\]
and infer that $\partial_t u$ has additional regularity for times $t > \tau$. In principle such a reasoning is possible but as it was noted already in \cite[Sec.~2]{ggko.2014}, the arguments there\footnote{We refer to the discussion in \cite[Sec.~2, Eqs.~(2.8)--(2.12)]{ggko.2014} and the (rigorous) time discretization in \cite[Prop.~7.6]{ggko.2014}. Utilizing the linear equations \eqref{lin_v1} and \eqref{lin_v2} (where $v^{(1)} \stackrel{\eqref{def_vwr}}{=} w^{(1)}$ and the difference between $v^{(2)}$ and $w^{(2)}$ is immaterial due to the choice of the weight $1 + \delta < 2 \beta$), the integrals in \eqref{prt_w12} are finite through control of the spatial integrals $\int_0^\infty t \left(\verti{v^{(1)}}_{k(1,0,1-\delta)+4,1-\delta}^2 + \verti{g^{(1)}}_{k(1,0,1-\delta),1-\delta}^2\right) \d t$ and $\int_0^\infty t \left(\verti{w^{(2)}}_{k(2,0,\delta)+4,1+\delta}^2 + \verti{r^{(2)}}_{k(2,0,\delta),1+\delta}^2\right) \d t$, respectively.} were also suitable to obtain finiteness of the integrals
\begin{subequations}\label{prt_w12}
\begin{equation}
\int_0^\infty t \verti{\partial_t w^{(1)}}_{k(1,0,1-\delta),-\delta}^2 \d t \sim \int_0^\infty t \verti{\partial_t u}_{\ell(1,0,1-\delta),-\delta}^2 \d t
\end{equation}
and
\begin{equation}
\int_0^\infty t \verti{\partial_t w^{(2)}}_{k(2,0,\delta),\delta}^2 \d t \sim \int_0^\infty t \verti{\partial_t u - \frac{\d u_0}{\d t}}_{\ell(2,0,\delta),\delta}^2 \d t.
\end{equation}
\end{subequations}
Now we may use finiteness of the norms $\verti{\partial_t u}_{\ell(1,0,1-\delta),-\delta}$ and $\verti{\partial_t u - \frac{\d u_0}{\d t}}_{\ell(2,0,\delta),\delta}$ for some time $\tau_\nu > 0$ with $\tau_\nu \searrow 0$ as $\nu \to \infty$. Applying \cite[Prop.~7.6]{ggko.2014} implies that also $\vertiii{\partial_t u(\cdot+\tau_\nu)}$ for $N_0 = 1$ (and with appropriately chosen derivative numbers) is finite\footnote{Notice that the solution of \cite[Prop.~7.6]{ggko.2014} coincides with $\partial_t u$, because uniqueness under the assumption $\int_{\tau_\nu}^\infty \verti{\partial_t u}_{\ell\left(1,0,1-\delta\right),- \delta} \, \d t < \infty$ holds true as can be seen in the uniqueness proof for \cite[Prop.~7.6]{ggko.2014}.}. In fact, since the arguments in \cite[Sec.~7]{ggko.2014} were also applicable for general weights in the interval $\left(\frac 1 2, 1\right)$, we also obtain
\[
\sup_{t \ge \tau_\nu} (t - \tau_\nu)^{2 \alpha - 1} \verti{\partial_t w^{(1)}}_{k(1,1,\alpha^\prime),\alpha^\prime-1}^2 \d t < \infty
\]
and
\[
\int_{\tau_\nu}^\infty \left(t - \tau_\nu\right)^{2 \alpha - 1} \left(\verti{\partial_t^2 w^{(1)}}_{k(1,1,\alpha^\prime),\alpha^\prime-1}^2 + \verti{\partial_t w^{(1)}}_{k(1,1,\alpha^\prime),\alpha^\prime}^2\right) \d t < \infty
\]
for all admissible weights $\alpha$. In particular the integral
\[
\int_{\tau_\nu}^\infty \left(t - \tau_\nu\right) \verti{\partial_t^2 w^{(1)}}_{k(1,1,1-\delta),-\delta}^2 \d t \sim \int_{\tau_\nu}^\infty \left(t - \tau_\nu\right) \verti{\partial_t^2 u}_{\ell(1,1,1-\delta),-\delta}^2 \d t
\]
is finite and we may argue as above that
\[
\int_{\tau_\nu}^\infty \left(t - \tau_\nu\right) \verti{\partial_t^2 w^{(2)}}_{k(2,1,\delta),\delta}^2 \d t \sim \int_{\tau_\nu}^\infty \left(t - \tau_\nu\right) \verti{\partial_t^2 u - \frac{\d^2 u_0}{\d t^2}}_{\ell(2,1,\delta),\delta}^2 \d t
\]
is finite. Hence $\verti{\partial_t^2 u}_{\ell(1,1,1-\delta),-\delta} < \infty$ and $\verti{\partial_t^2 u - \frac{\d^2 u_0}{\d t^2}}_{\ell(2,1,\delta),\delta} < \infty$ for some $t = \tilde \tau_\nu > \tau_\nu$, where $\tilde \tau_\nu \searrow 0$ as $\nu \to \infty$. Apparently the reasoning can be boot strapped and (taking the limit $\nu \to \infty$) we obtain
\begin{subequations}
\begin{equation}
\sup_{t \ge \tau} \verti{\partial_t^m w^{(1)}}_{k(1,m,\alpha^\prime)+2,\alpha^\prime-\frac 1 2}^2 < \infty, \quad \int_{\tau}^\infty \verti{\partial_t^{m+1} w^{(1)}}_{k(1,m,\alpha^\prime),\alpha^\prime-1}^2 \d t< \infty,
\end{equation}
and
\begin{equation}\label{dtmw1_fin}
\int_{\tau}^\infty \verti{\partial_t^m w^{(1)}}_{k(1,m,\alpha^\prime)+4,\alpha^\prime}^2 \d t < \infty
\end{equation}
\end{subequations}
for all $\tau > 0$, all $m = 0, \cdots, N_0-1$, and all admissible $\alpha$.

\medskip

In order to obtain additional spatial regularity, we first observe
\[
\int_\tau^\infty \verti{\partial_t^{m+1} w^{(2)}}_{k(2,m,\alpha^\prime),\alpha^\prime}^2 \d t \stackrel{\eqref{same_alpha}}{\lesssim} \int_\tau^\infty \verti{\partial_t^{m+1} w^{(1)}}_{k\left(1,m+1,\alpha^\prime\right)+4,\alpha^\prime}^2 \d t < \infty \quad \mbox{for all } \, \tau > 0.
\]
Then we may use
\[
\partial_t^m v^{(3)} \stackrel{\eqref{def_vwr}}{=} p(D-2) \partial_t^m w^{(2)} \stackrel{\eqref{lin_wnm}}{=} \partial_t^m r^{(2)} + x q_2(D) \partial_t^{m+1} v^{(2)} - x \partial_t^{m+1} w^{(2)} \quad \mbox{for } \, t, x > 0
\]
and therefore
\begin{equation}\label{est_iter}
\begin{aligned}
\verti{\partial_t^m v^{(3)}}_{k(2,m,\alpha^\prime),\alpha^\prime+1} &\lesssim \verti{\partial_t^m r^{(2)}}_{k(2,m,\alpha^\prime),\alpha^\prime + 1} + \verti{q_2(D) \partial_t^{m+1} v^{(2)}}_{k(2,m,\alpha^\prime),\alpha^\prime} + \verti{\partial_t^{m+1} w^{(2)}}_{k(2,m,\alpha^\prime),\alpha^\prime} \nonumber\\
&\hspace{-18pt}\stackrel{\eqref{def_vwr}, \eqref{same_alpha}}{\lesssim} \verti{\partial_t^m r^{(2)}}_{k(2,m,\alpha^\prime),\alpha^\prime + 1} + \verti{\partial_t^{m+1} w^{(1)}}_{k(1,m+1,\alpha^\prime)+4,\alpha^\prime}.
\end{aligned}
\end{equation}
Since by construction $\int_\tau^\infty \verti{\partial_t^m r^{(2)}}_{k(2,m,\alpha^\prime),\alpha^\prime + 1}^2 \d t < \infty$ (cf.~\eqref{norm_rhs_star}) and by \eqref{dtmw1_fin}, we also have
\[
\int_\tau^\infty \verti{\partial_t^m v^{(3)}}_{k(2,m,\alpha^\prime),\alpha^\prime+1}^2 \d t < \infty \quad \mbox{for all } \, \tau > 0.
\]
Our aim is to show
\begin{equation}\label{w2_est_v3}
\verti{\partial_t^m w^{(2)}}_{k(2,m,\alpha^\prime)+4,\alpha^\prime+1} \lesssim \verti{\partial_t^m v^{(3)}}_{k(2,m,\alpha^\prime),\alpha^\prime+1} \quad \mbox{for } \, t > 0.
\end{equation}
This follows by applying elliptic regularity of $p(D-2)$ (Proposition~\ref{prop:elliptic}), but care has to be taken as the decay assumptions \eqref{ell_assume} have to be satisfied: As in the proof of \cite[Lem.~B.3]{ggko.2014} we may approximate $\partial_t^m v^{(3)}$ by a sequence of functions $\psi^{(\nu)} \in C^\infty_0((0,\infty))$ in the norm $\verti{\cdot}_{k(2,m,\alpha^\prime),\alpha^\prime+1}$. Solving $p(D-2) \omega^{(\nu)} = \psi^{(\nu)}$ as in \cite[Lem.~B.3]{ggko.2014}, we find a smooth solution $\omega^{(\nu)}$ such that $D^\ell \omega^{(\nu)} = O(x^2)$ as $x \searrow 0$ and $D^\ell \omega^{(\nu)} = O\left(x^{\frac 3 2 - \beta}\right)$ as $x \nearrow \infty$ for all $\ell \ge 0$. By Proposition~\ref{prop:elliptic}, $\verti{\omega^{(\nu)} - \omega^{\left(\nu^\prime\right)}}_{k(2,m,\alpha^\prime)+4,\alpha^\prime+1} \lesssim \verti{\psi^{(\nu)} - \psi^{\left(\nu^\prime\right)}}_{k(2,m,\alpha^\prime),\alpha^\prime+1} \to 0$ as $\nu, \nu^\prime \to \infty$ and hence $\verti{\omega^{(\nu)} - \omega}_{k(2,m,\alpha^\prime)+4,\alpha^\prime+1} \to 0$ as $\nu \to \infty$ for some locally integrable $\omega$. Since $\verti{\omega^{(\nu)}}_{k(2,m,\alpha^\prime)+4,\alpha^\prime+1} \lesssim \verti{\psi^{(\nu)}}_{k(2,m,\alpha^\prime),\alpha^\prime+1}$, in particular $\verti{\omega}_{k(2,m,\alpha^\prime)+4,\alpha^\prime+1} \lesssim \verti{\partial_t^m v^{(3)}}_{k(2,m,\alpha^\prime),\alpha^\prime+1}$ for $t > 0$ by continuity of the norms. Hence, for establishing \eqref{w2_est_v3} it remains to prove
\begin{equation}\label{om_eq_w}
\omega = \partial_t^m w^{(2)} \quad \mbox{almost everywhere.}
\end{equation}
Note that
\[
p(D-2) \left(\omega - \partial_t^m w^{(2)}\right) = p(D-2) \lim_{\nu \to \infty} \omega^{(\nu)} - \partial_t^m v^{(3)} = \lim_{\nu \to \infty} \psi^{(\nu)} - \partial_t^m v^{(3)} = 0
\]
almost everywhere and hence
\begin{equation}\label{spanned}
\omega - \partial_t^m w^{(2)} \in \ker p(D-2) = \spanned \left\{x^2, x^{\beta+2},x^{\frac 3 2-\beta}, x^{\frac 1 2}\right\}.
\end{equation}
Because of $\verti{\partial_t^m w^{(2)}}_{1-\delta} \lesssim \verti{\partial_t^m w^{(1)}}_{4+\verti{I_2},1-\delta} < \infty$, necessarily $\partial_t^m w^{(2)} = o\left(x^{1-\delta}\right)$ as $x \searrow 0$ and $x \nearrow \infty$ for $t > 0$. Together with $\verti{\omega}_{k(2,m,\alpha^\prime)+4,\alpha^\prime+1} < \infty$ therefore
\[
\omega - \partial_t^m w^{(2)} = o\left(x^{1-\delta}\right) \quad \mbox{as } \, x \searrow 0 \quad \mbox{and} \quad \omega - \partial_t^m w^{(2)} = o\left(x^{2-\delta}\right) \quad \mbox{as } \, x \nearrow \infty
\]
almost everywhere in $t > 0$. In view of \eqref{spanned} this implies \eqref{om_eq_w}.

 \medskip
  
Finiteness of $\sup_{t \ge \tau} \verti{\partial_t^m w^{(2)}}_{k(2,m,\alpha^\prime)+2,\alpha^\prime+\frac 1 2}^2$ for $\tau > 0$ can be obtained by the standard trace estimate
\[
\sup_{t \ge \tau} \verti{\partial_t^m w^{(2)}}_{k(2,m,\alpha^\prime)+2,\alpha^\prime+\frac 1 2}^2 \lesssim \int_\tau^\infty \left(\verti{\partial_t^{m+1} w^{(2)}}_{k(2,m,\alpha^\prime),\alpha^\prime}^2 + \verti{\partial_t^m w^{(2)}}_{k(2,m,\alpha^\prime)+4,\alpha^\prime+1}^2\right) \d t.
\]
A proof is contained e.g.~in \cite[Lem.~B.2]{ggko.2014}.

\medskip

Apparently the argument can be boot strapped and we obtain
\begin{subequations}\label{finite}
\begin{align}
&\sup_{t \ge \tau} \verti{\partial_t^m w^{(N-m)}}_{k(N-m,m,\alpha^\prime)+2,\alpha^\prime-N-m-\frac 3 2}^2 < \infty,\\
&\int_\tau^\infty \verti{\partial_t^{m+1} w^{(N-m)}}_{k(N-m,m,\alpha^\prime),\alpha^\prime+N-m-2}^2 \d t < \infty,
\end{align}
and
\begin{equation}
\int_\tau^\infty \verti{\partial_t^m w^{(N-m)}}_{k(N-m,m,\alpha^\prime)+4,\alpha^\prime+N-m-1}^2 \d t < \infty
\end{equation}
\end{subequations}
for all $\tau > 0$, all $N = 1, \cdots, N_0$, $m = 0, \cdots, N-1$, and all admissible $\alpha$.

\subsubsection*{Proof of estimate~\eqref{mr_star}}
Now we consider equation~\eqref{lin_wnm}, where $r^{(n)}$, $v^{(n)}$, and $w^{(n)}$ -- defined through \eqref{def_vwr} -- are smooth. We are aiming at deriving estimate~\eqref{mr_wnm}. Therefore it is more convenient to use the logarithmic variable $s := \ln x$, for which equation~\eqref{lin_wnm} reads
\begin{equation}\label{eq_wnm_s}
e^s \partial_t^{m+1} w^{(n)} + p(\partial_s - n) \partial_t^m w^{(n)} = \partial_t^m r^{(n)} + e^s q_n(\partial_s) \partial_t^{m+1} v^{(n)} \quad \mbox{for $t > 0$ and $s \in \R$.}
\end{equation}
In order to simplify the notation, we write
\begin{equation}\label{notation}
\omega := \partial_t^m w^{(n)}, \quad P(\zeta) := p(\zeta - n), \quad \mbox{and} \quad \varphi := \partial_t^m r^{(n)} + e^s q_n(\partial_s) \partial_t^{m+1} v^{(n)},
\end{equation}
that is, equation~\eqref{eq_wnm_s} can be rephrased as
\begin{equation}\label{eq_om_phi}
e^s \partial_t \omega + P(\partial_s) \omega = \varphi \quad \mbox{for $t > 0$ and $s \in \R$.}
\end{equation}
We take a cut off $\eta \in C^\infty(\R)$ with $\eta \ge 0$, $\eta(s) \equiv 1$ for $\verti{s} \le 1$, and $\eta(s) \equiv 0$ for $\verti{s} \ge 2$. Let $\eta_R(s) := \eta(s/R)$ and test equation~\eqref{eq_om_phi} with $e^{- 2 \mu s} \eta_R^2 w^{(n)}$ in $L^2(\R_s)$, where $\mu$ is in the coercivity range of $P(\zeta)$ (We specifically choose $\mu := \alpha^\prime+n-1$ later on.). Thus we arrive at
\[
\frac 1 2 \frac{\d}{\d t} \int_\R e^{- 2 \left(\mu - \frac 1 2\right) s} \eta_R^2 \omega^2 \d s + \int_\R e^{- 2 \mu s} \eta_R^2 \omega P(\partial_s) \omega \, \d s = \int_\R e^{- 2 \mu s} \eta_R^2 \varphi \omega \, \d s.
\]
Now we commute one factor $\eta_R$ with the differential operator $P(\partial_s)$. Since in the commutator of $\eta_R$ and $P(\partial_s)$ at least one derivative acts on $\eta_R$, we obtain through integration by parts
\begin{align*}
& \frac 1 2 \frac{\d}{\d t} \int_\R e^{- 2 \left(\mu - \frac 1 2\right) s} (\eta_R \omega)^2 \d s + \int_\R e^{- 2 \mu s} (\eta_R \omega) P(\partial_s) (\eta_R \omega) \, \d s \\
& \le \frac{1}{2 \eps} \int_\R e^{- 2 \mu s} (\eta_R \varphi)^2 \, \d s + \frac{\eps}{2} \int_\R e^{- 2 \mu s} (\eta_R \omega)^2 \, \d s + \frac C R \int_{-2R}^{2R} e^{- 2 \mu s} \left(\omega^2 + (\partial_s\omega)^2 + (\partial_s^2\omega)^2\right) \d s,
\end{align*}
where $C > 0$ is independent of $R$ and $\eps > 0$ is arbitrary\footnote{All estimates and constants in this part only depend on $\kappa$ and $\mu$.}. As a next step, we use coercivity of $P(\partial_s)$ and obtain
\begin{equation}\label{basic}
\begin{aligned}
& \frac{\d}{\d t} \int_\R e^{- 2 \left(\mu - \frac 1 2\right) s} (\eta_R \omega)^2 \d s + \int_\R e^{- 2 \mu s} \left((\eta_R \omega)^2 + (\partial_s \eta_R \omega)^2 + (\partial_s^2 \eta_R \omega)^2\right) \, \d s\\
&\quad \lesssim \int_\R e^{- 2 \mu s} (\eta_R \varphi)^2 \, \d s + \frac C R \int_{-2R}^{2R} e^{- 2 \mu s} \left(\omega^2 + (\partial_s\omega)^2 + (\partial_s^2\omega)^2\right) \d s.
\end{aligned}
\end{equation}
Estimate~\eqref{basic} is a basic estimate that has to be combined with a higher-order estimate to arrive at \eqref{mr_wnm}. Therefore, we again consider equation~\eqref{eq_om_phi} and apply the operator $e^{- 2 \mu s} (\partial_s-1)^{\kappa+2} \eta_R$ to it (where $\kappa = k(n,m,\alpha^\prime)$). Testing with $\partial_s^{\kappa+2} (\eta_R \omega)$ in $L^2(\R_s)$, we obtain after integrating by parts and using a standard interpolation estimate
\begin{equation}\label{est_higher}
\begin{aligned}
&\frac{\d}{\d t} \int_\R e^{-2 \left(\mu - \frac 1 2\right) s} \left(\partial_s^{\kappa+2}(\eta_R \omega)\right)^2 \d s + \int_\R e^{-2 \mu s} \left(\partial_s^{\kappa+4} (\eta_R \omega)\right)^2 \d s \\
&\lesssim \eps^{-1} \sum_{j = 0}^\kappa \int_\R e^{- 2 \mu s} \verti{\partial_s^j \left(\eta_R \varphi\right)}^2 \, \d s + \eps \sum_{j = 0}^{\kappa+4} \int_\R e^{-2 \mu s} \left(\partial_s^j \left(\eta_R \omega\right)\right)^2 \d s  + \frac 1 R \int_{-2 R}^{2 R} e^{-2 \mu s} \left(\omega^2 + (\partial_s^{\kappa+4} \omega)^2\right) \d s,
\end{aligned}
\end{equation}
where $\eps > 0$ is arbitrary. By interpolation and increasing $C$, the second term in the second line can be absorbed. Using \eqref{basic}, \eqref{est_higher}, and interpolating once more, we obtain, after undoing the transformation $s = \ln x$,
\[
\frac{\d}{\d t} \verti{\eta_R \omega}_{\kappa+2,\mu-\frac 1 2}^2 + \verti{\eta_R \omega}_{\kappa+4,\mu}^2 \lesssim \verti{\eta_R \varphi}_{\kappa,\mu}^2 + \frac 1 R \verti{\eta_{2 R} \omega}_{\kappa+4,\mu}^2.
\]
Multiplying with $t^{2 \sigma}$, where $\sigma := \alpha + n + m - \frac 3 2$ and integrating in time (assuming $\tau^\prime > \tau > 0$), we obtain
\begin{eqnarray}\nonumber
\lefteqn{\sup_{t \in \left[\tau, \tau^\prime\right]} t^{2 \sigma} \verti{\eta_R \omega}_{\kappa+2,\mu-\frac 1 2}^2 + \int_{\tau}^{\tau^\prime} t^{2 \sigma} \verti{\eta_R \omega}_{\kappa+4,\mu}^2 \d t} \\
&\lesssim& \tau^{2 \sigma} \verti{\eta_R \omega_{|t=\tau}}_{\kappa+2,\mu-\frac 1 2}^2 + \int_{\tau}^{\tau^\prime} t^{2 \sigma} \verti{\eta_R \varphi}_{\kappa,\mu}^2 \d t + 2 \sigma \int_{\tau}^{\tau^\prime} t^{2 \sigma - 1} \verti{\eta_R \omega}_{\kappa+2,\mu-\frac 1 2}^2 \d t \nonumber\\
&&+ \frac 1 R \int_{\tau}^{\tau^\prime} t^{2 \sigma} \verti{\eta_{2 R} \omega}_{\kappa+4,\mu}^2 \d t. \label{est_s_var}
\end{eqnarray}
Due to \eqref{finite} and \eqref{notation} all terms appearing in \eqref{est_s_var} remain finite in the limit $R \to \infty$ (note that $\verti{\eta_{2 R} \omega}_{\kappa+4,\mu} \lesssim \verti{\partial_t^m w^{(n)}}_{k(n,m,\alpha^\prime)+4,\alpha^\prime+n-1}$ for $t > 0$)  so that we obtain
\begin{align*}
& \sup_{t \in \left[\tau, \tau^\prime\right]} t^{2 \sigma} \verti{\omega}_{\kappa+2,\mu-\frac 1 2}^2 + \int_{\tau}^{\tau^\prime} t^{2 \sigma} \verti{\omega}_{\kappa+4,\mu}^2 \d t \\
& \quad \lesssim \tau^{2 \sigma} \verti{\omega_{|t=\tau}}_{\kappa+2,\mu-\frac 1 2}^2 + \int_{\tau}^{\tau^\prime} t^{2 \sigma} \verti{\varphi}_{\kappa,\mu}^2 \d t + 2 \sigma \int_{\tau}^{\tau^\prime} t^{2 \sigma - 1} \verti{\omega}_{\kappa+2,\mu-\frac 1 2}^2 \d t.
\end{align*}
Undoing the notational change \eqref{notation} and increasing the set $\left[\tau,\tau^\prime\right]$, we end up with \eqref{mr_wnm}, provided that for $\alpha + n + m > \frac 3 2$
\begin{equation}\label{conv_init}
\tau^{2 \sigma} \verti{\omega_{|t=\tau}}_{\kappa+2,\mu-\frac 1 2}^2 = \tau^{2 (\alpha+n+m)-3} \verti{\partial_t^m w^{(n)}_{|t=\tau}}_{k\left(n,m,\alpha^\prime\right)+2,\alpha^\prime+n-\frac 3 2}^2 \to 0 \quad \mbox{as } \, \tau \searrow 0
\end{equation}
holds true at least for a subsequence. This is already clear for $(n,m) = (1,0)$ and $\alpha^\prime \in \left(\frac 1 2, 1\right)$ from the extension of the discretization argument in \cite[Sec.~7]{ggko.2014}. From \eqref{mr_wnm} for $(n,m) = (1,0)$, we get
\begin{equation}\label{lim1}
\int_0^\infty t^{2 \alpha-1} \left(\verti{\partial_t w^{(1)}}_{k\left(1,0,\alpha^\prime\right),\alpha^\prime-1}^2 + \verti{w^{(1)}}_{k\left(1,0,\alpha^\prime\right)+4,\alpha^\prime}^2\right) \, \d t < \infty \quad \mbox{for } \, \alpha^\prime \in \left(\frac 1 2, 1\right).
\end{equation}
Utilizing
\[
\int_0^\infty t^{2 \alpha} \verti{w^{(2)}}_{k\left(2,0,\alpha^\prime\right)+2,\alpha^\prime+\frac 1 2}^2 \, \d t \stackrel{\eqref{alpha+12}}{\lesssim} \int_0^\infty t^{2 \alpha} \verti{w^{(1)}}_{k\left(1,0,\alpha^\prime + \frac 1 2\right)+4,\alpha^\prime+\frac 1 2}^2 \, \d t < \infty \quad \mbox{for } \, \alpha^\prime \in \left(0,\frac 1 2\right),
\]
we conclude that \eqref{conv_init} holds for $(n,m) = (2,0)$ and $\alpha^\prime \in \left(0,\frac 1 2\right)$, too. Additionally, we have
\[
\int_0^\infty t^{2 \alpha} \verti{\partial_t w^{(1)}}_{k\left(1,1,\alpha^\prime\right)+2,\alpha^\prime-\frac 1 2}^2 \, \d t \stackrel{\eqref{alpha+12_alt}}{\lesssim} \int_0^\infty t^{2 \left(\alpha + \frac 1 2\right)-1} \verti{\partial_t w^{(1)}}_{k\left(1,0,\alpha^\prime+\frac 1 2\right), \left(\alpha^\prime+\frac 1 2\right) - 1}^2 \, \d t \stackrel{\eqref{lim1}}{<} \infty
\]
for $\alpha^\prime \in \left(0,\frac 1 2\right)$, that is, \eqref{conv_init} also holds for $(n,m) = (1,1)$ and $\alpha^\prime \in \left(0,\frac 1 2\right)$.

\medskip

Next, \eqref{mr_wnm} yields
\begin{equation}\label{lim2}
\int_0^\infty t^{2 \alpha+1} \left(\verti{\partial_t^2 w^{(1)}}_{k\left(1,1,\alpha^\prime\right),\alpha^\prime-1}^2 + \verti{\partial_t w^{(1)}}_{k\left(1,1,\alpha^\prime\right)+4,\alpha^\prime}^2\right) \, \d t < \infty \quad \mbox{for } \, \alpha^\prime \in \left(0,\frac 1 2\right).
\end{equation}
As
\[
\int_0^\infty t^{2 \alpha + 1} \verti{\partial_t w^{(1)}}_{k\left(1,1,\alpha^\prime\right)+2,\alpha^\prime-\frac 1 2}^2 \, \d t \stackrel{\eqref{alpha-12}}{\lesssim} \int_0^\infty t^{2 \left(\alpha - \frac 1 2\right) + 2} \verti{\partial_t w^{(1)}}_{k\left(1,1,\alpha^\prime-\frac 1 2\right)+4,\alpha^\prime - \frac 1 2}^2 \, \d t \stackrel{\eqref{lim2}}{<} \infty
\]
for $\alpha^\prime \in \left(\frac 1 2, 1\right)$, \eqref{conv_init} is true for $(n,m) = (1,1)$ and $\alpha^\prime \in \left(\frac 1 2,1\right)$ and \eqref{mr_wnm} yields
\begin{equation}\label{lim3}
\int_0^\infty t^{2 \alpha+1} \left(\verti{\partial_t^2 w^{(1)}}_{k\left(1,1,\alpha^\prime\right),\alpha^\prime-1}^2 + \verti{\partial_t w^{(1)}}_{k\left(1,1,\alpha^\prime\right)+4,\alpha^\prime}^2\right) \, \d t < \infty \quad \mbox{for } \, \alpha^\prime \in \left(\frac 1 2,1\right).
\end{equation}
Since we also have
\[
\int_0^\infty t^{2 \alpha + 1} \verti{\partial_t v^{(2)}}_{k\left(2,0,\alpha^\prime\right) + \verti{I_2}-1,\alpha^\prime}^2 \, \d t \stackrel{\eqref{same_alpha}}{\lesssim} \int_0^\infty t^{2 \alpha + 1} \verti{\partial_t w^{(1)}}_{k\left(1,1,\alpha^\prime\right)+4,\alpha^\prime}^2 \, \d t \stackrel{\eqref{lim2}}{<} \infty
\]
for $\alpha^\prime \in \left(0,\frac 1 2\right)$, we get from \eqref{mr_wnm}
\begin{equation}\label{lim4}
\int_0^\infty t^{2 \alpha+1} \left(\verti{\partial_t w^{(2)}}_{k\left(2,0,\alpha^\prime\right),\alpha^\prime}^2 + \verti{w^{(2)}}_{k\left(2,0,\alpha^\prime\right)+4,\alpha^\prime}^2\right) \, \d t < \infty \quad \mbox{for } \, \alpha^\prime \in \left(0,\frac 1 2\right).
\end{equation}
This enables us to estimate
\[
\int_0^\infty t^{2 \alpha} \verti{w^{(2)}}_{k\left(2,0,\alpha^\prime\right)+2,\alpha^\prime+\frac 12}^2 \, \d t \stackrel{\eqref{alpha-12}}{\lesssim} \int_0^\infty t^{2 \left(\alpha - \frac 1 2\right) + 1} \verti{w^{(2)}}_{k\left(2,0,\alpha^\prime-\frac 1 2\right)+4,\left(\alpha^\prime-\frac 12\right)+1}^2 \, \d t \stackrel{\eqref{lim4}}{<} \infty
\]
for $\alpha^\prime \in \left(\frac 1 2, 1\right)$, thus proving \eqref{conv_init} with $(n,m) = (2,0)$ and $\alpha^\prime \in \left(\frac 1 2,1\right)$.

\medskip

Induction on $n+m$ verifies \eqref{conv_init} for all $\alpha+n+m > \frac 3 2$.

\medskip

Now, starting from \eqref{mr_wnm}, we follow the reasoning of Section~\ref{ssec:heur_par} and arrive at estimate~\eqref{mr_star}.

\subsubsection*{Proof of estimate~\eqref{mr_main}}
To obtain the desired maximal-regularity estimate \eqref{mr_main}, we only need to be able to apply Proposition~\ref{prop:elliptic}, that is, we need to verify \eqref{ell_assume}. We exemplify the reasoning by treating the last line of \eqref{norm_sol_star}. For a fixed value $\left(\alpha,N\right) \in \AAA$ it suffices to show
\begin{equation}\label{dellu}
D^\ell \partial_t^m u = \sum_{i < \varrho} \frac{\d^m u_i}{\d t^m} i^\ell x^i + o\left(x^\varrho\right) \quad \mbox{as } \, x \searrow 0,
\end{equation}
where $\varrho := \alpha^\prime+N-m-1$, $\alpha^\prime = \alpha \pm \delta \in (0,1)$, and $\ell = 0,\cdots, k\left(N-m,m,\alpha^\prime\right)+4$.

\medskip

Due to $\vertiii{u}_* < \infty$, we have $D^\ell \partial_t^m w^{(N-m)}(x) = o(x^\varrho)$ as $x \searrow 0$. This implies $D^\ell \partial_t^m u(x) = D^\ell \pi^{(m)}(x) + D^\ell \tilde u(x)$, where $D^\ell \tilde u(x) = o(x^\varrho)$ as $x \searrow 0$ and $\pi^{(m)}(x)$ is a solution of the homogeneous equation $P(D) \pi^{(m)} = 0$ with $P(\zeta) := \left(\prod_{k = 0}^{N-m-1} p(\zeta-k)\right) \left(\prod_{i \in J_{N-m}} (\zeta-i)\right)$ (cf.~\eqref{def_wn}). Hence $\pi^{(m)} = \sum_j \frac{\d^m u_j}{\d t^m} x^j$, where the $u_j$ are smooth functions of time $t$ and $j$ is a zero of the polynomial $P(\zeta)$. In fact, the set of admissible exponents $j$ is smaller: Since we know that Proposition~\ref{prop:main_lin} also holds for $N_0 = 1$, we need to have $\pi^{(m)}(x) = \frac{\d^m u_0}{\d t^m} + \frac{\d^m u_\beta}{\d t^m} x^\beta + o(x^\beta)$ as $x \searrow 0$. Hence $j \in \{0,\beta\}$ or $\beta < j < N_0$.

\medskip

Now suppose that $\beta < j < N_0$ and $j \notin K_{N_0}$. Since $x^j$ does not appear in the expansion of $f$, we obtain from the linear equation \eqref{lin_pde} $\frac{\d u_{j-1}}{\d t} + p(j) u_j \equiv 0$, that is, if $\frac{\d^m u_j}{\d t^m} \not\equiv 0$ necessarily $\frac{\d^m u_{j-1}}{\d t^m} \not\equiv 0$. This implies $j - \floor{j} \in \{0,\beta\}$ (as otherwise negative $j$ would appear) thus contradicting the assumption $j \notin K_{N_0}$ (cf.~\eqref{def_kn}).

\medskip

Hence indeed \eqref{dellu} holds true and Proposition~\ref{prop:elliptic} yields estimate~\eqref{mr_main}.
\end{proof}

\section{The nonlinear problem\label{sec:nonlinear}}
\subsection{The main results\label{ssec:nonmain}}
In this section we prove our main theorem. For the treatment of the nonlinear problem \eqref{cauchy}, we need additional assumptions on the numbers $\ell(n,m,\alpha)$ which we already state here:
\begin{subequations}\label{cond_l_non}
\begin{align}
\ell\left(N-m,m^\prime,\alpha^\prime\right) &\ge \ell\left(N-m,m,\alpha^\prime\right) \quad \mbox{for } \, m^\prime \le m-1,\label{cond_l_non1} \\
\ell\left(1,m^\prime,\frac 1 2\pm\delta\right) &\ge \frac{\ell\left(N-m,m,\alpha^\prime\right)}{2} + 1 \quad \mbox{where } \begin{cases} 0 \le m^\prime \le m, \\ \alpha^\prime+N-m-1 < \beta, \end{cases} \label{cond_l_non20}\\
\ell\left(1,m^\prime,\frac 1 2\pm\delta\right) &\ge \ell\left(N-m,m,\alpha^\prime\right) + 3 \quad \mbox{where } \, \begin{cases} 0 \le m^\prime \le m, \\ \alpha^\prime+N-m-1 > \beta,\end{cases} \label{cond_l_non2}\\
\ell\left(N-m,m,\alpha^\prime-\frac 1 2\right) &\ge \ell\left(N-m,m,\alpha^\prime\right) + 2 \quad \mbox{for all } \alpha^\prime \in \left(\frac 1 2, 1\right), \label{cond_l_non3}\\
\ell\left(N-m-1,m,\alpha^\prime+\frac 1 2\right) &\ge \ell\left(N-m,m,\alpha^\prime\right) + 2 \quad \mbox{for all } \alpha^\prime\in \left(0, \frac 1 2\right), \label{cond_l_non4} \\
k &\ge \ell\left(1,m,\alpha^\prime\right)-2 \quad \mbox{for all } \, \alpha^\prime > \beta \, \mbox{ and } \, m \ge 0, \label{cond_l_non5} \\
\begin{split} \ell\left(1,0,1-\delta\right) &\ge \ell\left(1,m,\alpha^\prime\right)+2, \\ \ell\left(2,0,\delta\right) &\ge \ell\left(1,m,\alpha^\prime\right)+2 \end{split} \Bigg\} \quad \mbox{for all } \, \alpha^\prime > \beta \, \mbox{ and } \, m \ge 1. \label{cond_l_non6}
\end{align}
\end{subequations}
For the explicit choice \eqref{explicit}, we are indeed able to fulfill the ``linear" conditions \eqref{linear_cond} and the ``nonlinear" conditions \eqref{cond_l_non}. The relevance of conditions~\eqref{cond_l_non} will become clear in the proof of the main estimate for the nonlinearity (cf.~Proposition~\ref{prop:nonlinear_est}). The aim of this section is to show the following statement:
\begin{theorem}\label{th:regularity}
Suppose $N_0 \in \N_0$ and further suppose that conditions~\eqref{linear_cond} and \eqref{cond_l_non} are fulfilled (cf.~\eqref{explicit} for an explicit choice). Then there exist $\eps > 0$ and $\delta > 0$ such that for any $u^{(0)} \in U_0$ with $\vertiii{u^{(0)}}_0 \le \eps$ (cf.~\eqref{norm_initial}) problem~\eqref{cauchy} has a unique solution $u \in U$. This solution obeys the \emph{a-priori estimate} $\vertiii{u} \lesssim \vertiii{u^{(0)}}_0$ (cf.~\eqref{norm_sol}) with a constant independent of $u^{(0)}$.
\end{theorem}
It is quite apparent that Theorem~\ref{th:regularity} (with $N_0$ replaced by $N_0 + 1$) yields Theorem~\ref{th:main} (the regularity and decay properties of $u_i = u_i(t)$ and $R_{N_0} = R_{N_0}(x,t)$ in \eqref{expansion} follow from Lemma~\ref{lem:est_coeff} and the definition of $\vertiii{\cdot}$, cf.~\eqref{norm_sol}, by approximation with smooth functions as given in Definition~\ref{def:spaces}). Theorem~\ref{th:regularity} in turn has already been proven for $N_0 = 1$ in \cite[Th.~3.1]{ggko.2014} so that in particular uniqueness holds true. The fact that existence also holds in a smaller space follows by applying the following proposition:
\begin{proposition}\label{prop:nonlinear_est}
Suppose that $N_0 \in \N$ and $\delta > 0$ is chosen sufficiently small. Furthermore suppose that conditions~\eqref{linear_cond}~and~\eqref{cond_l_non} are fulfilled (cf.~\eqref{explicit} for an explicit choice). Then the following estimate for the nonlinearity $\NN(u)$ (defined in \eqref{nonlinearity} and \eqref{5linear}) holds true:
\begin{equation}\label{non_main}
\vertiii{\NN(u)}_1 \lesssim \max_{m = 2, 5} \vertiii{u}^m \quad \mbox{for any } \, u \in U,
\end{equation}
where the constant in the estimate is independent of $u$.
\end{proposition}
Before proving Proposition~\ref{prop:nonlinear_est}, we show how it can be used to prove Theorem~\ref{th:regularity}:
\begin{proof}[of Theorem~\ref{th:regularity}]
In Proposition~\ref{prop:main_lin} we have constructed a solution operator $S: \, U_0 \times F \to U$ to the linear problem \eqref{lin_cauchy}, so that the nonlinear problem \eqref{cauchy} is equivalent to
\begin{equation}\label{fixed}
u = S\left[u^{(0)},\NN(u)\right].
\end{equation}
This fixed-point problem can be solved by applying the contraction-mapping theorem, which for $N_0 = 1$ has indeed been carried out in the proof of \cite[Th.~3.1]{ggko.2014} under the assumption of small data $u^{(0)}$. This implies that for $\vertiii{u^{(0)}}_0 \ll 1$ the sequence $\left(u^{(\nu)}\right)_{\nu \in \N}$ defined through
\begin{equation}\label{def_sequence}
u^{(1)} := S[u^{(0)},0] \quad \mbox{and} \quad u^{(\nu+1)} := S\left[u^{(0)},\NN\left(u^{(\nu)}\right)\right] \quad \mbox{for } \, \nu \ge 1
\end{equation}
converges in $U$ for $N_0 = 1$ to the unique solution $u$. In particular $u^{(\nu)}$ converges point-wise to $u$ (cf.~Lemma~\ref{lem:c0control}). For arbitrary $N_0 \ge 1$ we note that since $u^{(\nu)} \in U$, applying the maximal-regularity estimate \eqref{mr_main} (cf.~Proposition~\ref{prop:main_lin}) and the nonlinear estimate \eqref{non_main} (cf.~Proposition~\ref{prop:nonlinear_est}), we obtain
\begin{equation}\label{est_sequence_1}
\vertiii{u^{(1)}} \le C \vertiii{u^{(0)}}_0 \quad \mbox{and} \quad \vertiii{u^{(\nu+1)}} \le C \left(\vertiii{u^{(0)}}_0 + \max_{m = 2,5} \vertiii{u^{(\nu)}}^m\right) \quad \mbox{for } \, \nu \ge 1,
\end{equation}
where $C > 0$ is independent of $u^{(\nu)}$. Assuming $\eps \le (4 C^2)^{-1}$ and $2 C \ge 1$, estimates~\eqref{est_sequence_1} upgrade to
\begin{equation}\label{est_sequence_2}
\vertiii{u^{(1)}} \le \frac{1}{4 C} \quad \mbox{and} \quad \vertiii{u^{(\nu+1)}} \le \frac{1}{4 C} + C \max_{m = 2,5} \vertiii{u^{(\nu)}}^m \quad \mbox{for } \, \nu \ge 1.
\end{equation}
Inequalities~\eqref{est_sequence_2} in turn imply
\begin{equation}\label{est_sequence_3}
\vertiii{u^{(\nu)}} \le \frac{1}{2 C} \quad \mbox{for all } \, \nu \ge 1.
\end{equation}
By weak-$*$-compactness, $u^{(\nu)}$ has a subsequence that converges in the weak-$*$-topology of $U$ to the unique limit $u$ and therefore $u \in U$. By weak-$*$ lower semi-continuity of the norm $\vertiii{\cdot}$, estimate~\eqref{est_sequence_3} implies
\begin{equation}\label{est_lim_u}
\vertiii{u} \le \frac{1}{2 C}.
\end{equation}
Applying the maximal-regularity estimate \eqref{mr_main} (cf.~Proposition~\ref{prop:main_lin}) and the nonlinear estimate \eqref{non_main} (cf.~Proposition~\ref{prop:nonlinear_est}) once more to the fixed-point equation \eqref{fixed}, we obtain
\[
\vertiii{u} \le C \left(\vertiii{u^{(0)}}_0 + \max_{m = 2,5} \vertiii{u}^m\right),
\]
which, in view of \eqref{est_lim_u} and since $2 C \ge 1$, upgrades to $\vertiii{u} \lesssim \vertiii{u^{(0)}}_0$.
\end{proof}

\subsection{Nonlinear estimates\label{ssec:nonest}}
We conclude the paper with estimating the nonlinear part $\NN(u)$ of equation~\eqref{lin_pde}:
\begin{proof}[of Proposition~\ref{prop:nonlinear_est}]
We repeat some of the observations on the formal structure of the nonlinearity $\NN(u)$ that are contained in \cite[Sec.~8]{ggko.2014}:
\subsubsection*{The formal structure of the nonlinearity}
We recall that the nonlinearity $\NN(u)$ can be written as
\[
\NN(u) \stackrel{\eqref{nonlinearity}}{=} p(D) u - \MM_\mathrm{sym}(1+u,\cdots,1+u),
\]
where $p(D) u = 5 \MM_\mathrm{sym}(1+u,1,\cdots,1)$ and $\MM_\mathrm{sym}$ is the symmetrization of $\MM$. Since we know\footnote{The identity \eqref{vanish_m_1} holds, as the single factor $D$ appears once, because the traveling wave $u_\mathrm{TW} \equiv 0$ is a solution of the nonlinear problem \eqref{cauchy}.}
\begin{equation}\label{vanish_m_1}
\MM_\mathrm{sym}(1,\cdots,1) \stackrel{\eqref{5linear}}{=} 0,
\end{equation}
we can conclude that by multi-linearity $\NN(u)$ is a linear combination of terms of the form
\begin{equation}\label{struct_nu}
\MM_\mathrm{sym}\left(u,u,\omega^{(3)}, \omega^{(4)}, \omega^{(5)}\right) \quad \mbox{with } \, \omega^{(3)}, \omega^{(4)}, \omega^{(5)} \in \{1,u\}.
\end{equation}
As $u$ can be decomposed into $u = (u-u_0) + u_0$, once more appealing to \eqref{vanish_m_1} and using multi-linearity, we infer from \eqref{struct_nu} that $\NN(u)$ is a linear combination (with constant coefficients) of terms of the form
\begin{equation}\label{struct_nu_2}
\MM_\mathrm{sym}\left(u-u_0,\omega^{(2)}, \cdots, \omega^{(5)}\right) \quad \mbox{with } \, \omega^{(2)} \in \{u_0,u-u_0\}, \; \omega^{(3)},\omega^{(4)}, \omega^{(5)} \in \{1,u_0,u-u_0\}.
\end{equation}
For the first entry of $\MM_\mathrm{sym}$ in \eqref{struct_nu_2} we may further use $u-u_0 = (u-u_0-u_\beta x^\beta) + u_\beta x^\beta$ and the fact that $\MM_\mathrm{sym}(x^\beta,1,\cdots,1) = \frac 1 5 p(D) x^\beta \stackrel{\eqref{poly}}{=} 0$. Thus $\NN(u)$ is a linear combination (with constant coefficients) of terms of the form
\begin{subequations}\label{struct_nu_3}
\begin{equation}\label{struct_nu_31}
\MM_\mathrm{sym}\left(u-u_0-u_\beta x^\beta,\omega^{(2)}, \cdots, \omega^{(5)}\right) \quad \mbox{with } \, \omega^{(2)} \in \{u_0,u-u_0\}, \; \omega^{(3)},\omega^{(4)}, \omega^{(5)} \in \{1,u_0,u-u_0\}.
\end{equation}
or
\begin{equation}\label{struct_nu_32}
\MM_\mathrm{sym}\left(u_\beta x^\beta,u-u_0, \omega^{(3)}, \omega^{(4)}, \omega^{(5)}\right) \quad \mbox{with } \, \omega^{(3)},\omega^{(4)}, \omega^{(5)} \in \{1,u_0,u-u_0\}.
\end{equation}
\end{subequations}
%
\subsubsection*{The norms}
In view of the definition of the norm $\vertiii{\cdot}_1$ in \eqref{norm_rhs}, it suffices to estimate the expression
\begin{equation}\label{rhs_term}
\int_0^\infty t^{2 (\alpha+N)-3} \verti{\partial_t^m \NN(u) - \sum_{\beta < i < \alpha^\prime+N-m-1} \frac{\d^m (\NN(u))_i}{\d t^m} x^i}_{\ell,\alpha^\prime+N-m-1}^2 \d t,
\end{equation}
where $\ell = \ell(N-m,m,\alpha^\prime)$ and as in \eqref{expansion}
\[
\NN(u) = \sum_{i \in K_{N_0}} \left(\NN(u)\right)_i x^i + O\left(x^{N_0}\right) \quad \mbox{as } \, x \searrow 0 \, \mbox{ and } \, t > 0,
\]
for each $(\alpha,N) \in \AAA$ with $\alpha^\prime = \alpha \pm \delta \in (0,1)$ and $m \in \{0,\cdots, N-1\}$ individually. We distinguish between sub- and super-critical terms, that is, terms for which either $\alpha^\prime+N-m-1 < \beta$ or $\alpha^\prime+N-m-1 > \beta$.
\subsubsection*{Estimating the sub-critical terms}
In this part of the proof, we concentrate on the \emph{sub-critical} terms, that is, the terms for which $\alpha^\prime+N-m-1 < \beta$. Then the sum $\sum_{\beta < i < \alpha^\prime+N-m-1} \frac{\d^m (\NN(u))_i}{\d t^m} x^i$ is identically zero. It is convenient to use the decomposition \eqref{struct_nu_2} for the nonlinearity $\NN(u)$, so that by distributing the $D$-derivatives and applying the triangle inequality, it suffices to estimate
\[
\int_0^\infty t^{2 (\alpha+N)-3} \verti{\partial_t^m \left(\omega^{(1)}\cdots\omega^{(5)}\right)}_{\alpha^\prime+N-m-1}^2 \d t,
\]
where
\begin{align*}
\omega^{(1)} &= D^{\ell_1} (u-u_0) \quad \mbox{with } \, 0 \le \ell_1 \le \ell+4,\\
\omega^{(2)} &\in \left\{u_0, D^{\ell_2} (u-u_0)\right\} \quad \mbox{with } \, 0 \le \ell_2 \le \frac{\ell+4}{2},\\
\omega^{(j)} &\in \left\{1, u_0, D^{\ell_j} (u-u_0)\right\} \quad \mbox{with } \, 0 \le \ell_j \le \frac{\ell+4}{3} \, \mbox{ for } \, j \ge 3.
\end{align*}
Now we distribute the time derivatives on the individual factors, so that it suffices to estimate
\[
\int_0^\infty t^{2 (\alpha+N)-3} \verti{\omega^{(1)}\cdots\omega^{(5)}}_{\alpha^\prime+N-m-1}^2 \d t,
\]
with new factors $\omega^{(j)}$ obeying
\begin{align*}
\omega^{(1)} &= \partial_t^{m_1} D^{\ell_1} (u-u_0) \quad \mbox{with } \, 0 \le \ell_1 \le \ell+4 \, \mbox{ and } \, 0 \le m_1 \le N-1,\\
\omega^{(2)} &\in \left\{\frac{\d^{m_2} u_0}{\d t^{m_2}}, \partial_t^{m_2} D^{\ell_2} (u-u_0)\right\} \quad \mbox{with } \, 0 \le \ell_2 \le \frac{\ell+4}{2} \, \mbox{ and } \, 0 \le m_2 \le N-1,\\
\omega^{(j)} &\in \left\{\delta_{m_j,0}, \frac{\d^{m_j} u_0}{\d t^{m_j}}, \partial_t^{m_j} D^{\ell_j} (u-u_0)\right\} \quad \mbox{with } \, 0 \le \ell_j \le \frac{\ell+4}{3} \, \mbox{ and } \, 0 \le m_j \le N-1 \, \mbox{ for } \, j \ge 3,
\end{align*}
where $\sum_{j = 1}^5 m_j = m$. This enables us to obtain the following bound:
\begin{equation}\label{est_sub_1}
\begin{aligned}
\int_0^\infty t^{2 (\alpha+N)-3} \verti{\omega^{(1)}\cdots\omega^{(5)}}_{\alpha^\prime+N-m-1}^2 \d t \lesssim \, & \int_0^\infty t^{2 (\alpha+N-m+m_1)-3} \verti{\omega^{(1)}}_{\alpha^\prime+N-m-1}^2 \d t \\
&\times \prod_{j = 2}^5 \sup_{t \ge 0} t^{2 m_j} \vertii{\omega^{(j)}}^2,
\end{aligned}
\end{equation}
where $\vertii{v} := \sup_{x < 0} \verti{v(x)}$. Then we note that
\begin{align*}
\int_0^\infty t^{2 (\alpha+N-m+m_1)-3} \verti{\omega^{(1)}}_{\alpha^\prime+N-m-1}^2 \d t &\lesssim \int_0^\infty t^{2 (\alpha+N-m+m_1)-3} \verti{\partial_t^{m_1} u}_{\ell\left(N-m,m_1,\alpha^\prime\right)+4,\alpha^\prime+N-m-1}^2 \d t \\
&\le \vertiii{u}^2
\end{align*}
for the first term on the right-hand side of \eqref{est_sub_1}, provided that condition~\eqref{cond_l_non1} is fulfilled.

\medskip

We now turn our attention to the terms $\omega^{(j)}$ with $j \ge 2$ in \eqref{est_sub_1}. If $\omega^{(j)} = \frac{\d^{m_j} u_0}{\d t^{m_j}}$ we have $\sup_{t \ge 0} t^{2 m_j} \vertii{\omega^{(j)}}^2 \lesssim \vertiii{u}^2$ by estimate~\eqref{est_coeff2} of Lemma~\ref{lem:est_coeff}. Finally, if $\omega^{(j)} = \partial_t^{m_j} D^{\ell_j} (u-u_0)$, Lemma~\ref{lem:c0control} yields $\sup_{t \ge 0} t^{2 m_j} \vertii{\omega^{(j)}}^2 \lesssim \vertiii{u}^2$, provided that condition~\eqref{cond_l_non20} is fulfilled.

\medskip

In summary we can estimate the term in \eqref{rhs_term} in the sub-critical case $\alpha^\prime+N-m-1 < \beta$ as follows:
\begin{equation}\label{main_subcrit}
\int_0^\infty t^{2 (\alpha+N)-3} \verti{\partial_t^m \NN(u) - \sum_{\beta < i < \alpha^\prime+N-m-1} \frac{\d^m (\NN(u))_i}{\d t^m} x^i}_{\ell,\alpha^\prime+N-m-1}^2 \d t \lesssim \vertiii{u}^4 \times \left(1 + \vertiii{u}^6\right),
\end{equation}
where $\ell = \ell(N-m,m,\alpha^\prime)$.
\subsubsection*{Estimating the super-critical terms}
We now aim at estimating terms of the form \eqref{rhs_term} with $\ell = \ell(N-m,m,\alpha^\prime)$ for each $(\alpha,N) \in \AAA$ with $\alpha^\prime = \alpha \pm \delta \in (0,1)$ and $m \in \{0,\cdots, N-1\}$, where we assume $\alpha^\prime+N-m-1 > \beta$ (the ``super-critical" terms). Here we use the decomposition of $\NN(u)$ in the form \eqref{struct_nu_3}, that is, the super-critical terms \eqref{rhs_term} can be estimated by a sum of terms of the form
\begin{equation}\label{omegas}
\int_0^\infty t^{2 (\alpha+N)-3} \verti{\partial_t^m \left(\omega^{(1)}\cdots\omega^{(5)}\right) - \sum_{i < \alpha^\prime+N-m-1} \frac{\d^m}{\d t^m}\left(\omega^{(1)}\cdots\omega^{(5)}\right)_i x^i}_{\alpha^\prime+N-m-1}^2 \d t,
\end{equation}
where one of the following two situations occurs:
\begin{itemize}
\item[(a)] The $\omega^{(j)}$ obey (cf.~\eqref{struct_nu_31})
\begin{align*}
\omega^{(1)} &= D^{\ell_1} (u - u_0 - u_\beta x^\beta) \quad \mbox{with } \, 0 \le \ell_1 \le \ell+4,\\
\omega^{(2)} &\in \left\{u_0, D^{\ell_2} (u-u_0)\right\} \quad \mbox{with } \, 0 \le \ell_2 \le \ell+4,\\
\omega^{(j)} &\in \left\{1, u_0, D^{\ell_j} (u-u_0)\right\} \quad \mbox{with } \, 0 \le \ell_j \le \frac{\ell+4}{2} \, \mbox{ for } \, j \ge 3.
\end{align*}
\item[(b)] The $\omega^{(j)}$ obey (cf.~\eqref{struct_nu_32})
\begin{align*}
\omega^{(1)} &= u_\beta x^\beta,\\
\omega^{(2)} &= D^{\ell_2} (u-u_0) \quad \mbox{with } \, 0 \le \ell_2 \le \ell+4,\\
\omega^{(j)} &\in \left\{1, u_0, D^{\ell_j} (u-u_0)\right\} \quad \mbox{with } \, 0 \le \ell_j \le \frac{\ell+4}{2} \, \mbox{ for } \, j \ge 3.
\end{align*}
\end{itemize}
%
\subsubsection*{Estimating \eqref{omegas} for terms of the form {\rm (a)}}
As a first step, we distribute the time derivatives on the individual factors $\omega^{(j)}$ so that terms of the form \eqref{omegas} fulfilling (a) can be estimated by a sum of terms of the form
\begin{equation}\label{omegas_a}
\int_0^\infty t^{2 (\alpha+N)-3} \verti{\omega^{(1)}\cdots\omega^{(5)} - \sum_{i < \alpha^\prime+N-m-1} \left(\omega^{(1)}\cdots\omega^{(5)}\right)_i x^i}_{\alpha^\prime+N-m-1}^2 \d t,
\end{equation}
where the $\omega^{(j)}$ are given by
\begin{align*}
\omega^{(1)} &= \partial_t^{m_1} D^{\ell_1} (u - u_0 - u_\beta x^\beta) \quad \mbox{with } \, 0 \le \ell_1 \le \ell+4 \, \mbox{ and } \, 0 \le m_1 \le m,\\
\omega^{(2)} &\in \left\{\frac{\d^{m_2} u_0}{\d t^{m_2}}, \partial_t^{m_2} D^{\ell_2} (u-u_0)\right\} \quad \mbox{with } \, 0 \le \ell_2 \le \ell+4 \, \mbox{ and } \, 0 \le m_2 \le m,\\
\omega^{(j)} &\in \left\{\delta_{m_j,0}, \frac{\d^{m_j} u_0}{\d t^{m_j}}, \partial_t^{m_j} D^{\ell_j} (u-u_0)\right\} \quad \mbox{with } \, 0 \le \ell_j \le \frac{\ell+4}{2} \, \mbox{ and } \, 0 \le m_j \le m \, \mbox{ for } \, j \ge 3,
\end{align*}
with $\sum_{j = 1}^5 m_j = m$. Since $\beta > \frac 1 2$ we have $\omega_i^{(1)} \equiv 0$ for $i < 1$ (cf.~\eqref{def_kn} for the definition of admissible exponents $i$). Hence we may decompose $\omega^{(1)}$ according to
\[
\omega^{(1)} = \left(\omega^{(1)} - \sum_{1 \le i_1 < \alpha^\prime+N-m-1} \omega^{(1)}_{i_1} x^{i_1}\right) + \sum_{1 \le i_1 < \alpha^\prime+N-m-1} \omega^{(1)}_{i_1} x^{i_1}
\]
and instead of terms of the form \eqref{omegas_a}, we may estimate terms of the structure
\begin{subequations}\label{terms_a12}
\begin{equation}\label{terms_a1}
\int_0^\infty t^{2 (\alpha+N)-3} \verti{\left(\omega^{(1)} - \sum_{1 \le i_1 < \alpha^\prime+N-m-1} \omega^{(1)}_{i_1} x^{i_1}\right) \omega^{(2)} \cdots\omega^{(5)}}_{\alpha^\prime+N-m-1}^2 \d t,
\end{equation}
and
\begin{equation}\label{terms_a2}
\int_0^\infty t^{2 (\alpha+N)-3} \verti{\omega^{(1)}_{i_1}}^2 \verti{\omega^{(2)}\cdots\omega^{(5)} - \sum_{i < \alpha^\prime+N-m-1-i_1} \left(\omega^{(2)}\cdots\omega^{(5)}\right)_i x^i}_{\alpha^\prime+N-m-1-i_1}^2 \d t,
\end{equation}
\end{subequations}
respectively, where in the latter case $1 \le i_1 < \alpha^\prime+N-m-1$. Furthermore, the term \eqref{terms_a2} identically vanishes if $\omega^{(j)} \in \left\{\delta_{m_j,0}, \frac{\d^{m_j} u_0}{\d t^{m_j}}\right\}$ for all $j \ge 2$, so that we may assume $\omega^{(2)} = \partial_t^{m_2} D^{\ell_2} (u-u_0)$ in this case.

\medskip

We begin by estimating terms of the form \eqref{terms_a1} through
\begin{align*}
& \int_0^\infty t^{2 (\alpha+N)-3} \verti{\left(\omega^{(1)} - \sum_{1 \le i_1 < \alpha^\prime+N-m-1} \omega^{(1)}_{i_1} x^{i_1}\right) \omega^{(2)} \cdots\omega^{(5)}}_{\alpha^\prime+N-m-1}^2 \d t \\
& \quad \lesssim \int_0^\infty t^{2 (\alpha+N-m+m_1)-3} \verti{\omega^{(1)} - \sum_{1 \le i_1 < \alpha^\prime+N-m-1} \omega^{(1)}_{i_1} x^{i_1}}_{\alpha^\prime+N-m-1}^2 \d t \times \prod_{j = 2}^5 \sup_{t \ge 0} t^{2 m_j} \vertii{\omega^{(j)}}^2.
\end{align*}
Then we note that
\begin{align*}
& \int_0^\infty t^{2 (\alpha+N-m+m_1)-3} \verti{\omega^{(1)} - \sum_{1 \le i_1 < \alpha^\prime+N-m-1} \omega^{(1)}_{i_1} x^{i_1}}_{\alpha^\prime+N-m-1}^2 \d t \\
& \quad \lesssim t^{2 (\alpha+N-m+m_1)-3} \verti{\partial_t^{m_1} u - \sum_{i_1 < \alpha^\prime+N-m-1} \frac{\d^{m_1} u_{i_1}}{\d x^{m_1}} x^{i_1}}_{\ell+4,\alpha^\prime+N-m-1}^2 \d t \lesssim \vertiii{u}^2,
\end{align*}
provided that the indices meet condition~\eqref{cond_l_non1}. By Lemma~\ref{lem:est_coeff} and Lemma~\ref{lem:c0control}, respectively, we can further bound $\sup_{t \ge 0} t^{2 m_j} \vertii{\omega^{(j)}}^2 \lesssim \vertiii{u}^2$ for $j \ge 2$ if  $\omega^{(j)} \not\equiv 1$ and we fulfill condition~\eqref{cond_l_non2}.

\medskip

Altogether, terms of the form \eqref{terms_a1} can be bound as follows:
\begin{equation}\label{super_1}
\int_0^\infty t^{2 (\alpha+N)-3} \verti{\left(\omega^{(1)} - \sum_{1 \le i_1 < \alpha^\prime+N-m-1} \omega^{(1)}_{i_1} x^{i_1}\right) \omega^{(2)} \cdots\omega^{(5)}}_{\alpha^\prime+N-m-1}^2 \d t \lesssim \vertiii{u}^4 \times \left(1 + \vertiii{u}^6\right).
\end{equation}

\medskip

For estimating the term \eqref{terms_a2}, we use the $L^2$-bound on $\omega_{i_1}^{(1)}$, i.e.,
\begin{equation}\label{fc_1}
\begin{aligned}
& \int_0^\infty t^{2 (\alpha+N)-3} \verti{\omega^{(1)}_{i_1}}^2 \verti{\omega^{(2)}\cdots\omega^{(5)} - \sum_{i < \alpha^\prime+N-m-1-i_1} \left(\omega^{(2)}\cdots\omega^{(5)}\right)_i x^i}_{\alpha^\prime+N-m-1-i_1}^2 \d t \\
& \quad \lesssim \int_0^\infty t^{2 m_1 + 2 i_1 - 1} \verti{\omega^{(1)}_{i_1}}^2 \d t \\
& \qquad \times \sup_{t \ge 0} t^{2 (\alpha+N-m_1-i_1)-2} \verti{\omega^{(2)}\cdots\omega^{(5)} - \sum_{i < \alpha^\prime+N-m-1-i_1} \left(\omega^{(2)}\cdots\omega^{(5)}\right)_i x^i}_{\alpha^\prime+N-m-1-i_1}^2.
\end{aligned}
\end{equation}
Then we observe that by estimate~\eqref{est_coeff1} of Lemma~\ref{lem:est_coeff} we have for the second line in \eqref{fc_1}
\[
\int_0^\infty t^{2 m_1 + 2 i_1 - 1} \verti{\omega^{(1)}_{i_1}}^2 \d t = \int_0^\infty t^{2 m_1 + 2 i_1 - 1} \verti{\frac{\d^{m_1} u_{i_1}}{\d t^{m_1}}}^2 \d t \lesssim \vertiii{u}^2
\]
due to $i_1 + m_1 < N \le N_0$. Furthermore, we may use the decomposition
\[
\omega^{(2)} = \left(\omega^{(2)} - \sum_{\beta \le i_2 < \alpha^\prime+N-m-1-i_1} \omega^{(2)}_{i_2} x^{i_2}\right) + \sum_{\beta \le i_2 < \alpha^\prime+N-m-1-i_1} \omega^{(2)}_{i_2} x^{i_2}
\]
and therefore the third line in \eqref{fc_1} reduces to estimating
\begin{subequations}\label{fc_3}
\begin{equation}\label{fc_31}
\sup_{t \ge 0} t^{2 (\alpha+N-m+m_2-i_1)-2} \verti{\omega^{(2)} - \sum_{\beta \le i_2 < \alpha^\prime+N-m-1-i_1} \omega^{(2)}_{i_2} x^{i_2}}_{\alpha^\prime+N-m-1-i_1}^2 \times \prod_{j = 3}^5 \sup_{t \ge 0} t^{2 m_j} \vertii{\omega^{(j)}}^2
\end{equation}
and
\begin{equation}\label{fc_32}
\begin{aligned}
& \sup_{t \ge 0} t^{2 m_2 + 2 i_2} \verti{\omega^{(2)}_{i_2}}^2 \\
& \quad \times \sup_{t \ge 0} t^{2 (\alpha+N-m_1-m_2-i_1-i_2)-2} \verti{\omega^{(3)}\cdots\omega^{(5)} - \sum_{i < \alpha^\prime+N-m-1-i_1-i_2} \left(\omega^{(3)}\cdots\omega^{(5)}\right)_i x^i}_{\alpha^\prime+N-m-1-i_1-i_2}^2,
\end{aligned}
\end{equation}
\end{subequations}
where the term \eqref{fc_32} identically vanishes unless $\omega^{(j)} = \partial_t^{m_j} D^{\ell_j} (u-u_0)$ for at least one $j \ge 3$. We will assume without loss of generality $\omega^{(3)} = \partial_t^{m_3} D^{\ell_3} (u-u_0)$ in this case. Then we can decompose $\omega^{(3)}$ into
\[
\omega^{(3)} = \left(\omega^{(3)} - \sum_{\beta \le i_3 < \alpha^\prime+N-m-1-i_1-i_2} \omega^{(3)}_{i_3} x^{i_3}\right) + \sum_{\beta \le i_3 < \alpha^\prime+N-m-1-i_1-i_2} \omega^{(3)}_{i_3} x^{i_3}
\]
and argue as in the previous step. Hence, in order to close the argument, we need to estimate terms of the from \eqref{fc_31}. Note that not necessarily $\alpha+N-m-\frac 1 2-i_1 \in K_{N_0}$ but we do know that there exist $\left(\alpha_1,\tilde N\right), \left(\alpha_2,\tilde N\right) \in \AAA$ such that
\begin{equation}\label{indices_eq}
\alpha_1 + \tilde N - m_2 - \frac 1 2 \le \alpha+N-m-i_1 \le \alpha_2+\tilde N - m_2 - \frac 1 2
\end{equation}
and  $\alpha_j +  \tilde N$ is maximal ($j = 1$) or minimal ($j = 2$) with this property. Since $i_1 \ge 1$, in particular $\alpha_j \le \alpha + \frac 1 2$ and $\tilde N \le N-1$ or $\alpha_j \le \alpha - \frac 1 2$ and $\tilde N \le N$ for $j = 1, 2$. Further noting that
\[
(t/x)^{2 (\alpha+N-m-1-i_1)} \le (t/x)^{2 \left(\alpha_1+\tilde N-m_2\right)-3} + (t/x)^{2 \left(\alpha_2+\tilde N-m_2\right)-3},
\]
we can estimate
\begin{eqnarray*}
\lefteqn{\sup_{t \ge 0} t^{2 (\alpha+N-m+m_2-i_1)-2} \verti{\omega^{(2)} - \sum_{\beta \le i_2 < \alpha^\prime+N-m-1-i_1} \omega^{(2)}_{i_2} x^{i_2}}_{\alpha^\prime+N-m-1-i_1}^2} \\
&\lesssim& \sup_{t \ge 0} t^{2 (\alpha+N-m+m_2-i_1)-2} \verti{\partial_t^{m_2} u - \sum_{i_2 < \alpha^\prime+N-m-1-i_1} \frac{\d^{m_2} u_{i_2}}{\d t^{m_2}} x^{i_2}}_{\ell+4,\alpha^\prime+N-m-1-i_1}^2 \\
&\stackrel{\eqref{def_wn},\eqref{def_lnm}}{\lesssim}& \sup_{t \ge 0} t^{2 (\alpha+N-m+m_2-i_1)-2} \verti{\partial_t^{m_2} w^{(N-m)}}_{k\left(N-m,m,\alpha^\prime\right)+4,\alpha^\prime+N-m-1-i_1}^2 \\
&\stackrel{\eqref{cond_l_non3}, \eqref{cond_l_non4}}{\lesssim}& \sup_{t \ge 0} t^{2 \left(\alpha_1+\tilde N\right)-3} \verti{\partial_t^{m_2} w^{\left(\tilde N-m_2\right)}}_{k\left(\tilde N-m_2,m_2,\alpha_1^\prime\right)+2,\alpha_1^\prime+\tilde N-m_2-\frac 32}^2 \\
&& + \sup_{t \ge 0} t^{2 \left(\alpha_2+\tilde N\right)-3} \verti{\partial_t^{m_2} w^{\left(\tilde N-m_2\right)}}_{k\left(\tilde N-m_2,m_2,\alpha_2^\prime\right)+2,\alpha_2^\prime+\tilde N-m_2-\frac 32}^2\\
&\stackrel{\eqref{def_wn},\eqref{def_lnm}}{\lesssim}& \sup_{t \ge 0} t^{2 \left(\alpha_1+\tilde N\right)-3} \verti{\partial_t^{m_2} u - \sum_{i_2 < \alpha_1^\prime+\tilde N-m_2-\frac 3 2} \frac{\d^{m_2} u_{i_2}}{\d t^{m_2}} x^{i_2}}_{\ell\left(\tilde N-m_2,m_2,\alpha_1^\prime\right)+2,\alpha_1^\prime+\tilde N-m_2-\frac 32}^2 \\
&& + \sup_{t \ge 0} t^{2 \left(\alpha_2+\tilde N\right)-3} \verti{\partial_t^{m_2} u - \sum_{i_2 < \alpha_2^\prime+\tilde N-m_2-\frac 3 2} \frac{\d^{m_2} u_{i_2}}{\d t^{m_2}} x^{i_2}}_{\ell\left(\tilde N-m_2,m_2,\alpha_2^\prime\right)+2,\alpha_2^\prime+\tilde N-m_2-\frac 32}^2 \\
&\stackrel{\eqref{norm_sol}}{\lesssim}& \vertiii{u}^2,
\end{eqnarray*}
where Proposition~\ref{prop:elliptic} (elliptic maximal regularity) as well as conditions~\eqref{cond_l_non1}, \eqref{cond_l_non3}, and \eqref{cond_l_non4} have been used.

\medskip

Finally, the product $\prod_{j = 3}^5 \sup_{t \ge 0} t^{2 m_j} \vertii{\omega^{(j)}}^2$ in \eqref{fc_3} can be estimated as before by using estimate~\eqref{est_coeff2} of Lemma~\ref{lem:est_coeff} and Lemma~\ref{lem:c0control}, respectively, provided condition~\eqref{cond_l_non2} is satisfied. This shows that we can bound terms of the form \eqref{terms_a2} as
\begin{equation}\label{super_2}
\begin{aligned}
&\int_0^\infty t^{2 (\alpha+N)-3} \verti{\omega^{(1)}_{i_1}}^2 \verti{\omega^{(2)}\cdots\omega^{(5)} - \sum_{i < \alpha^\prime+N-m-1-i_1} \left(\omega^{(2)}\cdots\omega^{(5)}\right)_i x^i}_{\alpha^\prime+N-m-1-i_1}^2 \d t \\
& \quad \lesssim \vertiii{u}^4 \times \left(1 + \vertiii{u}^6\right).
\end{aligned}
\end{equation}
The combination of \eqref{super_1} and \eqref{super_2} gives
\begin{equation}\label{super_a}
\int_0^\infty t^{2 (\alpha+N)-3} \verti{\omega^{(1)}\cdots\omega^{(5)} - \sum_{i < \alpha^\prime+N-m-1} \left(\omega^{(1)}\cdots\omega^{(5)}\right)_i x^i}_{\alpha^\prime+N-m-1}^2 \d t \lesssim \vertiii{u}^4 \times \left(1 + \vertiii{u}^6\right)
\end{equation}
for super-critical terms of type \eqref{omegas_a}.

\subsubsection*{Estimating \eqref{omegas} for terms of the form {\rm (b)}}
Distributing the time derivatives on the factors $\omega^{(j)}$ in \eqref{omegas} meeting (b), it suffices to estimate terms of the form
\begin{equation}\label{omegas_b}
\int_0^\infty t^{2(\alpha+N)-3} \verti{\omega^{(1)}}^2 \verti{\omega^{(2)}\cdots\omega^{(5)} - \sum_{i < \alpha^\prime+N-m-1-\beta} \left(\omega^{(2)}\cdots\omega^{(5)}\right)_i x^i}_{\alpha^\prime+N-m-1-\beta}^2,
\end{equation}
where the $\omega^{(j)}$ obey
\begin{align*}
\omega^{(1)} &= \frac{\d^{m_1} u_\beta}{\d t^{m_1}} \quad \mbox{with } \, 0 \le m_1 \le m,\\
\omega^{(2)} &= \partial_t^{m_2} D^{\ell_2} (u-u_0) \quad \mbox{with } \, 0 \le \ell_2 \le \ell+4 \, \mbox{ and } \, 0 \le m_2 \le m,\\
\omega^{(j)} &\in \left\{\delta_{m_j,0}, \frac{\d^{m_j} u_0}{\d t^{m_j}}, \partial_t^{m_j} D^{\ell_j} (u-u_0)\right\} \quad \mbox{with } \, 0 \le \ell_j \le \frac{\ell+4}{2} \, \mbox{ and } \, 0 \le m_j \le m \, \mbox{ for } \, j \ge 3,
\end{align*}
where $\sum_{j = 1}^5 m_j = m$. If $m_1 = m = N-1 = N_0-1$, we may use the $L^2$-bound on $\omega^{(1)}$ and obtain for the term \eqref{omegas_b}:
\begin{equation}\label{case_b1}
\begin{aligned}
&\int_0^\infty t^{2(\alpha+N)-3} \verti{\omega^{(1)}}^2 \verti{\omega^{(2)}\cdots\omega^{(5)} - \sum_{i < \alpha^\prime+N-m-1-\beta} \left(\omega^{(2)}\cdots\omega^{(5)}\right)_i x^i}_{\alpha^\prime+N-m-1-\beta}^2 \d t \\
& \quad \lesssim \int_0^\infty t^{2 (\beta+N)-3} \verti{\omega^{(1)}}^2 \d t \times \sup_{t \ge 0} t^{2(\alpha-\beta)} \verti{\omega^{(2)}}_{\alpha^\prime-\beta}^2 \times \prod_{j = 3}^5 \sup_{t \ge 0} \vertii{\omega^{(j)}}^2.
\end{aligned}
\end{equation}
Then we note that we can estimate the first term in the second line of \eqref{case_b1} through
\[
\int_0^\infty t^{2 (\beta+N)-3} \verti{\omega^{(1)}}^2 \d t = \int_0^\infty t^{2 (\beta+N)-3} \verti{\frac{\d^{N-1} u_\beta}{\d t^{N-1}}}^2 \d t \lesssim \vertiii{u}^2
\]
by estimate~\eqref{est_coeff1} of Lemma~\ref{lem:est_coeff}. Furthermore, since we are in the super-critical case, we have $(t/x)^{2(\alpha-\beta)} \le 1 + t/x$ and we can bound as follows\footnote{In the special case $N = 1$ (i.e.,~$N_0 = 1$ by assumption) we have $\alpha = \beta$ and the second summand on the right-hand side can be dropped.}
\begin{align*}
\sup_{t \ge 0} t^{2(\alpha-\beta)} \verti{\omega^{(2)}}_{\alpha^\prime-\beta}^2 &\lesssim \sup_{t \ge 0} t^{2(\alpha-\beta)} \verti{D (u - u_0)}_{\ell + 3,\alpha^\prime-\beta}^2 \\
&\lesssim \sup_{t \ge 0} \verti{u}_{\ell+4,-\delta}^2 + \sup_{t \ge 0} \verti{u-u_0}_{\ell+4,\delta}^2 + \left(1 - \delta_{1,N}\right) \sup_{t \ge 0} t \verti{u-u_0}_{\ell+4,\frac 1 2 \pm \delta}^2 \lesssim \vertiii{u}^2,
\end{align*}
where Proposition~\ref{prop:elliptic} has been used and conditions~\eqref{cond_l_non5} and \eqref{cond_l_non6} need to be satisfied.

\medskip

The product $\prod_{j = 3}^5 \sup_{t \ge 0} \vertii{\omega^{(j)}}^2$ in \eqref{case_b1} can be estimated by employing estimate~\eqref{est_coeff2} of Lemma~\ref{lem:est_coeff} and Lemma~\ref{lem:c0control}, respectively, if condition~\eqref{cond_l_non2} holds true. In summary we obtain
\begin{equation}\label{super_3}
\begin{aligned}
&\int_0^\infty t^{2(\alpha+N)-3} \verti{\omega^{(1)}}^2 \verti{\omega^{(2)}\cdots\omega^{(5)} - \sum_{i < \alpha^\prime+N-m-1-\beta} \left(\omega^{(2)}\cdots\omega^{(5)}\right)_i x^i}_{\alpha^\prime+N-m-1-\beta}^2 \\
&\quad\lesssim \vertiii{u}^4 \times \left(1 + \vertiii{u}^6\right)
\end{aligned}
\end{equation}
in the case $m_1 = m = N-1 = N_0-1$.

\medskip

In all other cases, we can estimate \eqref{omegas_b} by taking the $C^0$-bound on $\omega^{(1)}$ and get:
\begin{equation}\label{case_b2}
\begin{aligned}
&\int_0^\infty t^{2(\alpha+N)-3} \verti{\omega^{(1)}}^2 \verti{\omega^{(2)}\cdots\omega^{(5)} - \sum_{i < \alpha^\prime+N-m-1-\beta} \left(\omega^{(2)}\cdots\omega^{(5)}\right)_i x^i}_{\alpha^\prime+N-m-1-\beta}^2 \\
& \quad \lesssim \sup_{t \ge 0} t^{2 m_1 + 2 \beta} \verti{\omega^{(1)}}^2 \\
& \qquad \times \int_0^\infty t^{2 (\alpha+N-m_1-\beta)-3} \verti{\omega^{(2)}\cdots\omega^{(5)} - \sum_{i < \alpha^\prime+N-m-1-\beta} \left(\omega^{(2)}\cdots\omega^{(5)}\right)_i x^i}_{\alpha^\prime+N-m-1-\beta}^2 \d t.
\end{aligned}
\end{equation}
Then we notice
\[
\sup_{t \ge 0} t^{2 m_1 + 2 \beta} \verti{\omega^{(1)}}^2 = \sup_{t \ge 0} t^{2 m_1 + 2 \beta} \verti{\frac{\d^{m_1} u_\beta}{\d t^{m_1}}}^2 \lesssim \vertiii{u}^2
\]
by estimate~\eqref{est_coeff2} of Lemma~\ref{lem:est_coeff}. For the last line of \eqref{case_b2} we may use the decomposition

\[
\omega^{(2)} = \left(\omega^{(2)} - \sum_{\beta \le i_2 < \alpha^\prime+N-m-1-\beta} \omega^{(2)}_{i_2} x^{i_2}\right) + \sum_{\beta \le i_2 < \alpha^\prime+N-m-1-\beta} \omega^{(2)}_{i_2} x^{i_2},
\]
so that it suffices to estimate
\begin{subequations}\label{lc}
\begin{equation}\label{lc_1}
\int_0^\infty t^{2 (\alpha+N-m+m_2-\beta)-3} \verti{\omega^{(2)} - \sum_{\beta \le i_2 < \alpha^\prime+N-m-1-\beta} \omega^{(2)}_{i_2} x^{i_2}}_{\alpha^\prime+N-m-1-\beta}^2 \d t \times \prod_{j = 3}^5 \sup_{t \ge 0} t^{2 m_j} \vertii{\omega^{(j)}}^2
\end{equation}
and
\begin{equation}\label{lc_2}
\begin{aligned}
& \int_0^\infty t^{2 m_2 + 2 i_2 - 1} \verti{\omega^{(2)}_{i_2}}^2 \d t \\
& \quad \times \sup_{t \ge 0} t^{2 (\alpha+N-m_1-m_2-\beta-i_2)-2} \verti{\omega^{(3)}\cdots\omega^{(5)} - \sum_{i < \alpha^\prime+N-m-1-\beta-i_2} \left(\omega^{(3)}\cdots\omega^{(5)}\right)_i x^i}_{\alpha^\prime+N-m-1-\beta-i_2}^2,
\end{aligned}
\end{equation}
\end{subequations}
where the term \eqref{lc_2} is identically zero unless $\omega^{(j)} = \partial_t^{m_j} D^{\ell_j} (u-u_0)$ for at least one $j \ge 3$. Since $\beta + i_2 \ge 1$, the expression \eqref{lc_2} can be treated in the same way as the second and third line of \eqref{fc_1}.

\medskip

For the term \eqref{lc_1} we can treat the product $\prod_{j = 3}^5 \sup_{t \ge 0} t^{2 m_j} \vertii{\omega^{(j)}}^2$ as before (by using Lemma~\ref{lem:c0control} and estimate~\eqref{est_coeff2} of Lemma~\ref{lem:est_coeff}, respectively) given that we fulfill condition~\eqref{cond_l_non2}. For treating the $L^2$-part in \eqref{lc_1} we pick $\left(\alpha_1,\tilde N\right), \left(\alpha_2,\tilde N\right) \in \AAA$ such that
\begin{equation}\label{indices_eq2}
\alpha_1 + \tilde N - m_2 \le \alpha+N-m-\beta \le \alpha_2+\tilde N - m_2
\end{equation}
and  $\alpha_j +  \tilde N$ is maximal ($j = 1$) or minimal ($j = 2$) with this property. Since $\beta > \frac 1 2$ we have in particular $\alpha_j \le \alpha + \frac 1 2$ and $\tilde N - m_2 \le N-m-1$ or $\alpha_j \le \alpha - \frac 1 2$ and $\tilde N - m_2 \le N-m$ for $j = 1, 2$ under the assumption of a sufficiently small $\delta > 0$. Utilizing
\[
(t/x)^{2 (\alpha+N-m-\beta-1)} \le (t/x)^{2 \left(\alpha_1+\tilde N-m_2-1\right)} + (t/x)^{2 \left(\alpha_2+\tilde N-m_2-1\right)},
\]
we obtain the estimate
\begin{eqnarray*}
\lefteqn{\int_0^\infty t^{2 (\alpha+N-m+m_2-\beta)-3} \verti{\omega^{(2)} - \sum_{\beta \le i_2 < \alpha^\prime+N-m-1-\beta} \omega^{(2)}_{i_2} x^{i_2}}_{\alpha^\prime+N-m-1-\beta}^2 \d t}\\
&\lesssim& \int_0^\infty t^{2 (\alpha+N-m+m_2-\beta)-3} \verti{\partial_t^{m_2} u - \sum_{i_2 < \alpha^\prime+N-m-1-\beta} \frac{\d^{m_2} u_{i_2}}{\d t^{m_2}} x^{i_2}}_{\ell+4,\alpha^\prime+N-m-1-\beta}^2 \d t\\
&\stackrel{\eqref{def_wn},\eqref{def_lnm}}{\lesssim}& \int_0^\infty t^{2 (\alpha+N-m+m_2-\beta)-3} \verti{\partial_t^{m_2} w^{(N-m)}}_{k\left(N-m,m,\alpha^\prime\right)+4,\alpha^\prime+N-m-1-\beta}^2 \d t\\
&\stackrel{\eqref{cond_last}}{\lesssim}& \int_0^\infty t^{2 (\alpha_1+\tilde N)-3} \verti{\partial_t^{m_2} w^{\left(\tilde N-m_2\right)}}_{k\left(\tilde N-m_2,m_2,\alpha_1^\prime\right)+4,\alpha_1^\prime+\tilde N-m_2-1}^2 \d t \\
&& + \int_0^\infty t^{2 (\alpha_2+\tilde N)-3} \verti{\partial_t^{m_2} w^{\left(\tilde N-m_2\right)}}_{k\left(\tilde N-m_2,m_2,\alpha_2^\prime\right)+4,\alpha_2^\prime+\tilde N-m_2-1}^2 \d t\\
&\stackrel{\eqref{def_wn},\eqref{def_lnm}}{\lesssim}& \int_0^\infty t^{2 (\alpha_1+\tilde N)-3} \verti{\partial_t^{m_2} u - \sum_{i_2 < \alpha_1^\prime+\tilde N-m_2-1} \frac{\d^{m_2} u_{i_2}}{\d t^{m_2}} x^{i_2}}_{\ell\left(\tilde N-m_2,m_2,\alpha_1^\prime\right)+4,\alpha_1^\prime+\tilde N-m_2-1}^2 \d t \\
&& + \int_0^\infty t^{2 (\alpha_2+\tilde N)-3} \verti{\partial_t^{m_2} u - \sum_{i_2 < \alpha_2^\prime+\tilde N-m_2-1} \frac{\d^{m_2} u_{i_2}}{\d t^{m_2}} x^{i_2}}_{\ell\left(\tilde N-m_2,m_2,\alpha_2^\prime\right)+4,\alpha_2^\prime+\tilde N-m_2-1}^2 \d t \\
&\stackrel{\eqref{norm_sol}}{\lesssim}& \vertiii{u}^2,
\end{eqnarray*}
where we have used elliptic maximal regularity in the form of Proposition~\ref{prop:elliptic}, \eqref{cond_l_non1}, and the conditions
\begin{subequations}\label{cond_last}
\begin{align}
\ell\left(N-m,m,\alpha-\frac 1 2\right) &\ge \ell\left(N-m,m,\alpha^\prime\right) \quad \mbox{for all } \, \alpha \in \left(\frac 1 2, 1\right),\\
\ell\left(N-m-1,m,\alpha+\frac 1 2\right) &\ge \ell\left(N-m,m,\alpha^\prime\right) \quad \mbox{for all } \, \alpha \in \left(0, \frac 1 2\right).
\end{align}
\end{subequations}
These conditions in particular hold if the stronger conditions~\eqref{cond_l_non3}~and~\eqref{cond_l_non4} are fulfilled.

\medskip
In summary we obtain \eqref{super_3} for the case of arbitrary $m_1$ and $N$ as well.

\subsubsection*{Conclusion}
Gathering inequalities~\eqref{main_subcrit},~\eqref{super_a},~and~\eqref{super_3}, we obtain inequality~\eqref{non_main}, thus finishing the proof of Proposition~\ref{prop:nonlinear_est}.
\end{proof}

%
\bibliography{pre_higher_regularity_rev} 
\bibliographystyle{plain} 

\end{document}